\documentclass[11pt]{article}
\usepackage[total={7in, 8in}]{geometry}
\usepackage{graphicx}
\usepackage{amsmath, amsthm, latexsym, amssymb, color,cite,enumerate}
\usepackage{caption,subcaption,}
\theoremstyle{theorem}
\newtheorem{theorem}{Theorem}[section]
\newtheorem{hypothesis}[theorem]{Hypothesis}
\newtheorem{lemma}[theorem]{Lemma}
\newtheorem{definition}[theorem]{Definition}
\newtheorem{corollary}[theorem]{Corollary}
\newtheorem{proposition}[theorem]{Proposition}

\theoremstyle{remark}
\newtheorem{remark}{Remark}
\newtheorem{example}{Example}
\newcommand*{\myproofname}{Proof}

\renewcommand\Re{\operatorname{Re}}
\renewcommand\Im{\operatorname{Im}}

\newcommand\Sym{\operatorname{Sym}}
\newcommand\Exp{\operatorname{Exp}}
\newcommand\tr{\operatorname{tr}}
\newcommand\diag{\operatorname{diag}}

\renewcommand\det{\operatorname{det}}
\newcommand{\Ker}{\operatorname{Ker}}
\newcommand{\End}{\mbox{End}}

\newcommand{\Gl}{\mbox{Gl}}
\author{Evan Randles and Laurent Saloff-Coste}

\title{Positive-homogeneous operators, heat kernel estimates and the Legendre-Fenchel transform}
\date{}
\begin{document}
\maketitle
\begin{center}
\textit{Dedicated to Professor Rodrigo Ba\~{n}uelos on the occasion of his 60th birthday.}
\end{center}
\begin{abstract}
We consider a class of homogeneous partial differential operators on a finite-dimensional vector space and study their associated heat kernels. The heat kernels for this general class of operators are seen to arise naturally as the limiting objects of the convolution powers of complex-valued functions on the square lattice in the way that the classical heat kernel arises in the (local) central limit theorem. These so-called positive-homogeneous operators generalize the class of semi-elliptic operators in the sense that the definition is coordinate-free. More generally, we introduce a class of variable-coefficient operators, each of which is uniformly comparable to a positive-homogeneous operator, and we study the corresponding Cauchy problem for the heat equation. Under the assumption that such an operator has H\"{o}lder continuous coefficients, we construct a fundamental solution to its heat equation by the method of E. E. Levi, adapted to parabolic systems by A. Friedman and S. D. Eidelman. Though our results in this direction are implied by the long-known results of S. D. Eidelman for $2\vec{b}$-parabolic systems, our focus is to highlight the role played by the Legendre-Fenchel transform in heat kernel estimates. Specifically, we show that the fundamental solution satisfies an off-diagonal estimate, i.e., a heat kernel estimate, written in terms of the Legendre-Fenchel transform of the operator's principal symbol--an estimate which is seen to be sharp in many cases.
\end{abstract}
\noindent{\small\bf Keywords:} Semi-elliptic operators, quasi-elliptic operators, $2\vec{b}$-parabolic operators, heat kernel estimates, Legendre-Fenchel transform.\\

\noindent{\small\bf Mathematics Subject Classification:} Primary 35H30; Secondary  35K25.
\section{Introduction}

In this article, we consider a class of homogeneous partial differential operators on a finite dimensional vector space and study their associated heat kernels. These operators, which we call nondegenerate-homogeneous operators, are seen to generalize the well-studied classes of semi-elliptic operators introduced by F. Browder \cite{Browder1957}, also known as quasi-elliptic operators \cite{Volevich1962}, and a special ``positive" subclass of  semi-elliptic operators which appear as the spatial part of S. D. Eidelman's $2\vec{b}$-parabolic operators \cite{Eidelman1960}.  In particular, this class of operators contains all integer powers of the Laplacian.

\subsection{Semi-Elliptic Operators}

\noindent To motivate the definition of nondegenerate-homogeneous operators, given in the next section, we first introduce the class of semi-elliptic operators. Semi-elliptic operators are seen to be prototypical examples of nondegenerate-homogeneous operators; in fact, the definition of nondegenerate-homogeneous operators is given to formulate the following construction in a basis-independent way. Given $d$-tuple of positive integers $\mathbf{n}=(n_1,n_2,\dots,n_d)\in\mathbb{N}_+^d$ and a multi-index $\beta=(\beta_1,\beta_2,\dots,\beta_d)\in\mathbb{N}^d$, set $|\beta:\mathbf{n}|=\sum_{k=1}^d\beta_k/n_k$. Consider the constant coefficient partial differential operator
\begin{equation*}
\Lambda=\sum_{|\beta:\mathbf{n}|\leq1}a_{\beta}D^{\beta}
\end{equation*}
with principal part (relative to $\mathbf{n}$)
\begin{equation*}
\Lambda_p=\sum_{|\beta:\mathbf{n}|=1}a_{\beta}D^{\beta},
\end{equation*}
where $a_{\beta}\in\mathbb{C}$ and $D^{\beta}=(i\partial_{x_1})^{\beta_1}(i\partial_{x_2})^{\beta_2}\cdots(i\partial_{x_d})^{\beta_d}$ for each multi-index $\beta\in\mathbb{N}^d$.  Such an operator $\Lambda$ is said to be \textit{semi-elliptic} if the symbol of $\Lambda_p$, defined by $P_p(\xi)=\sum_{|\beta:\mathbf{n}|=1}a_{\beta}\xi^{\beta}$ for $\xi\in\mathbb{R}^d$, is non-vanishing away from the origin. If $\Lambda$ satisfies the stronger condition that $\Re P_p(\xi)$ is strictly positive away from the origin, we say that it is \textit{positive-semi-elliptic}. What seems to be the most important property of semi-elliptic operators is that their principal part $\Lambda_p$ is homogeneous in the following sense: If given any smooth function $f$ we put $\delta_t(f)(x)=f(t^{1/n_1}x_1,t^{1/n_2}x_2,\dots,t^{1/n_d}x_d)$ for  all $t>0$ and $x=(x_1,x_2,\dots,x_d)\in\mathbb{R}^d$, then
\begin{equation*}
t\Lambda=\delta_{1/t}\circ \Lambda_p\circ \delta_t
\end{equation*}
for all $t>0$. This homogeneous structure was used explicitly in the work of F. Browder and L. H\"{o}rmander and, in this article, we generalize this notion. We note that our definition for the differential operators $D^{\beta}$ is given to ensure a straightforward relationship between operators and symbols under our convention for the Fourier transform (defined in Subsection \ref{subsec:Preliminaries}); this definition differs only slightly from the standard references \cite{Hormander1963,Hormander1983,Rudin1991, Taylor1981} in which $i$ is replaced by $1/i$. In both conventions, the symbol of the operator $\Lambda=-\Delta=-\sum_{k=1}^d\partial^2_{x_k}$ is the positive polynomial $\xi\mapsto|\xi|^2=\sum_{k=1}^d\xi_k^2$. In fact, the principal symbols of all positive-semi-elliptic operators agree in both conventions.\\

\noindent As mentioned above, the class of semi-elliptic operators was introduced by F. Browder in \cite{Browder1957} who studied spectral asymptotics for a related class of variable-coefficient operators (operators of constant strength). Semi-elliptic operators appeared later in L. H\"{o}rmander's text \cite{Hormander1963} as model examples of hypoelliptic operators on $\mathbb{R}^d$ beyond the class of elliptic operators. Around the same time, L. R. Volevich \cite{Volevich1962} independently introduced the same class of operators but instead called them ``quasi-elliptic''. Since then, the theory of semi-elliptic operators, and hence quasi-elliptic operators, has reached a high level of sophistication and we refer the reader to the articles \cite{Apel2014, Artino1973, Artino1993, Artino1977, Artino1995, Browder1957, Hile2006, Hile2001, Hormander1963, Hormander1983, Kannai1969, Triebel1983, Tsutsumi1975}, which use the term semi-elliptic, and the articles \cite{Bondar2012, Bondar2012a, Bondar2008, Cavallucci1965, Demidenko1989, Demidenko1993, Demidenko1994, Demidenko1998, Demidenko2001, Demidenko2007, Demidenko2008, Demidenko2009, Giusti1967, Matsuzawa1968, Pini1962, Troisi1971, Volevich1960, Volevich1962}, which use the term quasi-elliptic, for an account of this theory. We would also like to point to the 1971 paper of M. Troisi \cite{Troisi1971} which gives a more complete list of references (pertaining to quasi-elliptic operators).\\

\noindent Shortly after F. Browder's paper \cite{Browder1957} appeared, S. D. Eidelman considered a subclass of semi-elliptic operators on $\mathbb{R}^{d+1}=\mathbb{R}\oplus\mathbb{R}^d$ (and systems thereof) of the form
\begin{equation}\label{eq:2b-Parabolic}
\partial_t+\sum_{|\beta:2\mathbf{m}|\leq 1}a_{\beta}D^{\beta}=\partial_t+\sum_{|\beta:\mathbf{m}|\leq 2}a_{\beta}D^{\beta},
\end{equation}
where $\mathbf{m}\in\mathbb{N}_+^d$ and the coefficients $a_{\beta}$ are functions of $x$ and $t$. Such an operator is said to be \textit{$2\mathbf{m}$-parabolic} if its spatial part, $\sum_{|\beta:2\mathbf{m}|\leq 1}a_{\beta}D^{\beta}$, is (uniformly) positive-semi-elliptic. We note however that Eidelman's work and the existing literature refer exclusively to $2\vec{b}$-parabolic operators, i.e., where $\mathbf{m}=\vec{b}$, and for consistency we write $2\vec{b}$-parabolic henceforth  \cite{Eidelman1960,Eidelman2004}. The relationship between positive-semi-elliptic operators and $2\vec{b}$-parabolic operators is analogous to the relationship between the Laplacian and the heat operator and, in the context of this article, the relationship between nondegenerate-homogeneous and positive-homogeneous operators described by Proposition \ref{prop:Dichotomy}.  The theory of $2\vec{b}$-parabolic operators, which generalizes the theory of parabolic partial differential equations (and systems), has seen significant advancement by a number of mathematicians since Eidelman's original work. We encourage the reader to see the recent text \cite{Eidelman2004} which provides an account of this theory and an exhaustive list of references. It should be noted however that the literature encompassing semi-elliptic operators and quasi-elliptic operators, as far as we can tell, has very few cross-references to the literature on $2\vec{b}$-parabolic operators beyond the 1960s. We suspect that the absence of cross-references is due to the distinctness of vocabulary. \\

\subsection{Motivation: Convolution powers of complex-valued functions on $\mathbb{Z}^d$}

We motivate the study of homogeneous operators by first demonstrating the natural appearance of their heat kernels in the study of convolution powers of complex-valued functions. To this end, consider a finitely supported function $\phi:\mathbb{Z}^d\rightarrow\mathbb{C}$ and define its convolution powers iteratively by 
\begin{equation*}
\phi^{(n)}(x)=\sum_{y\in\mathbb{Z}^d}\phi^{(n-1)}(x-y)\phi(y)
\end{equation*}
for $x\in\mathbb{Z}^d$ where $\phi^{(1)}=\phi$. In the special case that $\phi$ is a probability distribution, i.e., $\phi$ is non-negative and has unit mass, $\phi$ drives a random walk on $\mathbb{Z}^d$ whose nth-step transition kernels are given by $k_n(x,y)=\phi^{(n)}(y-x)$. Under certain mild conditions on the random walk, $\phi^{(n)}$ is well-approximated by a single Gaussian density; this is the classical local limit theorem. Specifically, for a symmetric, aperiodic and irreducible random walk, the theorem states that
\begin{equation}\label{eq:LLT}
\phi^{(n)}(x)=n^{-d/2}G_\phi(x/\sqrt{n})+o(n^{-d/2})
\end{equation}
uniformly for $x\in\mathbb{Z}^d$, where $G_\phi$ is the generalized Gaussian density
\begin{equation}\label{eq:Gaussian}
G_\phi(x)=\frac{1}{(2\pi)^d}\int_{\mathbb{R}^d}\exp\big(-\xi\cdot C_\phi\xi\big)e^{-ix\cdot\xi}\,d\xi=\frac{1}{(2\pi )^{d/2}\sqrt{\det C_\phi}}\exp\left(-\frac{x\cdot {C_\phi}^{-1}x}{2}\right);
\end{equation}
here, $C_\phi$ is the positive definite covariance matrix associated to $\phi$ and $\cdot$ denotes the dot product \cite{Spitzer1964,Lawler2010, Randles2015a}. The canonical example is that in which $C_{\phi}=I$ (e.g. Simple Random Walk) and in this case $\phi^{(n)}$ is approximated by the so-called heat kernel $K_{(-\Delta)}:(0,\infty)\times\mathbb{R}^d\rightarrow (0,\infty)$ defined by
\begin{equation*}
K_{(-\Delta)}^t(x)=(2\pi t)^{-d/2}\exp\left(-\frac{|x|^2}{2t}\right)
\end{equation*}
for $t>0$ and $x\in\mathbb{R}^d$. Indeed, we observe that $n^{-d/2}G_{\phi}(x/\sqrt{n})=K_{(-\Lambda)}^n(x)$ for each positive integer $n$ and $x\in\mathbb{Z}^d$ and so the local limit theorem \eqref{eq:LLT} is written equivalently as
\begin{equation*}
\phi^{(n)}(x)=K_{(-\Delta)}^n(x)+o(n^{-d/2})
\end{equation*}
uniformly for $x\in\mathbb{Z}^d$. In addition to its natural appearance as the \textit{attractor} in the local limit theorem above, $K^t_{(-\Delta)}(x)$ is a fundamental solution to the heat equation
\begin{equation*}
\partial_t+(-\Delta)=0
\end{equation*}
on $(0,\infty)\times\mathbb{R}^d$. In fact, this connection to random walk underlies the heat equation's probabilistic/diffusive interpretation. Beyond the probabilistic setting, this link between convolution powers and fundamental solutions to partial differential equations persists as can be seen in the examples below. In what follows, the heat kernels $(t,x)\mapsto K_{\Lambda}^t(x)$ are fundamental solutions to the corresponding heat-type equations of the form
\begin{equation*}
\partial_t+\Lambda=0.
\end{equation*}
The appearance of $K_{\Lambda}$ in local limit theorems (for $\phi^{(n)}$) is then found by evaluating $K_{\Lambda}^t(x)$ at integer time $t=n$ and lattice point $x\in\mathbb{Z}^d$.

\begin{example}\label{ex:intro1} Consider $\phi:\mathbb{Z}^2\rightarrow\mathbb{C}$ defined by
\begin{equation*}
\phi(x_1,x_2)=\frac{1}{22+2\sqrt{3}}\times\begin{cases}
	    8 & (x_1,x_2)=(0,0)\\
           5+\sqrt{3} & (x_1,x_2)=(\pm 1,0)\\
           -2 & (x_1,x_2)=(\pm 2,0)\\
            i(\sqrt{3}-1)& (x_1,x_2)=(\pm 1,-1)\\
            -i(\sqrt{3}-1)& (x_1,x_2)=(\pm 1,1)\\
            2\mp 2i & (x_1,x_2)=(0,\pm 1)\\
           0 & \mbox{otherwise}.
          \end{cases}
\end{equation*}
\begin{figure}[h!]
\begin{center}
\resizebox{\textwidth}{!}{
\begin{subfigure}[6cm]{0.5\textwidth}
		\includegraphics[width=\textwidth]{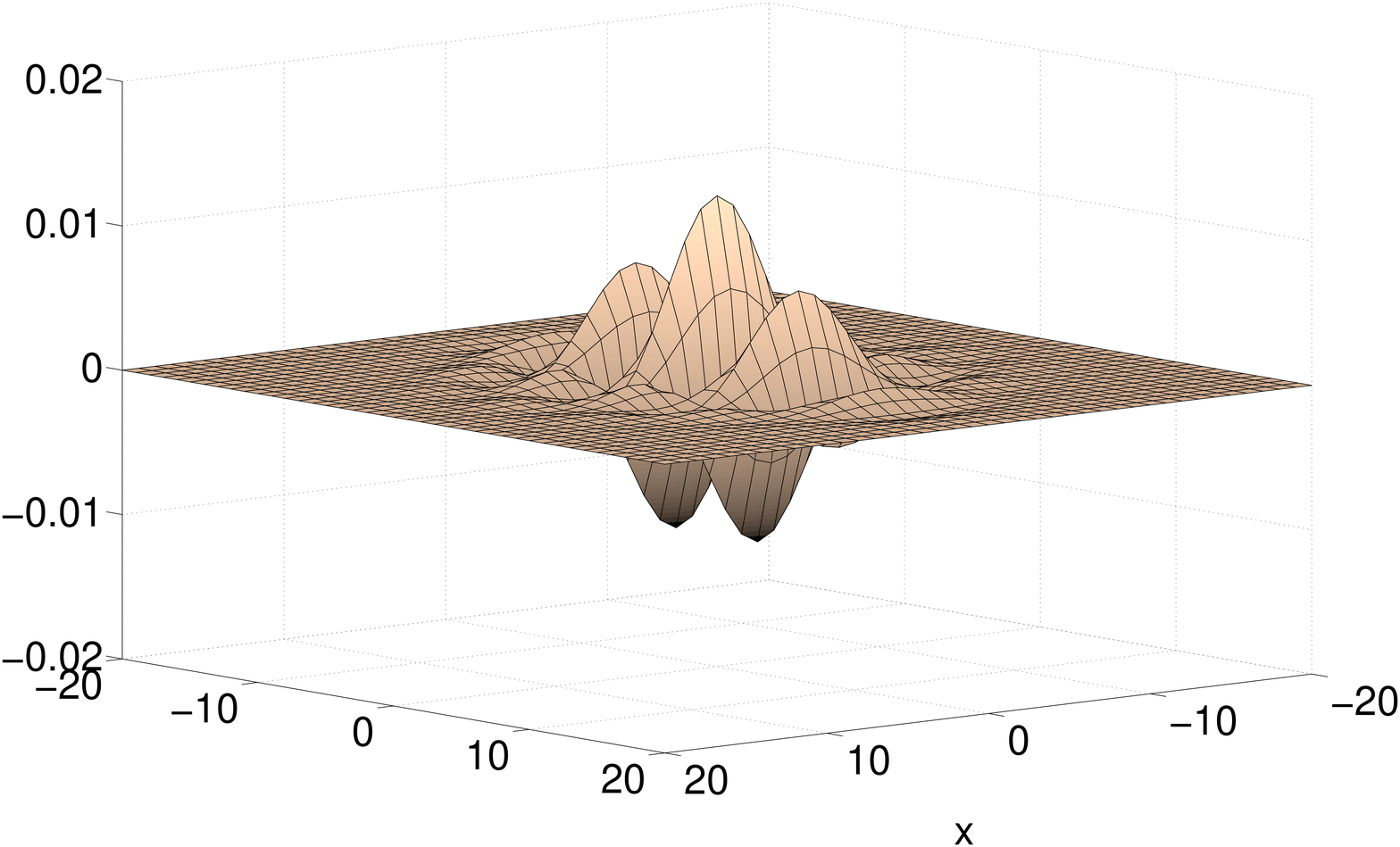}
		\caption{$\Re(\phi^{(n)})$ for $n=100$}
		\label{fig:ex_intro_100_1}
	    \end{subfigure}
	    \begin{subfigure}[5cm]{0.5\textwidth}
		\includegraphics[width=\textwidth]{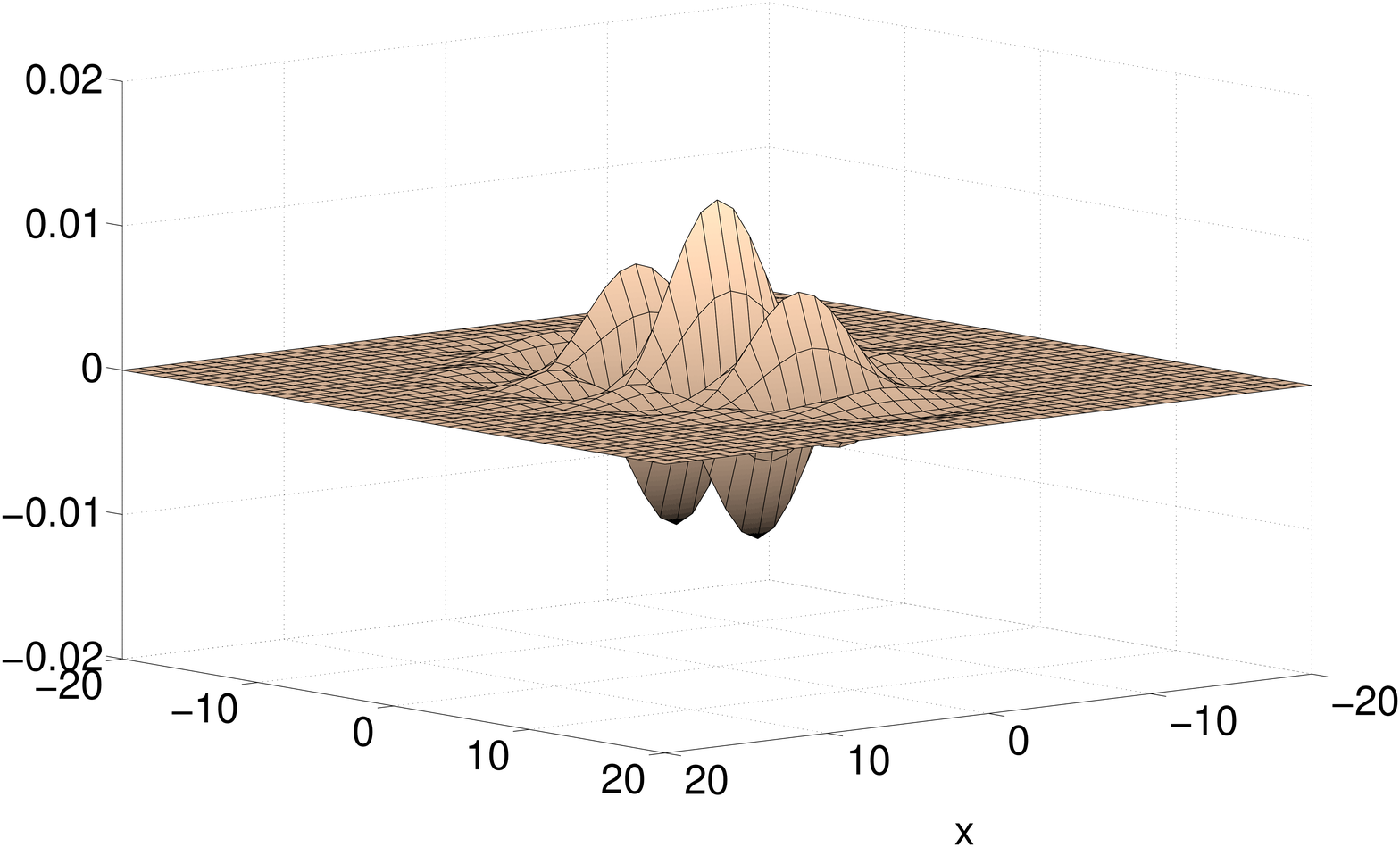}
		\caption{$\Re(e^{-i\pi x_2/3}K_{\Lambda}^n)$ for $n=100$}
		\label{fig:ex_intro_100_attractor_1}
	    \end{subfigure}}
\caption{The graphs of $\Re(\phi^{(n)})$ and $\Re(e^{-i\pi x_2/3}K_{\Lambda}^n)$ for $n=100$.}
\label{fig:ex1_intro}
\end{center}
\end{figure}

\noindent Analogous to the probabilistic setting, the large $n$ behavior of $\phi^{(n)}$ is described by a generalized local limit theorem in which the attractor is a fundamental solution to a heat-type equation. Specifically, the following local limit theorem holds (see \cite{Randles2015a} for details): 
\begin{equation*}
 \phi^{(n)}(x_1,x_2)=e^{-i\pi x_2/3}K_{\Lambda}^n(x_1,x_2)+o(n^{-3/4})
\end{equation*}
uniformly for $(x_1,x_2)\in\mathbb{Z}^2$ where $(t,x)\mapsto K_{\Lambda}^t(x)$ is the ``heat'' kernel for the heat-type equation $\partial_t+\Lambda=0$ where
\begin{equation*}
\Lambda=\frac{1}{22+2\sqrt{3}}\left(2\partial_{x_1}^4-i(\sqrt{3}-1)\partial_{x_1}^2\partial_{x_2}-4\partial_{x_2}^2\right).
\end{equation*}
This local limit theorem is illustrated in Figure \ref{fig:ex1_intro} which shows $\Re(\phi^{(n)})$ and the approximation $\Re(e^{-i\pi x_2/3}K_{\Lambda}^n)$ when $n=100$. 
\end{example}

\begin{example}\label{ex:intro2} Consider $\phi:\mathbb{Z}^2\rightarrow \mathbb{R}$ defined by $\phi=(\phi_1+\phi_2)/512$, where
\begin{equation*}
\phi_1(x_1,x_2)=
\begin{cases}
326 & (x_1,x_2)=(0,0)\\
20 & (x_1,x_2)=(\pm 2,0)\\
1 & (x_1,x_2)=(\pm 4,0)\\
64 & (x_1,x_2)=(0,\pm 1)\\
-16 & (x_1,x_2)=(0,\pm 2)\\
0 & \mbox{otherwise}
\end{cases}
\hspace{.5cm}\mbox{ and }\hspace{.5cm}
\phi_2(x_1,x_2)=
\begin{cases}
76 & (x_1,x_2)=(1,0)\\
52 & (x_1,x_2)=(-1,0)\\
\mp 4 & (x_1,x_2)=(\pm 3,0)\\
\mp 6 & (x_1,x_2)=(\pm 1,1)\\
\mp 6 & (x_1,x_2)=(\pm 1,-1)\\
\pm 2 & (x_1,x_2)=(\pm 3,1)\\
\pm 2 & (x_1,x_2)=(\pm 3,-1)\\
0 & \mbox{otherwise}.
\end{cases}
\end{equation*}
\begin{figure}[h!]
\begin{center}
\resizebox{\textwidth}{!}{
	    \begin{subfigure}[5cm]{0.5\textwidth}
		\includegraphics[width=\textwidth]{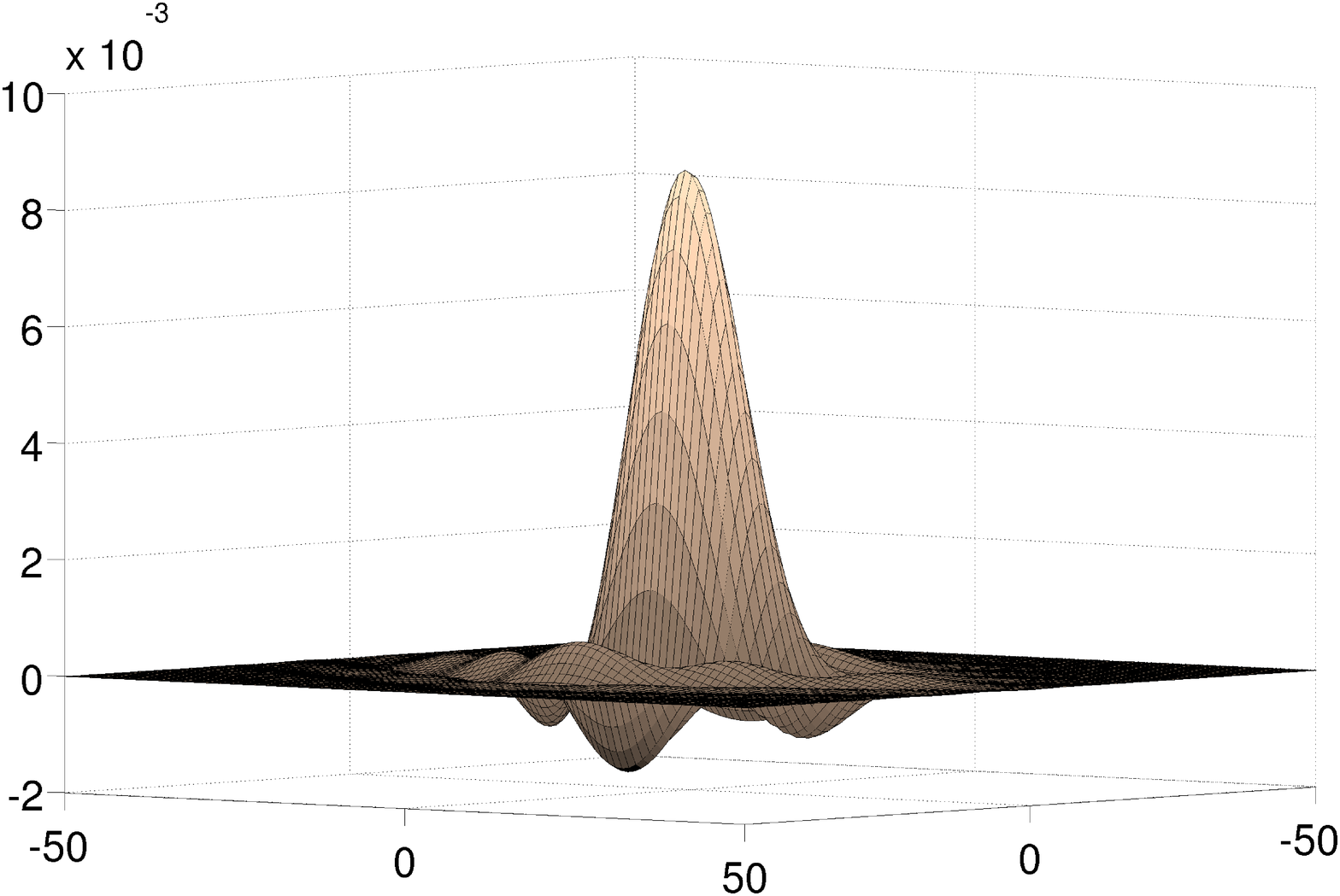}
		\caption{$\phi^{(n)}$ for $n=10,000$}
		\label{fig:ex5phi100_1}
	    \end{subfigure}
	    \begin{subfigure}[5cm]{0.5\textwidth}
		\includegraphics[width=\textwidth]{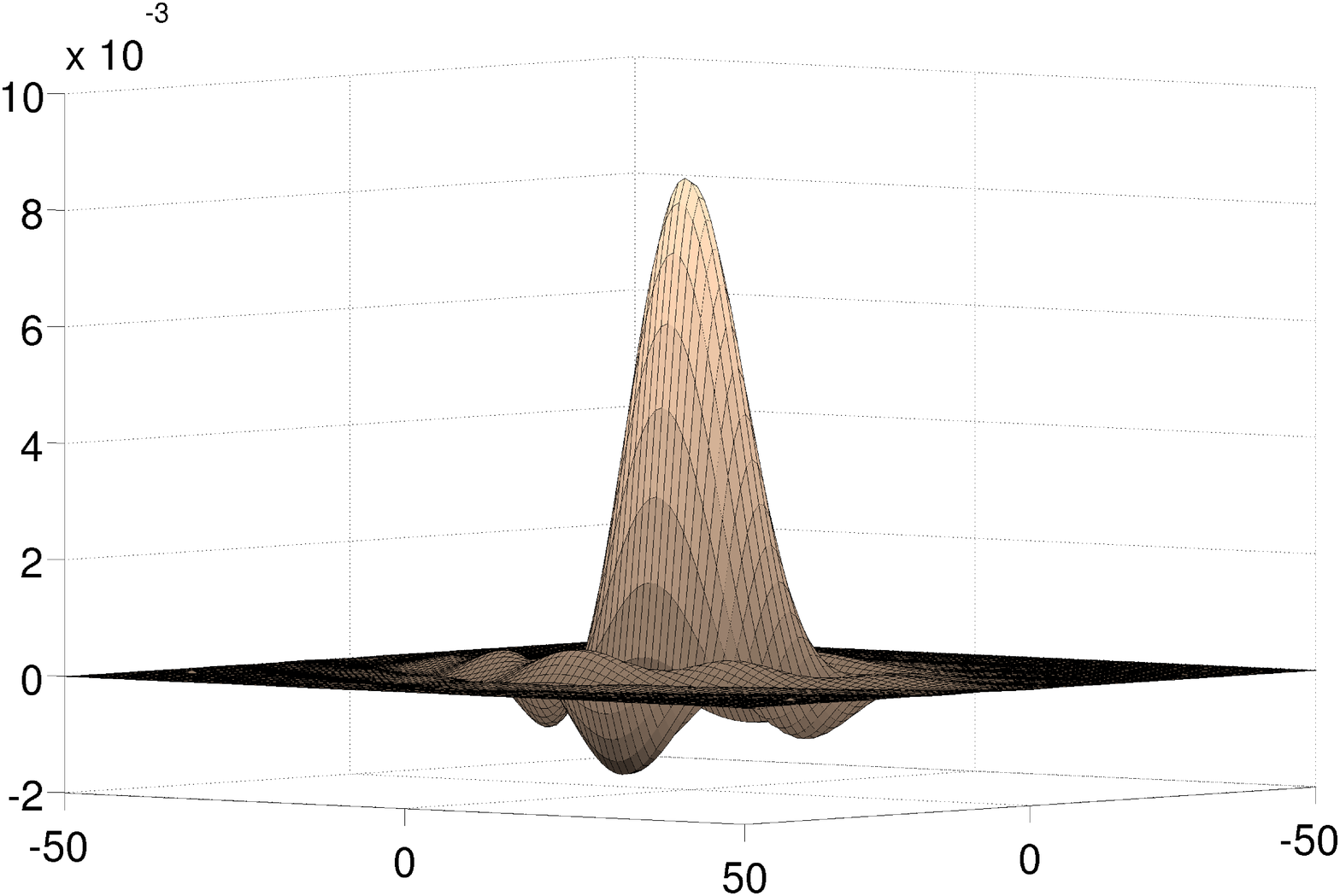}
		\caption{$K_{\Lambda}^n$ for $n=10,000$}
		\label{fig:ex5HP100_1}
	    \end{subfigure}}
\caption{The graphs of $\phi^{(n)}$ and $K_{\Lambda}^n$ for $n=10,000$.}
\label{fig:ex2_intro}
\end{center}
\end{figure}

\noindent In this example, the following local limit theorem, which is illustrated by Figure \ref{fig:ex2_intro}, describes the limiting behavior of $\phi^{(n)}$. We have
\begin{equation*}
\phi^{(n)}(x_1,x_2)=K^n_{\Lambda}(x_1,x_2)+o(n^{-5/12})
\end{equation*}
uniformly for $(x_1,x_2)\in\mathbb{Z}^2$ where $K_{\Lambda}$ is again a fundamental solution to $\partial_t+\Lambda=0$ where, in this case, 
\begin{equation*}
\Lambda=\frac{1}{64}\left(-\partial_{x_1}^6+2\partial_{x_2}^4+2\partial_{x_1}^3\partial_{x_2}^2\right).
\end{equation*}
\end{example}

\begin{example}\label{ex:intro3} Consider $\phi:\mathbb{Z}^2\rightarrow\mathbb{R}$ defined by 
\begin{equation*}
\phi(x,y)=\begin{cases}
	    3/8 & (x_1,x_2)=(0,0)\\
           1/8 & (x_1,x_2)=\pm(1,1)\\
           1/4 & (x_1,x_2)=\pm(1,-1)\\
            -1/16& (x_1,x_2)=\pm( 2,-2)\\
           0 & \mbox{otherwise}.
          \end{cases}
\end{equation*}
Here, the following local limit theorem is valid: 
\begin{equation*}\nonumber\label{eq:ex_3}
\phi^{(n)}(x_1,x_2)=\left(1+e^{i\pi(x_1+x_2)}\right)K_{\Lambda}^n(x_1,x_2)+o(n^{-3/4})\\
\end{equation*}
uniformly for $(x_1,x_2)\in\mathbb{Z}^2$. Here again, the attractor $K_{\Lambda}$ is the fundamental solution to $\partial_t+\Lambda=0$ where
\begin{equation*}
\Lambda=-\frac{1}{8}\partial_{x_1}^2+\frac{23}{384}\partial_{x_1}^4-\frac{1}{4}\partial_{x_1}\partial_{x_2}-\frac{25}{96}\partial_{x_1}^3\partial_{x_2}-\frac{1}{8}\partial_{x_2}^2+\frac{23}{64}\partial_{x_1}^2\partial_{x_2}^2-\frac{25}{96}\partial_{x_1}\partial_{x_2}^3+\frac{23}{384}\partial_{x_2}^4.
\end{equation*}
\end{example}

\noindent Looking back at preceding examples, we note that the operators appearing in Examples \ref{ex:intro1} and \ref{ex:intro2} are both positive-semi-elliptic and consist only of their principal parts. This is easily verified, for $\mathbf{n}=(4,2)=2(2,1)$ in Example \ref{ex:intro1} and $\mathbf{n}=(6,4)=2(3,2)$  in Example \ref{ex:intro2}. In contrast to Examples \ref{ex:intro1} and \ref{ex:intro2}, the operator $\Lambda$ which appears in Example \ref{ex:intro3} is not semi-elliptic in the given coordinate system. After careful study, the $\Lambda$ appearing in Example \ref{ex:intro3} can be written equivalently as
\begin{equation}\label{eq:DiagonalizedExample} 
\Lambda=-\frac{1}{8}\partial_{v_1}^2+\frac{23}{384}\partial_{v_2}^4
\end{equation}
where $\partial_{v_1}$ is the directional derivative in the $v_1=(1,1)$ direction and $\partial_{v_2}$ is the directional derivative in the $v_2=(1,-1)$ direction. In this way, $\Lambda$ is seen to be semi-elliptic with respect to some basis $\{v_1,v_2\}$ of $\mathbb{R}^2$ and, with respect to this basis, we have $\mathbf{n}=(2,4)=2(1,2)$.  For this reason, our formulation of nondegenerate-homogeneous operators (and positive-homogeneous operators), given in the next section, is made in a basis-independent way.\\

\noindent All of the operators appearing in Examples \ref{ex:intro1}, \ref{ex:intro2} and \ref{ex:intro3} share two important properties: homogeneity and positivity (in the sense of symbols). While we make these notions precise in the next section, loosely speaking, homogeneity is the property that $\Lambda$ ``plays well'' with some dilation structure on $\mathbb{R}^d$, though this structure is different in each example. Further, homogeneity for $\Lambda$ is reflected by an analogous one for the corresponding heat kernel $K_{\Lambda}$; in fact, the specific dilation structure is, in some sense, selected by $\phi^{(n)}$ as $n\rightarrow\infty$ and leads to the corresponding local limit theorem. In further discussion of these examples, a very natural question arises: Given $\phi:\mathbb{Z}^d\rightarrow\mathbb{C}$, how does one compute the operator $\Lambda$ whose heat kernel $K_{\Lambda}$ appears as the attractor in the local limit theorem for $\phi^{(n)}$? In the examples we have looked at, one studies the Taylor expansion of the Fourier transform $\hat{\phi}$ of $\phi$ near its local extrema and, here, the symbol of the relevant operator $\Lambda$ appears as certain scaled limit of this Taylor expansion. In general, however, this is a very delicate business and, at present, there is no known algorithm to determine these operators. In fact, it is possible that multiple (distinct) operators can appear by looking at the Taylor expansions about distinct local extrema of $\hat{\phi}$ (when they exist) and, in such cases, the corresponding local limit theorems involve sums of of heat kernels--each corresponding to a distinct $\Lambda$. This study is carried out in the article \cite{Randles2015a} wherein local limit theorems involve the heat kernels of the positive-homoegeneous operators studied in the present article. We note that the theory presented in \cite{Randles2015a} is not complete, for there are cases in which the associated Taylor approximations yield symbols corresponding to operators $\Lambda$ which fail to be positive-homogeneous (and hence fail to be positive-semi-elliptic) and further, the heat kernels of these (degenerate) operators appear as limits of oscillatory integrals which correspond to the presence of ``odd" terms in $\Lambda$, e.g., the Airy function. In one dimension, a complete theory of local limit theorems is known for the class of finitely supported functions $\phi:\mathbb{Z}\rightarrow\mathbb{C}$. Beyond one dimension, a theory for local limit theorems of complex-valued functions, in which the results of \cite{Randles2015a} will fit, remains open.\\


\noindent The subject of this paper is an account of positive-homogeneous operators and their corresponding heat equations. In Section \ref{sec:HomogeneousOperators}, we introduce positive-homogeneous operators and study their basic properties; therein, we show that each positive-homogeneous operator is semi-elliptic in some coordinate system. Section \ref{sec:Holder} develops the necessary background to introduce the class of variable-coefficient operators studied in this article; this is the class of $(2\mathbf{m,v})$-positive-semi-elliptic operators introduced in Section \ref{sec:UniformlySemiElliptic}--each of which is comparable to a constant-coefficient positive-homogeneous operator. In Section \ref{sec:FundamentalSolution}, we study the heat equations corresponding to uniformly $(2\mathbf{m,v})$-positive-semi-elliptic operators with H\"{o}lder continuous coefficients. Specifically, we use the famous method of E. E. Levi, adapted to parabolic systems by A. Friedman and S. D. Eidelman, to construct a fundamental solution to the corresponding heat equation. Our results in this direction are captured by those of S. D. Eidelman \cite{Eidelman1960} and the works of his collaborators, notably S. D. Ivashyshen and A. N. Kochubei \cite{Eidelman2004}, concerning $2\vec{b}$-parabolic systems. Our focus in this presentation is to highlight the essential role played by the Legendre-Fenchel transform in heat kernel estimates which, to our knowledge, has not been pointed out in the context of semi-elliptic operators. In a forthcoming work, we study an analogous class of operators, written in divergence form, with measurable-coefficients and their corresponding heat kernels. This class of measurable-coefficient operators does not appear to have been previously studied. The results presented here, using the Legendre-Fenchel transform, provides the background and context for our work there. 

\subsection{Preliminaries}\label{subsec:Preliminaries}

\textbf{Fourier Analysis:} Our setting is a real $d$-dimensional vector space $\mathbb{V}$ equipped with Haar (Lebesgue) measure $dx$ and the standard smooth structure; we do not affix $\mathbb{V}$ with a norm or basis. The dual space of $\mathbb{V}$ is denoted by $\mathbb{V}^*$ and the dual pairing is denoted by $\xi(x)$ for $x\in\mathbb{V}$ and $\xi\in\mathbb{V}^*$. Let $d\xi$ be the Haar measure on $\mathbb{V}^*$ which we take to be normalized so that our convention for the Fourier transform and inverse Fourier transform, given below, makes each unitary. Throughout this article, all functions on $\mathbb{V}$ and $\mathbb{V}^*$ are understood to be complex-valued. The usual Lebesgue spaces are denoted by $L^p(\mathbb{V})=L^p(\mathbb{V},dx)$ and equipped with their usual norms $\|\cdot\|_p$ for $1\leq p\leq \infty$. In the case that $p=2$, the corresponding inner product on $L^2(\mathbb{V})$ is denoted by $\langle\cdot,\cdot\rangle$. Of course, we will also work with $L^2(\mathbb{V}^*)
:=L^2(\mathbb{V}^*,d\xi)$; here the $L^2$-norm and inner product will be denoted by $\|\cdot\|_{2^*}$ and $\langle\cdot,\cdot\rangle_*$ respectively. The Fourier transform $\mathcal{F}:L^2(\mathbb{V})\rightarrow L^2(\mathbb{V}^*)$  and inverse Fourier transform $\mathcal{F}^{-1}:L^2(\mathbb{V}^*)\rightarrow L^2(\mathbb{V})$ are initially defined for Schwartz functions $f\in \mathcal{S}(\mathbb{V})$ and $g\in\mathcal{S}(\mathbb{V}^*)$ by
\begin{equation*}
\mathcal{F}(f)(\xi)=\hat{f}(\xi)=\int_{\mathbb{V}}e^{i\xi(x)}f(x)\,dx\hspace{1cm}\mbox{and}\hspace{1cm}\mathcal{F}^{-1}(g)(x)=\check{g}(x)=\int_{\mathbb{V}^*}e^{-i\xi(x)}g(\xi)\,d\xi
\end{equation*}
for $\xi\in\mathbb{V}^*$  and $x\in\mathbb{V}$ respectively. 

For the remainder of this article (mainly when duality isn't of interest), $W$ stands for any real $d$-dimensional vector space (and so is interchangeable with $\mathbb{V}$ or $\mathbb{V}^*$). For a non-empty open set $\Omega\subseteq W$, we denote by $C(\Omega)$ and $C_b(\Omega)$ the set of continuous functions on $\Omega$ and bounded continuous functions on $\Omega$, respectively. The set of smooth functions on $\Omega$ is denoted by $C^{\infty}(\Omega)$ and the set of compactly supported smooth functions on $\Omega$ is denoted by $C^{\infty}_0(\Omega)$. We denote by $\mathcal{D}'(\Omega)$ the space of distributions on $\Omega$; this is dual to the space $C^{\infty}_0(\Omega)$ equipped with its usual topology given by seminorms.  A partial differential operator $H$ on $W$ is said to be \textit{hypoelliptic} if it satisfies the following property: Given any open set $\Omega\subseteq W$ and any distribution $u\in \mathcal{D}'(\Omega)$ which satisfies $Hu=0$ in $\Omega$, then necessarily $u\in C^{\infty}(\Omega)$.\\

\noindent \textbf{Dilation Structure:} Denote by $\End(W)$ and $\Gl(W)$ the set of endomorphisms and isomorphisms of $W$ respectively. Given $E\in\End(W)$, we consider the one-parameter group $\{t^E\}_{t>0}\subseteq \Gl(W)$ defined by 
\begin{equation*}
t^E=\exp((\log t)E)=\sum_{k=0}^{\infty}\frac{(\log t)^k}{k!}E^k
\end{equation*}
for $t>0$.  These one-parameter subgroups of $\Gl(W)$ allow us to define continuous one-parameter groups of operators on the space of distributions as follows: Given $E\in\End(W)$ and $t>0$, first define $\delta_t^E(f)$ for $f\in C_0^{\infty}(W)$ by $\delta_t^E(f)(x)=f(t^Ex)$ for $x\in W$. Extending this to the space of distribution on $W$ in the usual way, the collection $\{\delta_t^E\}_{t>0}$ is a continuous one-parameter group of operators on $\mathcal{D}'(W)$; it will allow us to define homogeneity for partial differential operators in the next section.\\

\noindent \textbf{Linear Algebra, Polynomials And The Rest:} Given a basis $\mathbf{w}=\{w_1,w_2,\dots,w_d\}$ of $W$, we define the map $\phi_{\mathbf{w}}:W\rightarrow\mathbb{R}^d$ by setting $\phi_{\mathbf{w}}(w)=(x_1,x_2,\dots,x_d)$ whenever $w=\sum_{l=1}^d x_l w_l$. This map defines a global coordinate system on $W$; any such coordinate system is said to be a linear coordinate system on $W$. By definition, a polynomial on $W$ is a function $P:W\rightarrow\mathbb{C}$ that is a polynomial function in every (and hence any) linear coordinate system on $W$. A polynomial $P$ on $W$ is called a nondegenerate polynomial if $P(w)\neq 0$ for all $w\neq 0$. Further, $P$ is called a positive-definite polynomial if its real part, $R=\Re P$, is non-negative and has $R(w)=0$ only when $w=0$. The symbols $\mathbb{R,C,Z}$ mean what they usually do, $\mathbb{N}$ denotes the set of non-negative integers and $\mathbb{I}=[0,1]\subseteq\mathbb{R}$. The symbols $\mathbb{R}_+$, $\mathbb{N}_+$ and $\mathbb{I}_+$ denote the set of strictly positive elements of $\mathbb{R}$, $\mathbb{N}$ and $\mathbb{I}$ respectively. Likewise, $\mathbb{R}_+^d$, $\mathbb{N}_+^d$ and $\mathbb{I}_+^d$ respectively denote the set of $d$-tuples of these aforementioned sets. Given $\alpha=(\alpha_1,\alpha_2,\dots,\alpha_d)\in\mathbb{R}_+^d$ and a basis $\mathbf{w}=\{w_1,w_2,\dots,w_d\}$ of $W$, we denote by $E_{\mathbf{w}}^\alpha$ the isomorphism of $W$ defined by
\begin{equation}\label{eq:DefofE}
E_{\mathbf{w}}^\alpha w_k=\frac{1}{\alpha_k}w_k
\end{equation}
for $k=1,2,\dots, d$. We say that two real-valued functions $f$ and $g$ on a set $X$ are comparable if, for some positive constant $C$, $C^{-1}f(x)\leq g(x)\leq C f(x)$ for all $x\in X$; in this case we write $f\asymp g$. Adopting the summation notation for semi-elliptic operators of L. H\"{o}rmander's treatise \cite{Hormander1983}, for a fixed $\mathbf{n}=(n_1,n_2,\dots,n_d)\in\mathbb{N}_+^d$, we write
\begin{equation*}
|\beta:\mathbf{n}|=\sum_{k=1}^d\frac{\beta_k}{m_k}
\end{equation*} for all multi-indices $\beta=(\beta_1,\beta_2,\dots,\beta_d)\in\mathbb{N}^d$. Finally, throughout the estimates made in this article, constants denoted by $C$ will change from line to line without explicit mention.

\section{Homogeneous operators}\label{sec:HomogeneousOperators}
In this section we introduce two important classes of homogeneous constant-coefficient on $\mathbb{V}$. These operators will serve as ``model'' operators in our theory in the way that integer powers of the Laplacian serves a model operators in the elliptic theory of partial differential equations. To this end, let $\Lambda$ be a constant-coefficient partial differential operator on $\mathbb{V}$ and let $P:\mathbb{V}^*\rightarrow\mathbb{C}$ be its symbol. Specifically, $P$ is the polynomial on $\mathbb{V}^*$ defined by $P(\xi)=e^{-i\xi(x)}\Lambda(e^{i\xi(x)})$ for $\xi\in\mathbb{V}^*$ (this is independent of $x\in\mathbb{V}$ precisely because $\Lambda$ is a constant-coefficient operator). We first introduce the following notion of homogeneity of operators; it is mirrored by an analogous notion for symbols which we define shortly. 

\begin{definition}
Given $E\in\End(\mathbb{V})$, we say that a constant-coefficient partial differential operator $\Lambda$ is homogeneous with respect to the one-parameter group $\{\delta_t^E\}$ if
\begin{equation*}
\delta_{1/t}^E\circ \Lambda\circ \delta_t^E=t\Lambda
\end{equation*}
for all $t>0$; in this case we say that $E$ is a member of the exponent set of $\Lambda$ and write $E\in\Exp(\Lambda)$. 
\end{definition}

\noindent A constant-coefficient partial differential operator $\Lambda$ need not be homogeneous with respect to a unique one-parameter group $\{\delta_t^E\}$, i.e., $\Exp(\Lambda)$ is not necessarily a singleton. For instance, it is easily verified that, for the Laplacian $-\Delta$ on $\mathbb{R}^d$,
\begin{equation*}
\Exp(-\Delta)=2^{-1}I+\mathfrak{o}_d
\end{equation*}
where $I$ is the identity and $\mathfrak{o}_d$ is the Lie algebra of the orthogonal group, i.e., is given by the set of skew-symmetric matrices. Despite this lack of uniqueness, when $\Lambda$ is equipped with a nondegenerateness condition (see Definition \ref{def:HomogeneousOperators}), we will find that trace is the same for each member of $\Exp(\Lambda)$ and this allows us to uniquely define an ``order" for $\Lambda$; this is Lemma \ref{lem:Trace}.\\

\noindent Given a constant coefficient operator $\Lambda$ with symbol $P$, one can quickly verify that $E\in\Exp(\Lambda)$ if and only if
\begin{equation}\label{eq:homofsymbol}
tP(\xi)=P(t^F\xi)
\end{equation}
for all $t>0$ and $\xi\in\mathbb{V}^*$ where $F=E^*$ is the adjoint of $E$. More generally, if $P$ is any continuous function on $W$ and \eqref{eq:homofsymbol} is satisfied for some $F\in\End(\mathbb{V}^*)$, we say that $P$ \textit{is homogeneous with respect to} $\{t^F \}$ and write $F\in\Exp(P)$. This admitted slight abuse of notation should not cause confusion. In this language, we see that $E\in \Exp(\Lambda)$ if and only if $E^*\in\Exp(P)$.\\

\noindent We remark that the notion of homogeneity defined above is similar to that put forth for homogeneous operators on homogeneous (Lie) groups, e.g., Rockland operators \cite{Folland1982}. The difference is mostly a matter of perspective: A homogeneous group $G$ is equipped with a fixed dilation structure, i.e., it comes with a one-parameter group $\{\delta_t\}$, and homogeneity of operators is defined with respect to this fixed dilation structure. By contrast, we fix no dilation structure on $\mathbb{V}$ and formulate homogeneity in terms of an operator $\Lambda$ and the existence of a one-parameter group $\{\delta_t^E\}$ that ``plays'' well with $\Lambda$ in sense defined above. As seen in the study of convolution powers on the square lattice (see \cite{Randles2015a}), it useful to have this freedom.

\begin{definition}\label{def:HomogeneousOperators}
Let $\Lambda$ be constant-coefficient partial differential operator on $\mathbb{V}$ with symbol $P$. We say that $\Lambda$ is a nondegenerate-homogeneous operator if $P$ is a nondegenerate polynomial and $\Exp(\Lambda)$ contains a diagonalizable endomorphism. We say that $\Lambda$ is a positive-homogeneous operator if $P$ is a positive-definite polynomial and $\Exp(\Lambda)$ contains a diagonalizable endomorphism.
\end{definition}

\noindent For any polynomial $P$ on a finite-dimensional vector space $W$, $P$ is said to be \textit{nondegenerate-homogeneous} if $P$ is nondegenerate and $\Exp(P)$, defined as the set of $F\in\End(W)$ for which \eqref{eq:homofsymbol} holds, contains a diagonalizable endomorphism. We say that $P$ is \textit{positive-homogeneous} if it is a positive-definite polynomial and $\Exp(P)$ contains a diagonalizable endomorphism. In this language, we have the following proposition.

\begin{proposition}\label{prop:OperatorSymbolEquivalence}
Let $\Lambda$ be a positive homogeneous operator on $\mathbb{V}$ with symbol $P$. Then $\Lambda$ is a nondegenerate-homogeneous operator if and only if $P$ is a nondegenerate-homogeneous polynomial. Further, $\Lambda$ is a positive-homogeneous operator if and only if $P$ is a positive-homogeneous polynomial. 
\end{proposition}

\begin{proof}
Since the adjectives ``nondegenerate'' and ``positive'', in the sense of both operators and polynomials, are defined in terms of the symbol $P$, all that needs to be verified is that $\Exp(\Lambda)$ contains a diagonalizable endomorphism if and only if $\Exp(P)$ contains a diagonalizable endomorphism. Upon recalling that $E\in\Exp(\Lambda)$ if and only if $E^*\in\Exp(P)$, this equivalence is verified by simply noting that diagonalizability is preserved under taking adjoints. 
\end{proof}

\begin{remark} To capture the class of nondegenerate-homogeneous operators (or positive-homogeneous operators), in addition to requiring that that the symbol $P$ of an operator $\Lambda$ be nondegenerate (or positive-definite), one can instead demand only that $\Exp(\Lambda)$ contains an endomorphism whose characteristic polynomial factors over $\mathbb{R}$ or, equivalently, whose spectrum is real. This a priori weaker condition is seen to be sufficient by an argument which makes use of the Jordan-Chevalley decomposition. In the positive-homogeneous case, this argument is carried out in \cite{Randles2015a} (specifically Proposition 2.2) wherein positive-homogeneous operators are first defined by this (a priori weaker) condition. For the nondegenerate case, the same argument pushes through with very little modification.
\end{remark}

\noindent We observe easily that all positive-homogeneous operators are nondegenerate-homogeneous. It is the ``heat'' kernels corresponding to positive-homogeneous operators that naturally appear in \cite{Randles2015a} as the attractors of convolution powers of complex-valued functions. The following proposition highlights the interplay between positive-homogeneity and nondegenerate-homogeneity for an operator $\Lambda$ on $\mathbb{V}$ and its corresponding ``heat'' operator $\partial_t+\Lambda$ on $\mathbb{R}\oplus\mathbb{V}$.

\begin{proposition}\label{prop:Dichotomy}
Let $\Lambda$ be a constant-coefficient partial differential operator on $\mathbb{V}$ whose exponent set $\Exp(\Lambda)$ contains a diagonalizable endomorphism. Let $P$ be the symbol of $\Lambda$, set $R=\Re P$, and assume that there exists $\xi\in\mathbb{V}^*$ for which $R(\xi)>0$. We have the following dichotomy: $\Lambda$ is a positive-homogeneous operator on $\mathbb{V}$ if and only if $\partial_t+\Lambda$ is a nondegenerate-homogeneous operator on $\mathbb{R}\oplus\mathbb{V}$.
\end{proposition}

\begin{proof}
Given a diagonalizable endomorphism $E\in\Exp(\Lambda)$, set $E_1=I\oplus E$ where $I$ is the identity on $\mathbb{R}$. Obviously, $E_1$ is diagonalizable. Further, for any $f\in C_0^{\infty}(\mathbb{R}\oplus \mathbb{V})$, 
\begin{eqnarray*}
\left( (\partial_t+\Lambda)\circ \delta_{s}^{E_1}\right)(f)(t,x)&=& \left (\partial_t\left(f\left(st,s^Ex\right)\right)+\Lambda  \left(f\left(st,s^Ex\right)\right)\right)\\
&=&s(\partial_t+\Lambda)(f)(st,s^Ex)=s\left(\delta_s^{E_1}\circ\left(\partial_t+\Lambda\right)\right)(f)(t,x)
\end{eqnarray*}
for all $s>0$ and $(t,x)\in\mathbb{R}\oplus\mathbb{V}$. Hence
\begin{equation*}
\delta_{1/s}^{E_1}\circ (\partial_t+\Lambda)\circ\delta_t^{E_1}=s(\partial_t+\Lambda)
\end{equation*}
for all $s>0$ and therefore $E_1\in\Exp(\partial_t+\Lambda)$. 

It remains to show that $P$ is positive-definite if and only if the symbol of $\partial_t+\Lambda$ is nondegenerate.  To this end, we first compute the symbol of $\partial_t+\Lambda$ which we denote by $Q$. Since the dual space of $\mathbb{R}\oplus\mathbb{V}$ is isomorphic to $\mathbb{R}\oplus\mathbb{V}^*$, the characters of $\mathbb{R}\oplus\mathbb{V}$ are represented by the collection of maps $\left(\mathbb{R}\oplus\mathbb{V}\right)\ni(t,x)\mapsto \exp(-i(\tau t+\xi(x)))$ where $(\tau,\xi)\in\mathbb{R}\oplus\mathbb{V}^*$. Consequently,
\begin{equation*}
Q(\tau,\xi)=e^{-i(\tau t+\xi(x))}\left(\partial_t+\Lambda\right)(e^{i(\tau t+\xi(x)})=i\tau+P(\xi)
\end{equation*}
for $(\tau,\xi)\in\mathbb{R}\oplus\mathbb{V}^*$. We note that $P(0)=0$ because $E^*\in\Exp(P)$; in fact, this happens whenever $\Exp(P)$ is non-empty. Now if $P$ is a positive-definite polynomial, $\Re Q(\tau,\xi)=\Re P(\xi)=R(\xi)>0$ whenever $\xi\neq 0$. Thus to verify that $Q$ is a nondegenerate polynomial, we simply must verify that $Q(\tau,0)\neq 0$ for all non-zero $\tau\in\mathbb{R}$. This is easy to see because, in light of the above fact,  $Q(\tau,0)=i\tau+P(0)=i\tau\neq 0$ whenever $\tau\neq 0$ and hence $Q$ is nondegenerate. For the other direction, we demonstrate the validity of the contrapositive statement. Assuming that $P$ is not positive-definite, an application of the intermediate value theorem, using the condition that $R(\xi)>0$ for some $\xi\in\mathbb{V}^*$, guarantees that $R(\eta)=0$ for some non-zero $\eta\in\mathbb{V}^*$. Here, we observe that $Q(\tau,\eta)=i(\tau+\Im P(\eta))=0$ when $(\tau,\eta)=(-\Im P(\eta),\eta)$ and hence $Q$ is not nondegenerate.
\end{proof}

\noindent We will soon return to the discussion surrounding a positive-homogeneous operator $\Lambda$ and its heat operator $\partial_t+\Lambda$. It is useful to first provide representation formulas for nondegenerate-homogeneous and positive-homogeneous operators. Such representations connect our homogeneous operators to the class of semi-elliptic operators discussed in the introduction.  To this end, we define the ``base'' operators on $\mathbb{V}$. First, for any element $u\in\mathbb{V}$, we consider the differential operator $D_u:\mathcal{D}'(\mathbb{V})\rightarrow\mathcal{D}'(\mathbb{V})$ defined originally for $f\in C_0^{\infty}(\mathbb{V})$ by
\begin{equation*}
(D_uf)(x)=i\frac{\partial f}{\partial u}(x)=i\left(\lim_{t\rightarrow 0}\frac{f(x+tu)-f(x)}{t}\right)
\end{equation*}
for $x\in\mathbb{V}$. Fixing a basis $\mathbf{v}=\{v_1,v_2,\dots,v_d\}$ of $\mathbb{V}$, we introduce, for each multi-index $\beta\in\mathbb{N}^d$, $D_{\mathbf{v}}^{\beta}=\left(D_{v_1}\right)^{\beta_1}\left(D_{v_2}\right)^{\beta_2}\cdots\left( D_{v_d}\right)^{\beta_d}$. 

\begin{proposition}\label{prop:OperatorRepresentation}
Let $\Lambda$ be a nondegenerate-homogeneous operator on $\mathbb{V}$. Then there exist a basis $\mathbf{v}=\{v_1,v_2,\dots,v_d\}$ of $\mathbb{V}$ and $\mathbf{n}=(n_1,n_2,\dots,n_d)\in\mathbb{N}_+^d$ for which
\begin{equation}\label{eq:OperatorRepresentation1}
\Lambda=\sum_{|\beta:\mathbf{n}|=1}a_{\beta}D_{\mathbf{v}}^\beta. 
\end{equation}
where $\{a_{\beta}\}\subseteq\mathbb{C}$. The isomorphism $E_{\mathbf{v}}^{\mathbf{n}}\in\Gl(\mathbb{V})$, defined by \eqref{eq:DefofE}, is a member of $\Exp(\Lambda)$. Further, if $\Lambda$ is positive-homogeneous, then $\mathbf{n}=2\mathbf{m}$ for $\mathbf{m}=(m_1,m_2,\dots,m_d)\in\mathbb{N}_+^d$ and hence
\begin{equation*}
\Lambda=\sum_{|\beta:\mathbf{m}|=2}a_{\beta}D_{\mathbf{v}}^\beta.
\end{equation*}
\end{proposition}

\noindent We will sometimes refer to the $\mathbf{n}$ and $\mathbf{m}$ of the proposition as \textit{weights}. Before addressing the proposition, we first prove the following mirrored result for symbols.

\begin{lemma}\label{lem:PolynomialRepresentation}
Let $P$ be a nondegenerate-homogeneous polynomial on a $d$-dimensional real vector space $W.$ Then there exists a basis $\mathbf{w}=\{w_1,w_2,\dots,w_d\}$ of $W$ and $\mathbf{n}=(n_1,n_2,\dots,n_d)\in\mathbb{N}_+^d$ for which
\begin{equation*}
P(\xi)=\sum_{|\beta:\mathbf{n}|=1}a_{\beta}\xi^{\beta}
\end{equation*}
for all $\xi=\xi_1 w_1+\xi_2 w_2+\cdots+\xi_dw_d\in W$ where $\xi^{\beta}:=\left(\xi_1\right)^{\beta_1}\left(\xi_2\right)^{\beta_2}\cdots\left(\xi_d\right)^{\beta_d}$ and $\{a_\beta\}\subseteq\mathbb{C}$. The isomorphism $E_{\mathbf{w}}^{\mathbf{n}}\in\Gl(\mathbb{V})$, defined by \eqref{eq:DefofE}, is a member of $\Exp(P)$. Further, if $P$ is a positive-definite polynomial, i.e., it is positive-homogeneous, then $\mathbf{n}=2\mathbf{m}$ for $\mathbf{m}=(m_1,m_2,\dots,m_d)\in\mathbb{N}_+^d$ and hence
\begin{equation*}
P(\xi)=\sum_{|\beta:\mathbf{m}|=2}a_{\beta}\xi^{\beta}
\end{equation*}
for $\xi\in W$.
\end{lemma}

\begin{proof}
Let $E\in\Exp(P)$ be diagonalizable and select a basis $\mathbf{w}=\{w_1,w_2,\dots,w_d\}$ which diagonalizes $E$, i.e.,  $Ew_k=\delta_k w_k$ where $\delta_k\in\mathbb{R}$ for $k=1,2,\dots,d$. Because $P$ is a polynomial, there exists a finite collection $\{a_{\beta}\}\subseteq\mathbb{C}$ for which
\begin{equation*}
P(\xi)=\sum_{\beta}a_{\beta}\xi^{\beta}
\end{equation*}
for $\xi\in W$. By invoking the homogeneity of $P$ with respect to $E$ and using the fact that $t^Ew_k=t^{\delta_k}w_k$ for $k=1,2,\dots, d$, we have
\begin{equation*}
t\sum_{\beta}a_{\beta}\xi^{\beta}=\sum_{\beta}a_{\beta}(t^E\xi)^{\beta}=\sum_{\beta}a_{\beta}t^{\delta\cdot\beta}\xi^{\beta}
\end{equation*}
for all $\xi\in W$ and $t>0$ where $\delta\cdot\beta=\delta_1\beta_1+\delta_2\beta_2+\cdots+\delta_d\beta_d$. In view of the nondegenerateness of $P$, the linear independence of distinct powers of $t$ and the polynomial functions $\xi\mapsto\xi^{\beta}$, for distinct multi-indices $\beta$, as $C^{\infty}$ functions ensures that $a_{\beta}=0$ unless $\beta\cdot\delta=1$. We can therefore write
\begin{equation}\label{eq:1}
P(\xi)=\sum_{\beta\cdot\delta=1}a_{\beta}\xi^{\beta}
\end{equation}
for $\xi\in W$. We now determine $\delta=(\delta_1,\delta_2,\dots,\delta_d)$ by evaluating this polynomial along the coordinate axes. To this end, by fixing $k=1,2,\dots, d$ and setting $\xi=xw_k$ for $x\in\mathbb{R}$, it is easy to see that the summation above collapses into a single term $a_{\beta}x^{|\beta |}$ where $\beta=|\beta | e_k=(1/\delta_k)e_k$ (here $e_k$ denotes the usual $k$th-Euclidean basis vector in $\mathbb{R}^d$). Consequently, $n_k:=1/\delta_k\in\mathbb{N}_+$ for $k=1,2,\dots,d$ and thus, upon setting $\mathbf{n}=(n_1,n_2,\dots,n_d)$, \eqref{eq:1} yields
\begin{equation*}
P(\xi)=\sum_{|\beta:\mathbf{n}|=1}a_{\beta}\xi^{\beta}
\end{equation*} 
for all $\xi\in W$ as was asserted. In this notation, it is also evident that $E_{\mathbf{w}}^{\mathbf{n}}=E\in\Exp(P)$. Under the additional assumption that $P$ is positive-definite, we again evaluate $P$ at the coordinate axes to see that $\Re P(xw_k)=\Re( a_{n_ke_k})x^{n_k}$ for $x\in\mathbb{R}$. In this case, the positive-definiteness of $P$ requires $\Re (a_{n_ke_k})>0$ and $n_k\in 2\mathbb{N}_+$ for each $k=1,2,\dots,d$. Consequently, $\mathbf{n}=2\mathbf{m}$ for $\mathbf{m}=(m_1,m_2,\dots,m_d)\in\mathbb{N}_+^d$ as desired.
\end{proof}

\begin{proof}[Proof of Proposition \ref{prop:OperatorRepresentation}]
Given a nondegenerate-homogeneous $\Lambda$ on $\mathbb{V}$ with symbol $P$, $P$ is necessarily a nondegenerate-homogeneous polynomial on $\mathbb{V}^*$ in view of Proposition \ref{prop:OperatorSymbolEquivalence}. We can therefore apply Lemma \ref{lem:PolynomialRepresentation} to select a basis $\mathbf{v}^*=\{v_1^*,v_2^*,\dots,v_d^*\}$ of $\mathbb{V}^*$ and $\mathbf{n}=(n_1,n_2,\dots,n_d)\in\mathbb{N}_+^d$ for which 
\begin{equation}
P(\xi)=\sum_{|\beta:\mathbf{n}|=1}a_{\beta}\xi^{\beta}
\end{equation}
for all $\xi=\xi_1 v_1^*+\xi_2 v_2^*+\cdots \xi_d v_d^*$ where $\{a_{\beta}\}\subseteq\mathbb{C}$. We will denote by $\mathbf{v}$, the dual basis to $\mathbf{v}^*$, i.e., $\mathbf{v}=\{v_1,v_2,\dots,v_d\}$ is the unique basis of $\mathbb{V}$ for which $v_k^*(v_l)=1$ when $k=l$ and $0$ otherwise. In view of the duality of the bases $\mathbf{v}$ and $\mathbf{v}^*$, it is straightforward to verify that, for each multi-index $\beta$, the symbol of $D_{\mathbf{v}}^{\beta}$ is $\xi^{\beta}$ in the notation of Lemma \ref{lem:PolynomialRepresentation}. Consequently, the constant-coefficient partial differential operator defined by the right hand side of \eqref{eq:OperatorRepresentation1} also has symbol $P$ and so it must be equal to $\Lambda$ because operators and symbols are in one-to-one correspondence. Using \eqref{eq:OperatorRepresentation1}, it is now straightforward to verify that $E_{\mathbf{v}}^{\mathbf{n}}\in\Exp(\Lambda)$. The assertion that $\mathbf{n}=2\mathbf{m}$ when $\Lambda$ is positive-homogeneous follows from the analogous conclusion of Lemma \ref{lem:PolynomialRepresentation} by the same line of reasoning.
\end{proof}

\noindent In view of Proposition \ref{prop:OperatorRepresentation}, we see that all nondegenerate-homogeneous operators are semi-elliptic in some linear coordinate system (that which is defined by $\mathbf{v}$). An appeal to Theorem 11.1.11 of \cite{Hormander1983} immediately yields the following corollary.

\begin{corollary}
Every nondegenerate-homogeneous operator $\Lambda$ on $\mathbb{V}$ is hypoelliptic.
\end{corollary}

\noindent Our next goal is to associate an ``order'' to each nondegenerate-homogeneous operator. For a positive-homogeneous operator $\Lambda$, this order will be seen to govern the on-diagonal decay of its heat kernel $K_{\Lambda}$ and so, equivalently, the ultracontractivity of the semigroup $e^{-t\Lambda}$.  With the help of Lemma \ref{lem:PolynomialRepresentation}, the few lemmas in this direction come easily.

\begin{lemma}\label{lem:PtoInfty}
Let $P$ be a nondegenerate-homogeneous polynomial on a $d$-dimensional real vector space $W$. Then $\lim_{\xi \rightarrow \infty}|P(\xi)|=\infty$; here $\xi\rightarrow \infty$ means that $|\xi|\rightarrow \infty$ in any (and hence every) norm on $W$.
\end{lemma}
\begin{proof}
The idea of the proof is to construct a function which bounds $|P|$ from below and obviously blows up at infinity. To this end, let $\mathbf{w}$ be a basis for $W$ and take $\mathbf{n}\in\mathbb{N}_+^d$ as guaranteed by Lemma \ref{lem:PolynomialRepresentation}; we have $E_{\mathbf{w}}^{\mathbf{n}}\in\Exp(P)$ where $E_{\mathbf{w}}^{\mathbf{n}}w_k=(1/n_k)w_k$ for $k=1,2,\dots, d$. Define $|\cdot|_{\mathbf{w}}^{\mathbf{n}}:W\rightarrow [0,\infty)$ by
\begin{equation*}
|\xi |_{\mathbf{w}}^{\mathbf{n}}=\sum_{k=1}^d |\xi_k|^{n_k}
\end{equation*}
where $\xi=\xi_1 w_1+\xi_2 w_2+\cdots+\xi_d w_d\in W$. We observe immediately $E_\mathbf{w}^{\mathbf{n}}\in\Exp(|\cdot|_{\mathbf{w}}^{\mathbf{n}})$ because $t^{E_\mathbf{w}^{\mathbf{n}}}w_k=t^{1/n_k}w_k$ for $k=1,2,\dots, d$. An application of Proposition \ref{prop:ComparePoly} (a basic result appearing in our background section, Section \ref{sec:Holder}), which uses the nondegenerateness of $P$, gives a positive constant $C$ for which $|\xi|_{\mathbf{w}}^{\mathbf{n}}\leq C|P(\xi)|$ for all $\xi\in W$. The lemma now follows by simply noting that $|\xi |_{\mathbf{w}}^{\mathbf{n}}\rightarrow\infty$ as $\xi\rightarrow\infty$.
\end{proof}

\begin{lemma}
Let $P$ be a polynomial on $W$ and denote by $\Sym(P)$ the set of $O\in\End(W)$ for which $P(O\xi)=P(\xi)$ for all $\xi\in W$. If $P$ is a nondegenerate-homogeneous polynomial, then $\Sym(P)$, called the symmetry group of $P$, is a compact subgroup of $\Gl(W)$.
\end{lemma}
\begin{proof}
Our supposition that $P$ is a nondegenerate polynomial ensures that, for each $O\in\Sym(P)$, $\Ker (O)$ is empty and hence $O\in\Gl(W)$. Consequently, given $O_1$ and $O_2\in\Sym(P)$, we observe that $P(O_1^{-1}\xi)=P(O_1O_1^{-1}\xi)=P(\xi)$ and $P(O_1O_2\xi)=P(O_2\xi)=P(\xi)$ for all $\xi\in W$; therefore $\Sym(P)$ is a subgroup of $\Gl(W)$. 

To see that $\Sym(P)$ is compact, in view of the finite-dimensionality of $\Gl(W)$ and the Heine-Borel theorem, it suffices to show that $\Sym(P)$ is closed and bounded. First, for any sequence $\{O_n\}\subseteq\Sym(P)$ for which $O_n\rightarrow O$ as $n\rightarrow\infty$, the continuity of $P$ ensures that $P(O\xi)=\lim_{n\rightarrow \infty}P(O_n\xi)=\lim_{n\rightarrow\infty}P(\xi)=P(\xi)$ for each $\xi\in W$ and therefore $\Sym(P)$ is closed. It remains to show that $\Sym(P)$ is bounded; this is the only piece of the proof that makes use of the fact that $P$ is nondegenerate-homogeneous and not simply homogeneous. Assume that, to reach a contradiction, that there exists an unbounded sequence $\{O_n\}\subseteq\Sym(P)$. Choosing a norm $|\cdot|$ on $W$, let $S$ be the corresponding unit sphere in $W$. Then there exists a sequence $\{\xi_n\}\subseteq W$ for which $|\xi_n|=1$ for all $n\in\mathbb{N}_+$ but $\lim_{n\rightarrow\infty}|O_n\xi_n|=\infty$. In view of Lemma \ref{lem:PtoInfty},
\begin{equation*}
\infty=\lim_{n\rightarrow\infty}|P(O_n\xi_n)|=\lim_{n\rightarrow\infty}|P(\xi_n)|\leq \sup_{\xi\in S}|P(\xi)|,
\end{equation*}
which cannot be true for $P$ is necessarily bounded on $S$ because it is continuous.
\end{proof}

\begin{lemma}\label{lem:Trace}
Let $\Lambda$ be a nondegenerate-homogeneous operator. For any $E_1,E_2\in\Exp(\Lambda)$,
\begin{equation*}
\tr E_1=\tr E_2.
\end{equation*}
\end{lemma}
\begin{proof}
Let $P$ be the symbol of $\Lambda$ and take $E_1,E_2\in\Exp(\Lambda)$. Since $E_1^*,E_2^*\in \Exp(P)$, $t^{E_1^*}t^{-E_2^*}\in\Sym(P)$ for all $t>0$. As $\Sym(P)$ is a compact group in view of the previous lemma, the determinant map $\det:\Gl(\mathbb{V}^*)\rightarrow\mathbb{C}^*$, a Lie group homomorphism, necessarily maps $\Sym(P)$ into the unit circle. Consequently, 
\begin{equation*}
1=|\det(t^{E_1^*}t^{-E_2^*})|=|\det(t^{E_1^*})\det(t^{-E_2^*})=|t^{\tr {E_1^*}}t^{-\tr {E_2^*}}|=t^{\tr {E_1^*}}t^{-\tr {E_2^*}}
\end{equation*}
for all $t>0$. Therefore, $\tr E_1=\tr E_1^*=\tr E_2^*=\tr E_2$ as desired.
\end{proof}

\noindent By the above lemma, to each nondegenerate-homogenerous operator $\Lambda$, we define the \textit{homogeneous order} of $\Lambda$ to be the number
\begin{equation*}
\mu_{\Lambda}=\tr E
\end{equation*}
for any $E\in\Exp(\Lambda)$. By an appeal to Proposition \ref{prop:OperatorRepresentation}, $E_{\mathbf{v}}^{\mathbf{n}}\in\Exp(\Lambda)$ for some $\mathbf{n}\in\mathbb{N}_+$ and so we observe that
\begin{equation}\label{eq:HomogeneousOrderExplicit}
\mu_{\Lambda}=\frac{1}{n_1}+\frac{1}{n_2}+\cdots+\frac{1}{n_d}.
\end{equation}
In particular, $\mu_{\Lambda}$ is a positive rational number. We note that the term ``homogeneous-order'' does not coincide with the usual ``order'' for a partial differential operator. For instance, the Laplacian $-\Delta$ on $\mathbb{R}^d$ is a second order operator; however, because $2^{-1}I\in\Exp(-\Delta)$, its homogeneous order is $\mu_{(-\Delta)}=\tr 2^{-1}I=d/2$.

\subsection{Positive-homogeneous operators and their heat kernels}

\noindent We now restrict our attention to the study of positive-homogeneous operators and their associated heat kernels. To this end, let $\Lambda$ be a positive-homogeneous operator on $\mathbb{V}$ with symbol $P$ and homogeneous order $\mu_{\Lambda}$. The heat kernel for $\Lambda$ arises naturally from the study of the following Cauchy problem for the corresponding heat equation $\partial_t+\Lambda=0$: Given initial data $f:\mathbb{V}\rightarrow\mathbb{C}$ which is, say, bounded and continuous, find $u(t,x)$ satisfying 
\begin{equation}\label{eq:CauchyProblem}
\begin{cases}
\left(\partial_t+\Lambda\right) u=0 & \mbox{in } (0,\infty)\times\mathbb{V}\\
u(0,x)=f(x)& \mbox{for } x\in\mathbb{V}.
\end{cases}
\end{equation}
The initial value problem \eqref{eq:CauchyProblem} is solved by putting
\begin{equation*}
u(t,x)=\int_{\mathbb{V}}K_{\Lambda}^t(x-y)f(y)\,dy
\end{equation*}
where $K_{\Lambda}^{(\cdot)}(\cdot):(0,\infty)\times \mathbb{V}\rightarrow\mathbb{C}$ is defined by
\begin{equation*}
K_{\Lambda}^t(x)=\mathcal{F}^{-1}\left(e^{-tP}\right)(x)=\int_{\mathbb{V}^*}e^{-i\xi(x)}e^{-tP(\xi)}\,d\xi
\end{equation*}
for $t>0$ and $x\in\mathbb{V}$; we call $K_{\Lambda}$ \textit{the heat kernel} associated to $\Lambda$. Equivalently, $K_{\Lambda}$ is the integral (convolution) kernel of the continuous semigroup $\{e^{-t\Lambda}\}_{t>0}$ of bounded operators on $L^2(\mathbb{V})$ with infinitesimal generator $-\Lambda$. That is, for each $f\in L^2(\mathbb{V})$,
\begin{equation}\label{eq:ConvolutionSemigroupDefinition}
\left(e^{-t\Lambda}f\right)(x)=\int_{\mathbb{V}}K_{\Lambda}^t(x-y)f(y)\,dy
\end{equation}
for $t>0$ and $x\in\mathbb{V}$. \noindent Let us make some simple observations about $K_{\Lambda}$. First, by virtue of Lemma \ref{lem:PtoInfty}, it follows that $K_{\Lambda}^t\in \mathcal{S}(\mathbb{V})$ for each $t>0$. Further, for any $E\in\Exp(\Lambda)$,
\begin{eqnarray*}
\lefteqn{K_{\Lambda}^t(x)=\int_{\mathbb{V}^*}e^{-i\xi(x)}e^{-P(t^{E^*}\xi)}\,d\xi}\\
&&\hspace{2cm}=\int_{\mathbb{V}^*}e^{-i(t^{-E^*})\xi(x)}e^{-P(\xi)}\det (t^{-E^*})\,d\xi =\frac{1}{t^{\tr E}}\int_{\mathbb{V}^*}e^{-i\xi(t^{-E}x)}e^{-P(\xi)}\,d\xi=\frac{1}{t^{\mu_{\Lambda}}}K_{\Lambda}^1(t^{-E}x)
\end{eqnarray*}
for $t>0$ and $x\in\mathbb{V}$. This computation immediately yields the so-called on-diagonal estimate for $K_{\Lambda}$,
\begin{equation*}
\|e^{-t\Lambda}\|_{1\to\infty}=\|K_{\Lambda}^t\|_{\infty}=\frac{1}{t^{\mu_{\Lambda}}}\|K_{\Lambda}^1\|_{\infty}\leq \frac{C}{t^{\mu_{\Lambda}}}
\end{equation*}
for $t>0$; this is equivalently a statement of ultracontractivity for the semigroup $e^{-t\Lambda}$. As it turns out, we can say something much stronger.

\begin{proposition}\label{prop:CCEstimates}
Let $\Lambda$ be a positive-homogeneous operator with symbol $P$ and homogeneous order $\mu_{\Lambda}$. Let $R^{\#}:\mathbb{V}\rightarrow\mathbb{R}$ be the Legendre-Fenchel transform of $R=\Re P$ defined by
\begin{equation*}
R^{\#}(x)=\sup_{\xi\in\mathbb{V}^*}\{\xi(x)-R(\xi)\}
\end{equation*}
for $x\in\mathbb{V}$. Also, let $\mathbf{v}$ and $\mathbf{m}\in\mathbb{N}_+^d$ be as guaranteed by Proposition \ref{prop:OperatorRepresentation}. Then, there exit positive constants $C_0$ and $M$ and, for each multi-index $\beta$, a positive constant $C_{\beta}$ such that, for all $k\in\mathbb{N}$,
\begin{equation}\label{eq:CCDerivativeEstimate}
\left|\partial_t^kD_{\mathbf{v}}^{\beta}K_{\Lambda}^t(x-y)\right|\leq \frac{C_{\beta}C_0^k k!}{t^{\mu_{\Lambda}+k+|\beta:2\mathbf{m}|}}\exp\left(-tMR^{\#}\left(\frac{x-y}{t}\right)\right)
\end{equation}
for all $x,y\in\mathbb{V}$ and $t>0$. In particular,
\begin{equation}\label{eq:CCEstimate}
\left|K_{\Lambda}^t(x-y)\right|\leq \frac{C_0}{t^{\mu_{\Lambda}}}\exp\left(-tMR^{\#}\left(\frac{x-y}{t}\right)\right)
\end{equation}
for all $x,y\in\mathbb{V}$ and $t>0$.
\end{proposition}

\begin{remark}
In view of \eqref{eq:HomogeneousOrderExplicit}, the exponent on the prefactor in \eqref{eq:CCDerivativeEstimate} can be equivalently written, for any multi-index $\beta$ and $k\in\mathbb{N}$, as $\mu_{\Lambda}+k+|\beta:2\mathbf{m}|=k+|\mathbf{1}+\beta:2\mathbf{m}|=|\mathbf{1}+2k\mathbf{m}+\beta:2\mathbf{m}|$ where $\mathbf{1}=(1,1,\dots,1)$.
\end{remark}

\begin{remark}
We note that the estimates of Proposition \ref{prop:CCEstimates} are written in terms of the difference $x-y$ and can (trivially) be expressed in terms of a single spatial variable $x$. The estimates are written in this way to emphasize the role that $K$ plays as an integral kernel. We will later replace $\Lambda$ in \eqref{eq:HeatEquation} by a comparable variable-coefficient operator $H$ and, in that setting, the associated heat kernel is not a convolution kernel and so we seek estimates involving two spatial variables $x$ and $y$. To that end, the estimates here form a template for estimates in the variable-coefficient setting.
\end{remark}

\noindent We prove the proposition above in the Section \ref{sec:FundamentalSolution}; the remainder of this section is dedicated to discussing the result and connecting it to the existing theory. Let us first note that the estimate \eqref{eq:CCDerivativeEstimate} is mirrored by an analogous space-time estimate, Theorem 5.3 of \cite{Randles2015a}, for the convolution powers of complex-valued functions on $\mathbb{Z}^d$ satisfying certain conditions (see Section 5 of \cite{Randles2015a}). The relationship between these two results, Theorem 5.3 of \cite{Randles2015a} and Proposition \ref{prop:CCEstimates}, parallels the relationship between Gaussian off-diagonal estimates for random walks and the analogous off-diagonal estimates enjoyed by the classical heat kernel \cite{Hebisch1993}.\\

\noindent Let us first show that the estimates \eqref{eq:CCDerivativeEstimate} and \eqref{eq:CCEstimate} recapture the well-known estimates of the theory of parabolic equations and systems in $\mathbb{R}^d$ -- a theory in which the Laplacian operator $\Delta=\sum_{l=1}^d\partial_{x_l}^2$ and its integer powers play a central role. To place things into the context of this article, let us observe that, for each positive integer $m$, the partial differential operator $(-\Delta)^m$ is a positive-homogeneous operator on $\mathbb{R}^d$ with symbol $P(\xi)=|\xi|^{2m}$; here, we identify $\mathbb{R}^d$ as its own dual equipped with the dot product and Euclidean norm $|\cdot|$. Indeed, one easily observes that $P=|\cdot|^{2m}$ is a positive-definite polynomial and $E=(2m)^{-1}I\in\Exp((-\Delta)^m)$ where $I\in\Gl(\mathbb{R}^d)$ is the the identity. Consequently, the homogeneous order of $(-\Delta)^m$ is $d/2m=(2m)^{-1}\tr(I)$ and the Legendre-Fenchel transform of $R=\Re P=|\cdot|^{2m}$ is easily computed to be $R^{\#}(x)=C_m|x|^{2m/(2m-1)}$ where $C_m=(2m)^{1/(2m-1)}-(2m)^{-2m/(2m-1)}>0$. Hence, \eqref{eq:CCEstimate} is the well-known estimate
\begin{equation*}
\left|K_{(-\Delta)^m}^t(x-y)\right|\leq \frac{C_0}{t^{d/2m}}\exp\left(-M\frac{|x-y|^{2m/(2m-1)}}{t^{1/(2m-1)}}\right)
\end{equation*}
for $x,y\in\mathbb{R}^d$ and $t>0$; this so-called off-diagonal estimate is ubiquitous to the theory of ``higher-order" elliptic and parabolic equations \cite{Friedman1964, Eidelman1969, Robinson1991a, Davies1997}. To write the derivative estimate \eqref{eq:CCDerivativeEstimate} in this context, we first observe that the basis given by Proposition \ref{prop:OperatorRepresentation} can be taken to be the standard Euclidean basis, $\mathbf{e}=\{e_1,e_2,\dots,e_d\}$ and further, $\mathbf{m}=(m,m,\dots,m)$ is the (isotropic) weight given by the proposition. Writing $D^{\beta}=D^{\beta}_{\mathbf{e}}=(i\partial_{x_1})^{\beta_1}(i\partial_{x_2})^{\beta_2}\cdots(i\partial_{x_d})^{\beta_d}$ and $|\beta|=\beta_1+\beta_2+ \cdots + \beta_d$ for each multi-index $\beta$, \eqref{eq:CCDerivativeEstimate} takes the form
\begin{equation*}
\left|\partial_t^kD^{\beta}K_{(-\Delta)^m}^t(x-y)\right|\leq \frac{C_0}{t^{(d+|\beta|)/2m+k}}\exp\left(-M\frac{|x-y|^{2m/(2m-1)}}{t^{1/(2m-1)}}\right)
\end{equation*}
for $x,y\in\mathbb{R}^d$ and $t>0$, c.f., \cite[Property 4, p. 93]{Eidelman1969}.\\

\noindent The appearance of the $1$-dimensional Legendre-Fenchel transform in heat kernel estimates was previously recognized and exploited in \cite{Barbatis1996} and \cite{Blunck2005} in the context of elliptic operators. Due to the isotropic nature of elliptic operators, the $1$-dimensional transform is sufficient to capture the inherent isotropic decay of corresponding heat kernels. Beyond the elliptic theory, the appearance of the full $d$-dimensional Legendre-Fenchel transform is remarkable because it sharply captures the general anisotropic decay of $K_{\Lambda}$. Consider, for instance, the particularly simple positive-homogeneous operator $\Lambda=-\partial_{x_1}^6+\partial_{x_2}^8$ on $\mathbb{R}^2$ with symbol $P(\xi_1,\xi_2)=\xi_1^6+\xi_2^8$. It is easily checked that the operator $E$ with matrix representation $\diag(1/6,1/8)$, in the standard Euclidean basis, is a member of the $\Exp(\Lambda)$ and so the homogeneous order of $\Lambda$ is $\mu_{\Lambda}=\tr(\diag(1/6,1/8))=7/24$. Here we can compute the Legendre-Fenchel transform of $R=\Re P=P$ directly to obtain $R^{\#}(x_1,x_2)=c_1|x_1|^{6/5}+c_2|x_2|^{8/7}$ for $(x_1,x_2)\in\mathbb{R}^2$ where $c_1$ and $c_2$ are positive constants. In this case, Proposition \ref{prop:CCEstimates} gives positive constants $C_0$ and $M$ for which
\begin{equation}\label{eq:SeparableExample}
|K_{\Lambda}^t(x_1-y_1,x_2-y_2)|\leq\frac{C_0}{t^{7/24}}\exp\left(- \left(M_1\frac{|x_1-y_1|^{6/5}}{t^{1/5}}+M_2\frac{|x_2-y_2|^{8/7}}{t^{1/7}}\right)\right)
\end{equation}
for $(x_1,x_2),(y_1,y_2)\in\mathbb{R}^2$ and $t>0$ where $M_1=c_1M$ and $M_2=c_2M$. We note however that $\Lambda$ is ``separable'' and so we can write $K_{\Lambda}^t(x_1,x_2)=K_{(-\Delta)^3}^t(x_1)K_{(-\Delta)^4}^t(x_2)$ where $\Delta$ is the $1$-dimensional Laplacian operator. In view of Theorem 8 of \cite{Barbatis1996} and its subsequent remark, the estimate \eqref{eq:SeparableExample} is seen to be sharp (modulo the values of $M_1,M_2$ and $C$). To further illustrate the proposition for a less simple positive-homogeneous operator, we consider the operator $\Lambda$ appearing in Example \ref{ex:intro3}. In this case, 
\begin{equation*}
R(\xi_1,\xi_2)=P(\xi_1,\xi_2)=\frac{1}{8}(\xi_1+\xi_2)^2+\frac{23}{384}(\xi_1-\xi_2)^4
\end{equation*}
and one can verify directly that the $E\in\End(\mathbb{R}^2)$, with matrix representation  
\begin{equation*}
E_{\mathbf{e}}=
\begin{pmatrix}
3/8 & 1/8\\
1/8 & 3/8
\end{pmatrix}
\end{equation*}
in the standard Euclidean basis, is a member of $\Exp(\Lambda)$. From this, we immediately obtain $\mu_{\Lambda}=\tr(E)=3/4$ and one can directly compute
\begin{equation*}
R^{\#}(x_1,x_2)=c_1|x_1+x_2|^2+c_2|x_1-x_2|^{4/3}
\end{equation*}
for $(x_1,x_2)\in\mathbb{R}^2$ where $c_1$ and $c_2$ are positive constants.
An appeal to Proposition \ref{prop:CCEstimates} gives positive constants $C_0$ and $M$ for which
\begin{equation*}
|K^t_{\Lambda}(x_1-y_1,x_2-y_2)|\leq \frac{C_0}{t^{3/4}}\exp\left(-\left(M_1\frac{|(x_1-y_1)+(x_2-y_2)|^2}{t}+M_2\frac{|(x_1-y_1)-(x_2-y_2)|^{4/3}}{t^{1/3}}\right)\right)
\end{equation*}
for $(x_1,x_2),(y_1,y_2)\in\mathbb{R}^2$ and $t>0$ where $M_1=c_1M$ and $M_2=c_2M$. Furthermore, $\mathbf{m}=(1,2)\in\mathbb{N}_+^2$ and the basis $\mathbf{v}=\{v_1,v_2\}$ of $\mathbb{R}^2$ given in discussion surrounding \eqref{eq:DiagonalizedExample} are precisely those guaranteed by Proposition \ref{prop:OperatorRepresentation}. Appealing to the full strength of Proposition \ref{prop:CCEstimates}, we obtain positive constants $C_0 $, $M$ and, for each multi-index $\beta$, a positive constant $C_\beta$ such that, for each $k\in\mathbb{N}$,
\begin{eqnarray*}
\lefteqn{\left|\partial_t^kD_{\mathbf{v}}^\beta K_{\Lambda}(x_1-y_1,x_2-y_2)\right|}\\
&\leq&\frac{C_\beta C_0^k k!}{t^{3/4+k+|\beta:2\mathbf{m}|}}\exp\left(-\left(M_1\frac{|(x_1-y_1)+(x_2-y_2)|^2}{t}+M_2\frac{|(x_1-y_1)-(x_2-y_2)|^{4/3}}{t^{1/3}}\right)\right)
\end{eqnarray*}
for $(x_1,x_2),(y_1,y_2)\in\mathbb{R}^2$ and $t>0$ where $M_1=c_1M$ and $M_2=c_2M$.\\

\noindent In the context of homogeneous groups, the off-diagonal behavior for the heat kernel of a positive Rockland operator (a positive self-adjoint operator which is homogeneous with respect to the fixed dilation structure) has been studied in \cite{Hebisch1989,Dziubanski1989,Auscher1994}  (see also \cite{Artino1995}). Given a positive Rockland operator $\Lambda$ on homogeneous group $G$, the best known estimate for the heat kernel $K_{\Lambda}$, due to Auscher, ter Elst and Robinson, is of the form 
 \begin{equation}\label{eq:HebischEst}
 |K_{\Lambda}^t(h^{-1}g)|\leq\frac{C_0}{t^{\mu_{\Lambda}}}\exp\left(-M\left(\frac{\|h^{-1}g\|^{2m}}{t}\right)^{1/(2m-1)}\right)
 \end{equation}
 where $\|\cdot\|$ is a homogeneous norm on $G$ (consistent with $\Lambda)$ and $2m$ is the highest order derivative appearing in $\Lambda$. In the context of $\mathbb{R}^d$, given a symmetric and positive-homogeneous operator $\Lambda$ with symbol $P$, the structure $G_D=(\mathbb{R}^d,\{\delta_t^D\})$ for $D=2mE$ where $E\in\Exp(\Lambda)$ is a homogeneous group on which $\Lambda$ becomes a positive Rockland operator. On $G_D$, it is quickly verified that $\|\cdot\|=R(\cdot)^{1/2m}$ is a homogeneous norm (consistent with $\Lambda$) and so the above estimate is given in terms of $R(\cdot)^{1/(2m-1)}$ which is, in general, dominated by the Legendre-Fenchel transform of $R$. To see this, we need not look further than our previous and  simple example in which $\Lambda= -\partial_{x_1}^6+ \partial_{x_2}^8$. Here $2m=8$ and so $R(x_1,x_2)^{1/(2m-1)}=(|x_1|^6+|x_2|^8)^{1/7}$. In view of \eqref{eq:SeparableExample}, the estimate \eqref{eq:HebischEst} gives the correct decay along the $x_2$-coordinate axis; however, the bounds decay at markedly different rates along the $x_1$-coordinate axis. This illustrates that the estimate \eqref{eq:HebischEst} is suboptimal, at least in the context of $\mathbb{R}^d$, and thus leads to the natural question: For positive-homogeneous operators on a general homogeneous group $G$, what is to replace the Legendre-Fenchel transform in heat kernel estimates?\\
 
\noindent Returning to the general picture, let $\Lambda$ be a positive-homogeneous operator on $\mathbb{V}$ with symbol $P$ and homogeneous order $\mu_{\Lambda}$. To highlight some remarkable properties about the estimates \eqref{eq:CCDerivativeEstimate} and \eqref{eq:CCEstimate} in this general setting, the following proposition concerning $R^{\#}$ is useful; for a proof, see Section 8.3 of \cite{Randles2015a}.
\begin{proposition}\label{prop:LegendreTransformProperties}
Let $\Lambda$ be a positive-homogeneous operator with symbol $P$ and let $R^{\#}$ be the Legendre-Fenchel transform of $R=\Re P$. Then, for any $E\in\Exp(\Lambda)$, $I-E\in\Exp(R^{\#})$. Moreover $R^{\#}$ is continuous, positive-definite in the sense that $R^{\#}(x)\geq 0$ and $R^{\#}(x)=0$ only when $x=0$. Further, $R^{\#}$ grows superlinearly in the sense that, for any norm $|\cdot |$ on $\mathbb{V}$,
\begin{equation*}
\lim_{x\to\infty}\frac{|x|}{R^{\#}(x)}=0;
\end{equation*}
in particular, $R^{\#}(x)\rightarrow\infty$ as $x\rightarrow\infty$.
\end{proposition}
\noindent Let us first note that, in view of the proposition, we can easily rewrite \eqref{eq:CCEstimate}, for any $E\in \Exp(\Lambda)$, as
\begin{equation*}
\left| K_{\Lambda}^t(x-y)\right|\leq \frac{C_0}{t^{\mu_{\Lambda}}}\exp\left(-MR^{\#}\left(t^{-E}(x-y)\right)\right)
\end{equation*}
for $x,y\in\mathbb{V}$ and $t>0$; the analogous rewriting is true for \eqref{eq:CCDerivativeEstimate}. The fact that $R^{\#}$ is positive-definite and grows superlinearly ensures that the convolution operator $e^{-t\Lambda}$ defined by \eqref{eq:ConvolutionSemigroupDefinition} for $t>0$ is a bounded operator from $L^p$ to $L^q$ for any $1\leq p,q\leq \infty$. Of course, we already knew this because $K_{\Lambda}^t$ is a Schwartz function; however, when replacing $\Lambda$ with a variable-coefficient operator $H$, as we will do in the sections to follow, the validity of the estimate \eqref{eq:CCEstimate} for the kernel of the semigroup $\{e^{-tH}\}$ initially defined on $L^2$,  guarantees that the semigroup extends to a strongly continuous semigroup $\{e^{-tH_p}\}$ on $L^p(\mathbb{R}^d)$ for all $1\leq p\leq \infty$ and, what's more, the respective infinitesimal generators $-H_p$ have spectra independent of $p$ \cite{Davies1995}. Further, the estimate \eqref{eq:CCEstimate} is key to establishing the boundedness of the Riesz transform, it is connected to the resolution of Kato's square root problem and it provides the appropriate starting point for uniqueness classes of solutions to $\partial_t+H=0$ \cite{Auscher2001,Ouhabaz2009}. With this motivation in mind, following some background in Section \ref{sec:Holder}, we introduce a class of variable-coefficient operators in Section \ref{sec:UniformlySemiElliptic} called $(2\mathbf{m,v})$-positive-semi-elliptic operators, each such operator $H$  comparable to a fixed positive-homogeneous operator. In Section \ref{sec:FundamentalSolution}, under the assumption that $H$ has H\"{o}lder continuous coefficients and this notion of comparability is uniform,  we construct a fundamental solution to the heat equation $\partial_t+H=0$ and show the essential role played by the Legendre-Fenchel transform in this construction. As mentioned previously, in a forthcoming work we will study the semigroup $\{e^{-tH}\}$ where $H$ is a divergence-form operator, which is comparable to a fixed positive-homogeneous operator, whose coefficients are at worst measurable. As the Legendre-Fenchel transform appears here by a complex change of variables followed by a minimization argument, in the measurable coefficient setting it appears quite naturally by an application of the so-called Davies' method, suitably adapted to the positive-homogeneous setting.

\section{Contracting groups, H\"{o}lder continuity and the Legendre-Fenchel transform}\label{sec:Holder}

\noindent In this section, we provide the necessary background on one-parameter contracting groups, anisotropic H\"{o}lder continuity, and the Legendre-Fenchel transform and its interplay with the two previous notions.

\subsection{One-parameter contracting groups}

\noindent In what follows, $W$ is a $d$-dimensional real vector space with a norm $|\cdot|$; the corresponding operator norm on $\Gl(W)$ is denoted by $\|\cdot\|$. Of course, since everything is finite-dimensional, the usual topologies on $W$ and $\Gl(W)$ are insensitive to the specific choice of norms. 
 
\begin{definition}
Let $\{T_t\}_{t>0}\subseteq \Gl(W)$ be a continuous one-parameter group. $\{T_t\}$ is said to be contracting if
\begin{equation*}
\lim_{t\rightarrow 0}\|T_t\|=0.
\end{equation*}
\end{definition}
\noindent We easily observe that, for any diagonalizable $E\in\End(W)$ with strictly positive spectrum, the corresponding one-parameter group $\{t^E\}_{t>0}$ is contracting. Indeed, if there exists a basis $\mathbf{w}=\{w_1,w_2,\dots,w_d\}$ of $W$ and a collection of positive numbers $\lambda_1,\lambda_2,\dots,\lambda_d$ for which $Ew_k=\lambda_kw_k$ for $k=1,2,\dots,d$, then the one parameter group $\{t^E\}_{t>0}$ has $t^{E}w_k=t^{\lambda_k}w_k$ for $k=1,2,\dots,d$ and $t>0$. It then follows immediately that $\{t^E\}$ is contracting. 

\begin{proposition}\label{prop:ComparePoly}
Let $Q$ and $R$ be continuous real-valued functions on $W$. If $R(w)>0$ for all $w\neq 0$ and there exists $E\in\Exp(Q)\cap\Exp(R)$ for which $\{t^E\}$ is contracting, then, for some positive constant $C$, $Q(w)\leq C R(w)$ for all $w\in W$. If additionally $Q(w)>0$ for all $w\neq 0$, then $Q\asymp R$.
\end{proposition}

\begin{proof}
Let $S$ denote the unit sphere in $W$ and observe that
\begin{equation*}
\sup_{w\in S}\frac{Q(w)}{R(w)}=:C<\infty
\end{equation*}
because $Q$ and $R$ are continuous and $R$ is non-zero on $S$. Now, for any non-zero $w\in W$, the fact that $t^E$ is contracting implies that $t^Ew\in S$ for some $t>0$ by virtue of the intermediate value theorem. Therefore, $Q(w)=Q(t^Ew)/t\leq CR(t^E w)/t=CR(w)$. In view of the continuity of $Q$ and $R$, this inequality must hold for all $w\in W$. When additionally $Q(w)>0$ for all non-zero $w$, the conclusion that $Q\asymp R$ is obtained by reversing the roles of $Q$ and $R$ in the preceding argument. 
\end{proof}

\begin{corollary}\label{cor:MovingConstants} Let $\Lambda$ be a positive-homogeneous operator on $\mathbb{V}$ with symbol $P$ and let $R^{\#}$ be the Legendre-Fenchel transform of $R=\Re P$. Then, for any positive constant $M$, $R^{\#}\asymp (MR)^{\#}$.
\end{corollary}
\begin{proof}
By virtue of Proposition \ref{prop:OperatorRepresentation}, let $\mathbf{m}\in\mathbb{N}_+^d$ and $\mathbf{v}$ be a basis for $\mathbb{V}$ and for which $E_{\mathbf{v}}^{2\mathbf{m}}\in \Exp(\Lambda)$. In view of Proposition \ref{prop:LegendreTransformProperties}, $R^{\#}$ and $(MR)^{\#}$ are both continuous, positive-definite and have $I-E_{\mathbf{v}}^{2\mathbf{m}}\in \Exp(R^{\#})\cap \Exp((MR)^{\#})$. In view of \eqref{eq:DefofE}, it is easily verified that $I-E_{\mathbf{v}}^{2\mathbf{m}}=E_\mathbf{v}^\omega$ where
\begin{equation}\label{eq:DefOfOmega}
\omega:=\left(\frac{2m_1}{2m_1-1},\frac{2m_2}{2m_2-1},\dots\frac{2m_d}{2m_d-1}\right)\in \mathbb{R}_+^d
\end{equation}
and so it follows that $\{t^{E_{\mathbf{v}}^{\omega}}\}$ is contracting. The corollary now follows directly from Proposition \ref{prop:ComparePoly}.
\end{proof}

\begin{lemma}\label{lem:Scaling}
Let $P$ be a positive-homogeneous polynomial on $W$ and let $\mathbf{n}=2\mathbf{m}\in\mathbb{N}_+^d$ and $\mathbf{w}$ be a basis for $W$ for which the conclusion of Lemma \ref{lem:PolynomialRepresentation} holds. Let $R=\Re P$ and let $\beta$ and $\gamma$ be multi-indices such that $\beta\leq\gamma$ (in the standard partial ordering of multi-indices); we shall assume the notation of Lemma \ref{lem:PolynomialRepresentation}.
\begin{enumerate}
\item\label{item:Scaling1} For any $n\in\mathbb{N}_+$ such that $|\beta:\mathbf{m}|\leq 2n$, there exist positive constants $M$ and $M'$ for which
\begin{equation*}
|\xi^{\gamma}\nu^{\beta-\gamma}|\leq M(R(\xi)+R(\nu))^n+M'
\end{equation*}
for all $\xi,\nu\in W$.
\item\label{item:Scaling2} If $|\beta:\mathbf{m}|=2$, there exist positive constants $M$ and $M'$ for which
\begin{equation*}
|\xi^{\gamma}\nu^{\beta-\gamma}|\leq M R(\xi)+M'R(\nu)
\end{equation*}
for all $\nu,\xi\in W$.

\item\label{item:Scaling3} If $|\beta:\mathbf{m}|=2$ and $\beta>\gamma$, then for every $\epsilon>0$ there exists a positive constant $M$ for which
\begin{equation*}
|\xi^{\gamma}\nu^{\beta-\gamma}|\leq \epsilon R(\xi)+MR(\nu)
\end{equation*}
for all $\nu,\xi\in W$.
\end{enumerate}
\end{lemma}
\begin{proof}
Assuming the notation of Lemma \ref{lem:PolynomialRepresentation}, let $E=E_{\mathbf{w}}^{2\mathbf{m}}\in\End(W)$ and consider the contracting group $\{t^{E\oplus E}\}=\{t^{E}\oplus t^{E}\}$ on $W\oplus W$. Because $R$ is a positive-definite polynomial, it immediately follows that $W\oplus W\ni (\xi,\nu)\mapsto R(\xi)+R(\nu)$ is positive-definite. Let $|\cdot|$ be a norm on $W\oplus W$ and respectively denote by $B$ and $S$ the corresponding unit ball and unit sphere in this norm.

To see Item \ref{item:Scaling1}, first observe that
\begin{equation*}
\sup_{(\xi,\nu)\in S}\frac{|\xi^{\gamma}\nu^{\beta-\gamma}|}{(R(\xi)+R(\nu))^n}=:M<\infty
\end{equation*}
Now, for any $(\xi,\nu)\in W\oplus W\setminus B$, because $\{t^{E\oplus E}\}$ is contracting, it follows from the intermediate value theorem that, for some $t\geq 1$, $t^{-(E\oplus E)}(\xi,\nu)=(t^{-E}\xi,t^{-E}\nu)\in S$. Correspondingly,
\begin{eqnarray*}
|\xi^{\gamma}\nu^{\beta-\gamma}|&=&t^{|\beta:2\mathbf{m}|}|(t^{-E}\xi)^{\gamma}(t^{-E})^{\beta-\gamma}|\\
&\leq & t^{|\beta:2\mathbf{m}|}M(R(t^{-E}\xi)+R(t^{-E}\nu))^n\\
&\leq &t^{|\beta:\mathbf{m}|/2-n}M(R(\xi)+R(\nu))^n\\
&\leq &M(R(\xi)+R(\nu))^n
\end{eqnarray*}
because $|\beta:\mathbf{m}|/2\leq n$. One obtains the constant $M'$ and hence the desired inequality by simply noting that $|\xi^{\gamma}\nu^{\beta-\gamma}|$ is bounded for all $(\xi,\nu)\in B$.

For Item \ref{item:Scaling2}, we use analogous reasoning to obtain a positive constant $M$ for which $|\xi^{\gamma}\nu^{\beta-\gamma}|\leq M (R(\xi)+R(\nu))$ for all $(\xi,\nu)\in S$. Now, for any non-zero $(\xi,\nu)\in W\oplus W $, the intermediate value theorem gives $t>0$ for which $t^{E\oplus E}(\xi,\nu)=(t^{E}\xi,t^{E}\nu)\in S$ and hence
\begin{equation*}
|\xi^{\gamma}\nu^{\beta-\gamma}|\leq t^{-|\beta:2\mathbf{m}|}M (R(t^{E}\xi)+R(t^E\nu))=M(R(\xi)+R(\nu))
\end{equation*}
where we have used the fact that $|\beta:2\mathbf{m}|=|\beta:\mathbf{m}|/2=1$ and that $E\in \Exp(R)$. As this inequality must also trivially hold at the origin, we can conclude that it holds for all $\xi,\nu\in W$, as desired.

Finally, we prove Item \ref{item:Scaling3}. By virtue of Item \ref{item:Scaling2}, for any $\xi,\nu\in W$ and $t>0$,
\begin{eqnarray*}
\lefteqn{\hspace{-1.5cm}|\xi^{\gamma}\nu^{\beta-\gamma}|=|(t^Et^{-E}\xi)^{\gamma}\nu^{\beta-\gamma}|=t^{|\gamma:2\mathbf{m}|}|(t^{-E}\xi)^{\gamma}\nu^{\beta-\gamma}|}\\
&\leq& t^{|\gamma:2\mathbf{m}|}\left(MR(t^{-E}\xi)+M'R(\nu)\right)=Mt^{|\gamma:2\mathbf{m}|-1}R(\xi)+M't^{|\gamma:2\mathbf{m}|}R(\nu).
\end{eqnarray*}
Noting that $|\gamma:2\mathbf{m}|-1<0$ because $\gamma<\beta$, we can make the coefficient of $R(\xi)$ arbitrarily small by choosing $t$ sufficiently large and thereby obtaining the desired result.
\end{proof}

\subsection{Notions of regularity and H\"{o}lder continuity}

Throughout the remainder of this article, $\mathbf{v}$ will denote a fixed basis for $\mathbb{V}$ and correspondingly we henceforth assume the notational conventions appearing in Proposition \ref{prop:OperatorRepresentation} and $\mathbf{n}=2\mathbf{m}$ is fixed. For $\alpha\in\mathbb{R}_+^d$, consider the homogeneous norm $|\cdot|_{\mathbf{v}}^{\alpha}$ defined by
\begin{equation*}
|x|_{\mathbf{v}}^{\alpha}=\sum_{i=1}^{d}|x_i|^{\alpha_i}
\end{equation*}
for $x\in\mathbb{V}$ where $\phi_{\mathbf{v}}(x)=(x_1,x_2,\dots,x_d)$. As one can easily check, 
\begin{equation*}
|t^{E_{\mathbf{v}}^\alpha} x|_{\mathbf{v}}^{\alpha}=t|x|_{\mathbf{v}}^{\alpha}
\end{equation*}
for all $t>0$ and $x\in\mathbb{V}$ where $E_{\mathbf{v}}^\alpha\in \Gl(\mathbb{V})$ is defined by \eqref{eq:DefofE}.
\begin{definition}\label{def:Consistent}
Let $\mathbf{m}\in\mathbb{N}_+^d$. We say that $\alpha\in\mathbb{R}_+^d$ is consistent with $\mathbf{m}$ if
\begin{equation}\label{eq:Consistent}
E_{\mathbf{v}}^{\alpha}= a(I-E_{\mathbf{v}}^{2\mathbf{m}})
\end{equation}
for some $a>0$.
\end{definition}
\noindent As one can check, $\alpha$ is consistent with $\mathbf{m}$ if and only if $\alpha=a^{-1}\omega$ where $\omega$ is defined by \eqref{eq:DefOfOmega}.

\begin{definition}
Let $\Omega\subseteq\Omega'\subseteq\mathbb{V}$ and let $f:\Omega'\rightarrow \mathbb{C}$. We say that $f$ is $\mathbf{v}$-H\"{o}lder continuous on $\Omega$ if for some $\alpha\in\mathbb{I}_+^d$ and positive constant $M$,
\begin{equation}\label{eq:HolderCondition}
|f(x)-f(y)|\leq M|x-y|_{\mathbf{v}}^{\alpha}
\end{equation}
for all $x,y\in\Omega$. In this case we will say that $\alpha$ is the $\mathbf{v}$-H\"{o}lder exponent of $f$. If $\Omega=\Omega'$ we will simply say that $f$ is $\mathbf{v}$-H\"{o}lder continuous with exponent $\alpha$. 
 \end{definition}
 
\noindent The following proposition essentially states that, for bounded functions, H\"{o}lder continuity is a local property; its proof is straightforward and is omitted.
 
 \begin{proposition}
 Let $\Omega\subseteq\mathbb{V}$ be open and non-empty. If $f$ is bounded and $\mathbf{v}$-H\"{o}lder continuous of order $\alpha\in\mathbb{I}_+^d$, then, for any $\beta<\alpha$, $f$ is also $\mathbf{v}$-H\"{o}lder continuous of order $\beta$. 
 \end{proposition}
 
\noindent In view of the proposition, we immediately obtain the following corollary.

\begin{corollary}\label{cor:HolderConsistent}
Let $\Omega\subseteq \mathbb{V}$ be open and non-empty and $\mathbf{m}\in\mathbb{N}_+^d$. If $f$ is bounded and $\mathbf{v}$-H\"{o}lder continuous on $\Omega$ of order $\beta\in\mathbb{I}_+^d$, there exists $\alpha\in\mathbb{I}_+^d$ which is consistent with $\mathbf{m}$ for which $f$ is also $\mathbf{v}$-H\"{o}lder continuous of order $\alpha$.
\end{corollary}
\begin{proof}
The statement follows from the proposition by choosing any $\alpha$, consistent with $\mathbf{m}$, such that $\alpha\leq \beta$. 
\end{proof}

\noindent The following definition captures the minimal regularity we will require of fundamental solutions to the heat equation.
\begin{definition}
Let $\mathbf{n}\in\mathbb{N}_+^d$, $\mathbf{v}$ be a basis of $\mathbb{V}$ and let $\mathcal{O}$ be a non-empty open subset of $[0,T]\times\mathbb{V}$. A function $u(t,x)$ is said to be $(\mathbf{n,v})$-regular on $\mathcal{O}$ if on $\mathcal{O}$ it is continuously differentiable in $t$ and has continuous (spatial) partial derivatives $D_{\mathbf{v}}^{\beta}u(t,x)$ for all multi-indices $\beta$ for which $|\beta:\mathbf{n}|\leq 1$.
\end{definition}

\subsection{The Legendre-Fenchel transform and its interplay with $\mathbf{v}$-H\"{o}lder continuity}

Throughout this section, $R$ is the real part of the symbol $P$ of a positive-homogeneous operator $\Lambda$ on $\mathbb{V}$. We assume the notation of Proposition \ref{prop:LegendreTransformProperties} (and hence Proposition \ref{prop:OperatorRepresentation}) and write $E=E_{\mathbf{v}}^{2\mathbf{m}}$. Let us first record two important results which follow essentially from Proposition \ref{prop:LegendreTransformProperties}. 

\begin{corollary}\label{cor:LegendreCompareHomogeneousNorm}
\begin{equation*}
R^{\#}\asymp |\cdot|_{\mathbf{v}}^\omega.
\end{equation*}
where $\omega$ was defined in \eqref{eq:DefOfOmega}.
\end{corollary}
\begin{proof}
In view of Propositions \ref{prop:OperatorRepresentation} and \ref{prop:LegendreTransformProperties}, $E_{\mathbf{v}}^\omega=I-E_{\mathbf{v}}^{2\mathbf{m}}\in \Exp(R^{\#})\cap\Exp(|\cdot|_{\mathbf{v}}^{\omega})$. After recalling that $\{t^{E_{\mathbf{v}}^{\omega}}\}$ is contracting, Proposition \ref{prop:ComparePoly} yields the desired result immediately.
\end{proof}
\noindent By virtue of Proposition \ref{prop:LegendreTransformProperties}, standard arguments immediately yield the following corollary. 
\begin{corollary}\label{cor:ExponentialofLegendreTransformIntegrable}
For any $\epsilon>0$ and polynomial $Q:\mathbb{V}\rightarrow\mathbb{C}$, i.e., $Q$ is a polynomial in any coordinate system, then 
\begin{equation*}
Q(\cdot)e^{-\epsilon R^{\#}(\cdot)}\in L^{\infty}(\mathbb{V})\cap L^1(\mathbb{V}).
\end{equation*}
\end{corollary}

\begin{lemma}\label{lem:PreLegendreSubscale}
Let $\gamma=(2m_{\max}-1)^{-1}$. Then for any $T>0$, there exists $M>0$ such that
\begin{equation*}
R^{\#}(x)\leq Mt^{\gamma}R^{\#}(t^{-E}x)
\end{equation*}
for all $x\in\mathbb{V}$ and $0<t\leq T$. 
\end{lemma}

\begin{proof}
In view of Corollary \ref{cor:LegendreCompareHomogeneousNorm}, it suffices to prove the statement
\begin{equation*}
|t^Ex|_{\mathbf{v}}^{\omega}\leq Mt^{\gamma}|x|_{\mathbf{v}}^{\omega}
\end{equation*}
for all $x\in\mathbb{V}$ and $0<t\leq T$ where $M>0$ and $\omega$ is given by \eqref{eq:DefOfOmega}. But for any $0<t\leq T$ and $x\in\mathbb{V}$,
\begin{equation*}
|t^{E}x|_{\mathbf{v}}^{\omega}=\sum_{j=1}^d t^{1/(2m_j-1)}|x_j|^{\omega_j}\leq t^{\gamma}\sum_{j=1}^d T^{(1/(2m_j-1)-\gamma)}|x_j|^{\omega_j}
\end{equation*}
from which the result follows.
\end{proof}

\begin{lemma}\label{lem:LegendreSubscale}
Let $\alpha\in\mathbb{I}_+^d$ be consistent with $\mathbf{m}$. Then there exists positive constants $\sigma$ and $\theta$ such that $0<\sigma<1$ and for any $T>0$ there exists $M>0$ such that
\begin{equation*}
|x|_{\mathbf{v}}^{\alpha}\leq Mt^{\sigma}(R^{\#}(t^{-E}x))^{\theta}
\end{equation*}
for all $x\in\mathbb{V}$ and $0<t\leq T$. 
\end{lemma}

\begin{proof}
By an appeal to Corollary \ref{cor:LegendreCompareHomogeneousNorm} and Lemma \ref{lem:PreLegendreSubscale}, 
\begin{equation*}
|x|_{\mathbf{v}}^{\omega}\leq Mt^{\gamma}R^{\#}(t^{-E}x)
\end{equation*}
for all $x\in\mathbb{V}$ and $0<t\leq T$. Since $\alpha$ is consistent with $\mathbf{m}$, $\alpha=a^{-1}\omega$ where $a$ is that of Definition \ref{def:Consistent}, the desired inequality follows by setting $\sigma=\gamma/a$ and $\theta=1/a$. Because $\alpha\in \mathbb{I}_+^d$, it is necessary that $a \geq 2m_{\min}/(2m_{\min}-1)$ whence $0<\sigma\leq (2m_{\min}-1)/(2m_{\min}(2m_{\max}-1))<1.$
\end{proof}

\noindent The following corollary is an immediate application of Lemma \ref{lem:LegendreSubscale}.

\begin{corollary}\label{cor:HolderLegendreEstimate}
Let $f:\mathbb{V}\rightarrow\mathbb{C}$ be $\mathbf{v}$-H\"{o}lder continuous with exponent $\alpha\in\mathbb{I}_+^d$ and suppose that $\alpha$ is consistent with $\mathbf{m}$. Then there exist positive constants $\sigma$ and $\theta$ such that $0<\sigma<1$ and, for any $T>0$, there exists $M>0$ such that
\begin{equation*}
|f(x)-f(y)|\leq M t^{\sigma}(R^{\#}(t^{-E}))^{\theta}
\end{equation*}
for all $x,y\in\mathbb{V}$ and $0<t\leq T$.
\end{corollary}

\section{On $(2\mathbf{m,v})$-positive-semi-elliptic operators}\label{sec:UniformlySemiElliptic}
In this section, we introduce a class of variable-coefficient operators on $\mathbb{V}$ whose heat equations are studied in the next section. These operators, in view of Proposition \ref{prop:OperatorRepresentation}, generalize the class of positive-homogeneous operators. Fix a basis $\mathbf{v}$ of $\mathbb{V}$,  $\mathbf{m}\in\mathbb{N}_+^d$ and, in the notation of the previous section, consider a differential operator $H$ of the form
\begin{eqnarray*}
H=\sum_{|\beta:\mathbf{m}|\leq 2}a_{\beta}(x)D_{\mathbf{v}}^{\beta}&=&\sum_{|\beta:\mathbf{m}|=2}a_{\beta}(x)D_{\mathbf{v}}^{\beta}+\sum_{|\beta:\mathbf{m}|< 2}a_{\beta}(x)D_{\mathbf{v}}^{\beta}\\
&:=&H_p+H_l
\end{eqnarray*}
where the coefficients $a_{\beta}:\mathbb{V}\rightarrow \mathbb{C}$ are bounded functions. The symbol of $H$, $P:\mathbb{V}\times\mathbb{V}^*\rightarrow \mathbb{C}$, is defined by
\begin{eqnarray*}
P(y,\xi)=\sum_{|\beta:\mathbf{m}|\leq 2}a_{\beta}(y)\xi^{\beta}
&=&\sum_{|\beta:\mathbf{m}|=2}a_{\beta}(y)\xi^{\beta}+\sum_{|\beta:\mathbf{m}|<2}a_{\beta}(y)\xi^{\beta}\\
&:=&P_p(y,\xi)+P_l(y,\xi).
\end{eqnarray*}
for $y\in\mathbb{V}$ and $\xi\in\mathbb{V}^*$. We shall call $H_p$ the principal part of $H$ and correspondingly, $P_p$ is its principal symbol. Let's also define $R:\mathbb{V}^*\rightarrow \mathbb{R}$ by
\begin{equation}\label{eq:HSymbol}
R(\xi)=\Re P_p(0,\xi)
\end{equation}
for $\xi\in\mathbb{V}^*$. At times, we will freeze the coefficients of $H$ and $H_p$ at a point $y\in\mathbb{V}$ and consider the constant-coefficient operators they define,  namely $H(y)$ and $H_p(y)$ (defined in the obvious way). We note that, for each $y\in\mathbb{V}$, $H_p(y)$ is homogeneous with respect to the one-parameter group $\{\delta_t^E\}_{t>0}$ where $E=E_{\mathbf{v}}^{2\mathbf{m}}\in\Gl(\mathbb{V})$ is defined by \eqref{eq:DefofE}. That is, $H_p$ is homogeneous with respect to the same one-parameter group of dilations at each point in space. This also allows us to uniquely define the \textit{homogeneous order of} $H$ by
\begin{equation}\label{eq:HHomogeneousOrder}
\mu_{H}=\tr E=(2m_1)^{-1}+(2m_2)^{-1}+\cdots+(2m_d)^{-1}.
\end{equation}
We remark that this is consistent with our definition of homogeneous-order for constant-coefficient operators and we remind the reader that this notion differs from the usual order a partial differential operator (see the discussion surrounding \eqref{eq:HomogeneousOrderExplicit}). As in the constant-coefficient setting, $H_p(y)$ is not necessarily homogeneous with respect to a unique group of dilations, i.e., it is possible that $\Exp(H_p(y))$ contains members of $\Gl(\mathbb{V})$ distinct from $E$. However, we shall henceforth only work with the endomorphism $E$, defined above, for worrying about this non-uniqueness of dilations does not aid our understanding nor will it sharpen our results. Let us further observe that, for each $y\in\mathbb{V}$, $P_p(y,\cdot)$ and $R$ are homogeneous with respect to $\{t^{E^*}\}_{t>0}$ where $E^{*}\in \Gl(\mathbb{V}^*)$.

\begin{definition}
The operator $H$ is called $(2\mathbf{m,v})$-positive-semi-elliptic if for all $y\in\mathbb{V}$, $\Re P_p(y,\cdot)$ is a positive-definite polynomial. $H$ is called uniformly $(2\mathbf{m,v})$-positive-semi-elliptic if it is $(2\mathbf{m,v})$-positive-semi-elliptic and there exists $\delta>0$ for which
\begin{equation*}
\Re P_p(y,\xi)\geq\delta R(\xi)
\end{equation*}
for all $y\in\mathbb{V}$ and $\xi\in\mathbb{V}^*$. When the context is clear, we will simply say that $H$ is positive-semi-elliptic and uniformly positive-semi-elliptic respectively.
\end{definition}

\noindent In light of the above definition, a semi-elliptic operator $H$ is one that, at every point $y\in\mathbb{V}$, its frozen-coefficient principal part $H_p(y)$, is a constant-coefficient positive-homogeneous operator which is homogeneous with respect to the same one-parameter group of dilations on $\mathbb{V}$.  A uniformly positive-semi-elliptic operator is one that is semi-elliptic and is uniformly comparable to a constant-coefficient positive-homogeneous operator, namely $H_p(0)$.  In this way, positive-homogeneous operators take a central role in this theory.

\begin{remark}
In view of Proposition \ref{prop:OperatorRepresentation}, the definition of $R$ via \eqref{eq:HSymbol} agrees with that we have given for constant-coefficient positive-homogeneous operators. 
\end{remark}
\begin{remark}
For an $(2\mathbf{m,v})$-positive-semi-elliptic operator $H$, uniform semi-ellipticity can be formulated in terms of $\Re P_p(y_0,\cdot)$ for any $y_0\in\mathbb{V}$; such a notion is equivalent in view of Proposition \ref{prop:ComparePoly}.
\end{remark}

\section{The heat equation}\label{sec:FundamentalSolution}

\noindent For a uniformly positive-semi-elliptic operator $H$, we are interested in constructing a fundamental solution to the heat equation,
\begin{equation}\label{eq:HeatEquation}
(\partial_t+H)u=0
\end{equation}
on the cylinder $[0,T]\times \mathbb{V}$; here and throughout $T>0$ is arbitrary but fixed. By definition, a fundamental solution to \eqref{eq:HeatEquation} on $[0,T]\times\mathbb{V}$ is a function $Z:(0,T]\times\mathbb{V}\times\mathbb{V}\rightarrow \mathbb{C}$ satisfying the following two properties:

\begin{enumerate}
 \item For each $y\in\mathbb{V}$, $Z(\cdot,\cdot,y)$ is $(2\mathbf{m,v})$-regular on $(0,T)\times\mathbb{V}$ and satisfies \eqref{eq:HeatEquation}.
 \item For each $f\in C_b(\mathbb{V})$,
\begin{equation*}
\lim_{t\downarrow 0}\int_{\mathbb{V}}Z(t,x,y)f(y)dy=f(x)
\end{equation*}
for all $x\in\mathbb{V}$.
\end{enumerate}

\noindent Given a fundamental solution $Z$ to \eqref{eq:HeatEquation},  one can easily solve the Cauchy problem: Given $f\in C_b(\mathbb{V})$, find $u(t,x)$ satisfying
\begin{equation*}
\begin{cases}
(\partial_t+H)u=0 & \mbox{on}\hspace{.25cm} (0,T)\times\mathbb{V}\\
u(0,x)=f(x) & \mbox{for}\hspace{.25cm} x\in\mathbb{V}.
\end{cases}
\end{equation*}
This is, of course, solved by putting 
\begin{equation*}
u(t,x)=\int_{\mathbb{V}}Z(t,x,y)f(y)\,dy
\end{equation*}
for $x\in\mathbb{V}$ and $0<t\leq T$ and interpreting $u(0,x)$ as that defined by the limit of $u(t,x)$ as $t\downarrow 0$.
\noindent The remainder of this paper is essentially dedicated to establishing the following result:

\begin{theorem}\label{thm:FundamentalSolution}
Let $H$ be uniformly $(2\mathbf{m,v})$-positive-semi-elliptic with bounded $\mathbf{v}$-H\"{o}lder continuous coefficients. Let $R$ and $\mu_H$ be defined by \eqref{eq:HSymbol} and \eqref{eq:HHomogeneousOrder} respectively and denote by $R^{\#}$ the Legendre-Fenchel transform of $R$. Then, for any $T>0$, there exists a fundamental solution $Z:(0,T]\times\mathbb{V}\times\mathbb{V}\rightarrow \mathbb{C}$ to \eqref{eq:HeatEquation} on $[0,T]\times\mathbb{V}$ such that, for some positive constants $C$ and $M$, 
\begin{equation}\label{eq:FundamentalSolution}
|Z(t,x,y)|\leq \frac{C}{t^{\mu_{H}}}\exp\left(-tMR^{\#}\left(\frac{x-y}{t}\right)\right)
\end{equation}
for $x,y\in\mathbb{V}$ and $0<t\leq T$.
\end{theorem}

\noindent We remark that, by definition, the fundamental solution $Z$ given by Theorem \ref{thm:FundamentalSolution} is $(2\mathbf{m,v})$-regular. Thus $Z$ is necessarily continuously differentiable in $t$ and has continuous spatial derivatives of all orders $\beta$ such that $|\beta:\mathbf{m}|\leq 2$.\\

\noindent As we previously mentioned, the result above is implied by the work of S. D. Eidelman for $2\vec{b}$-parabolic systems on $\mathbb{R}^d$ (where $\vec{b}=\mathbf{m}$) \cite{Eidelman1960,Eidelman2004}. Eidelman's systems, of the form \eqref{eq:2b-Parabolic}, are slightly more general than we have considered here, for their coefficients are also allowed to depend on $t$ (but in a uniformly H\"{o}lder continuous way). Admitting this $t$-dependence is a relatively straightforward matter and, for simplicity of presentation, we have not included it (see Remark \ref{rmk:Gp}). In this slightly more general situation, stated in $\mathbb{R}^d$ and in which $\mathbf{v}=\mathbf{e}$ is the standard Euclidean basis, Theorem 2.2 (p.79) \cite{Eidelman2004} guarantees the existence of a fundamental solution $Z(t,x,y)$ to \eqref{eq:2b-Parabolic}, which has the same regularity appearing in Theorem \ref{thm:FundamentalSolution} and satisfies 
\begin{equation}\label{eq:EidelmansEstimate}
|Z(t,x,y)|\leq \frac{C}{t^{1/(2m_1)+1/(2m_2)+\cdots+1/(2m_d)}}\exp\left(-M\sum_{k=1}^d\frac{|x_k-y_k|^{2m_k/(2m_k-1)}}{t^{1/(2m_k-1)}}\right)
\end{equation}
for $x,y\in\mathbb{R}^d$ and $0<t\leq T$ where $C$ and $M$ are positive constants. By an appeal to Corollary \ref{cor:LegendreCompareHomogeneousNorm}, we have $R^{\#}\asymp |\cdot|_{\mathbf{v}}^{\omega}$ and from this we see that the estimates \eqref{eq:FundamentalSolution} and \eqref{eq:EidelmansEstimate} are comparable.\\

\noindent In view of Corollary \ref{cor:HolderConsistent}, the hypothesis of Theorem \ref{thm:FundamentalSolution} concerning the coefficients of $H$ immediately imply the following a priori stronger condition:

\begin{hypothesis}\label{hypoth:HolderCoefficients}
There exists $\alpha\in\mathbb{I}_+^d$ which is consistent with $\mathbf{m}$ and for which the coefficients of $H$ are bounded and $\mathbf{v}$-H\"{o}lder continuous on $\mathbb{V}$ of order $\alpha$.
\end{hypothesis}

\subsection{Levi's Method}

\noindent In this subsection, we construct a fundamental solution to \eqref{eq:HeatEquation} under only the assumption that $H$, a uniformly $(2\mathbf{m,v})$-positive-semi-elliptic operator, satisfies Hypothesis \ref{hypoth:HolderCoefficients}. Henceforth, all statements include Hypothesis \ref{hypoth:HolderCoefficients} without explicit mention. We follow the famous method of E. E. Levi, c.f., \cite{Levi1907} as it was adopted for parabolic systems in \cite{Eidelman1969} and \cite{Friedman1964}. Although well-known, Levi's method is lengthy and tedious and we will break it into three steps. Let's motivate these steps by first discussing the heuristics of the method.\\

\noindent We start by considering the auxiliary equation

\begin{equation}\label{eq:FrozenCoeffHeat}
\big(\partial_t+\sum_{|\beta:\mathbf{m}|=2}a_{\beta}(y)D_{\mathbf{v}}^{\beta}\big)u=(\partial_t+H_p(y))u=0
\end{equation}
where $y\in\mathbb{V}$ is treated as a parameter. This is the so-called frozen-coefficient heat equation. As one easily checks, for each $y\in\mathbb{V}$,
\begin{equation*}
G_p(t,x;y):=\int_{\mathbb{V}^*}e^{-i\xi(x)}e^{-tP_p(y,\xi)}d\xi\hspace{.5cm}(x\in\mathbb{V},t>0)
\end{equation*}
solves \eqref{eq:FrozenCoeffHeat}. By the uniform semi-ellipticity of $H$, it is clear that $G_p(t,\cdot;y)\in \mathcal{S}(\mathbb{V})$ for $t>0$ and $y\in\mathbb{V}$. As we shall see, more is true: $G_p$ is an approximate identity in the sense that
\begin{equation*}
\lim_{t\downarrow 0}\int_{\mathbb{V}}G_p(t,x-y;y)f(y)\,dy=f(x)
\end{equation*}
for all $f\in C_b(\mathbb{V})$. Thus, it is reasonable to seek a fundamental solution to \eqref{eq:HeatEquation} of the form
\begin{eqnarray}\label{eq:FundSolution}\nonumber
Z(t,x,y)&=&G_p(t,x-y;y)+\int_0^t\int_{\mathbb{V}}G_p(t-s,x-z;z)\phi(s,z,y)dzds\\
&=&G_p(t,x-y;y)+W(t,x,y)
\end{eqnarray}
where $\phi$ is to be chosen to ensure that the correction term $W$ is $(2\mathbf{m,v})$-regular, accounts for the fact that  $G_p$ solves \eqref{eq:FrozenCoeffHeat} but not \eqref{eq:HeatEquation}, and is ``small enough'' as $t\rightarrow 0$ so that the approximate identity aspect of $Z$ is inherited directly from $G_p$.\\

\noindent Assuming for the moment that $W$ is sufficiently regular, let's apply the heat operator to \eqref{eq:FundSolution} with the goal of finding an appropriate $\phi$ to ensure that $Z$ is a solution to \eqref{eq:HeatEquation}. Putting
\begin{equation*}
K(t,x,y)=-(\partial_t+H)G_p(t,x-y;y),
\end{equation*}
we have formally,
\begin{eqnarray}\label{formaldifferentiation}\nonumber
(\partial_t+H)Z(t,x,y)&=&-K(t,x,y)+(\partial_t+H)\int_{0}^t\int_{\mathbb{V}}G_p(t-s,x-z;z)\phi(s,z,y)\,dz\,ds\\ \nonumber
&=&-K(t,x,y)+\lim_{s\uparrow t}\int_{\mathbb{V}}G_p(t-s,x-z;z)\phi(s,z,y)\,dz\\\nonumber
&&\hspace{3cm}-\int_0^t\int_{\mathbb{V}}-(\partial_t+H)G_p(t-s,x-z;z)\phi(s,z,y)\,dz\,ds\\
&=&-K(t,x,y)+\phi(t,x,y)-\int_{0}^t\int_{\mathbb{V}}K(t-s,x,z)\phi(s,z,y)\,dz\,ds
\end{eqnarray}
where we have made use of Leibniz' rule and our assertion that $G_p$ is an approximate identity. Thus, for $Z$ to satisfy \eqref{eq:HeatEquation}, $\phi$ must satisfy the integral equation
\begin{eqnarray}\label{eq:IntegralEquation}\nonumber
K(t,x,y)&=&\phi(t,x,y)-\int_0^{t}\int_{\mathbb{V}}K(t-s,x,z)\phi(s,z,y)\,dz\,ds\\
&=&\phi(t,x,y)-L(\phi)(t,x,y).
\end{eqnarray}
Viewing $L$ as a linear integral operator, \eqref{eq:IntegralEquation} is the equation $K=(I-L)\phi$ which has the solution
\begin{equation}\label{eq:PhiHeuristic}
\phi=\sum_{n=0}^{\infty} L^nK
\end{equation}
provided the series converges in an appropriate sense.\\

\noindent Taking the above as purely formal, our construction will proceed as follows: We first establish estimates for $G_p$ and show that $G_p$ is an approximate identity; this is Step 1. In Step 2, we will define $\phi$ by \eqref{eq:PhiHeuristic} and, after deducing some subtle estimates, show that $\phi$'s defining series converges whence \eqref{eq:IntegralEquation} is satisfied. Finally in Step $3$, we will make use of the estimates from Steps 1 and 2 to validate the formal calculation made in \eqref{formaldifferentiation}. Everything will be then pieced together to show that $Z$, defined by \eqref{eq:FundSolution}, is a fundamental solution to \eqref{eq:HeatEquation}. Our entire construction depends on obtaining precise estimates for $G_p$ and for this we will rely heavily on the homogeneity of $P_p$ and the Legendre-Fenchel transform of $R$.\\

\begin{remark}\label{rmk:Gp}
One can allow the coefficients of $H$ to also depend on $t$ in a uniformly continuous way, and Levi's method pushes though by instead taking $G_p$ as the solution to a frozen-coefficient initial value problem \cite{Eidelman1960,Eidelman2004}.\\
\end{remark}

\noindent\textbf{Step 1. Estimates for $G_p$ and its derivatives}

\noindent The lemma below is a basic building block used in our construction of a fundamental solution to \eqref{eq:HeatEquation} via Levi's method and it makes essential use of the uniform semi-ellipticity of $H$. We note however that the precise form of the constants obtained, as they depend on $k$ and $\beta$, are more detailed than needed for the method to work. Also, the partial differential operators $D_{\mathbf{v}}^{\beta}$ of the lemma are understood to act of the $x$ variable of $G_p(t,x;y)$.

\begin{lemma}\label{lem:LegendreEstimate}
There exist positive constants $M$ and $C_0$ and, for each multi-index $\beta$, a positive constant $C_{\beta}$ such that, for any $k\in\mathbb{N}$,
\begin{equation}\label{eq:LegendreEstimate1}
|\partial_t^kD^{\beta}_{\mathbf{v}}G_p(t,x;y)|\leq\frac{C_{\beta}C_0^k k!}{t^{\mu_H+k+|\beta:2\mathbf{m}|}}\exp\left(-tMR^{\#}\left(x/t\right)\right)
\end{equation}
for all $x,y\in\mathbb{V}$ and $t>0$.
\end{lemma}
\noindent Before proving the lemma, let us note that $tR^{\#}(x/t)=R^{\#}(t^{-E}x)$ for all $t>0$ and $x\in\mathbb{V}$ in view of Proposition \ref{prop:LegendreTransformProperties}. Thus the estimate \eqref{eq:LegendreEstimate1} can be written equivalently as
\begin{equation}\label{eq:LegendreEstimate11}
|\partial_t^kD^{\beta}_{\mathbf{v}}G_p(t,x;y)|\leq\frac{C_{\beta}C_0^k k!}{t^{\mu_H+k+|\beta:2\mathbf{m}|}}\exp(-MR^{\#}(t^{-E}x))
\end{equation}
for $x,y\in\mathbb{V}$ and $t>0$. We will henceforth use these forms interchangeably and without explicit mention.
\begin{proof}
Let us first observe that, for each $x,y\in\mathbb{V}$ and $t>0$,
\begin{eqnarray*}
\partial_t^kD_{\mathbf{v}}^{\beta}G_p(t,x;y)&=&\int_{\mathbb{V}^*}(P_p(y,\xi))^k\xi^{\beta}e^{-i\xi(x)}e^{-tP_p(y,\xi)}\,d\xi\\
&=&\int_{\mathbb{V}^*}(P_p(y,t^{-E^*}\xi))^k(t^{-E^*}\xi)^{\beta}e^{-i\xi(t^{-E}x)}e^{-P_p(y,\xi)}t^{-\tr E}\,d\xi\\
&=&t^{-\mu_H-k-|\beta:2\mathbf{m}|}\int_{\mathbb{V}^*}(P_p(y,\xi))^k\xi^{\beta}e^{-i\xi(t^{-E}x)}e^{-P_p(y,\xi)}d\,\xi
\end{eqnarray*}
where we have used the homogeneity of $P_p$ with respect to $\{t^{E^*}\}$ and the fact that $\mu_H=\tr E$. Therefore
\begin{equation}\label{eq:LegendreEstimate2}
t^{\mu_H+k+|\beta:2\mathbf{m}|}(\partial_t^kD_{\mathbf{v}}^{\beta}G_p(t,\cdot\,;y))(t^{E}x)=\int_{\mathbb{V}^*}(P_p(y,\xi))^k\xi^{\beta}e^{-i\xi(x)}e^{-P_p(y,\xi)}d\xi
\end{equation}
for all $x,y\in\mathbb{V}$ and $t>0$. Thus, to establish \eqref{eq:LegendreEstimate1} (equivalently \eqref{eq:LegendreEstimate11}) it suffices to estimate the right hand side of \eqref{eq:LegendreEstimate2} which is independent of $t$.

The proof of the desired estimate requires making a complex change of variables and for this reason we will work with the complexification of $\mathbb{V}^*$, whose members are denoted by $z=\xi-i\nu$ for $\xi,\nu\in\mathbb{V}^*$; this space is isomorphic to $\mathbb{C}^d$. We claim that there are positive constants $C_0, M_1,M_2$ and, for each multi-index $\beta$, a positive constant $C_{\beta}$ such that, for each $k\in\mathbb{N}$,
\begin{equation}\label{eq:LegendreEstimate3}
|(P_p(y,\xi-i\nu))^k(\xi-i\nu)^{\beta}e^{-P_p(y,\xi-i\nu)}|\leq C_{\beta}C_0^k k!e^{-M_1 R(\xi)}e^{M_2R(\nu)} 
\end{equation}
for all $\xi,\nu\in\mathbb{V}^*$ and $y\in\mathbb{V}$. Let us first observe that
\begin{equation*}
P_p(y,\xi-i\nu)=P_p(y,\xi)+\sum_{|\beta:\mathbf{m}|=2}\sum_{\gamma<\beta}a_{\beta,\gamma}\xi^{\gamma}(-i\nu)^{\beta-\gamma}
\end{equation*}
for all $z,\nu\in\mathbb{V}^*$ and $y\in\mathbb{V}$, where $a_{\beta,\gamma}$ are bounded functions of $y$ arising from the coefficients of $H$ and the coefficients of the multinomial expansion. By virtue of the uniform semi-ellipticity of $H$ and the boundedness of the coefficients, we have
\begin{equation*}
-\Re P_p(y,\xi-i\nu)\leq -\delta R(\xi)+C\sum_{|\beta:\mathbf{m}|=2}\sum_{\gamma<\beta}|\xi^{\gamma}\nu^{\beta-\gamma}|
\end{equation*}
for all $\xi,\nu\in\mathbb{V}^*$ and $y\in\mathbb{V}$ where $C$ is a positive constant. By applying Lemma \ref{lem:Scaling} to each term $|\xi^{\gamma}\nu^{\beta-\gamma}|$ in the summation, we can find a positive constant $M$ for which the entire summation is bounded above by $\delta/2 R(\xi)+MR(\nu)$ for all $\xi,\nu\in\mathbb{V}^*$.  By setting $M_1=\delta/6$, we have
\begin{equation}\label{eq:LegendreEstimate4}
-\Re P_p(y,\xi-i\nu)\leq -3M_1 R(\xi)+M R(\nu)
\end{equation}
for all $\xi,\nu\in\mathbb{V}^*$ and $y\in\mathbb{V}$. By analogous reasoning (making use of item 1 of Lemma \ref{lem:Scaling}), there exists a positive constant $C$ for which
\begin{equation*}
|P_p(y,\xi-i\nu)|\leq C (R(\xi)+ R(\nu))
\end{equation*}
for all $\xi,\nu\in\mathbb{V}^*$ and $y\in\mathbb{V}$. Thus, for any $k\in\mathbb{N}$,
\begin{equation}\label{eq:LegendreEstimate5}
|P_p(y,\xi-i\nu)|^k\leq\frac{C^k k!}{M_1^k}\frac{(M_1 (R(\xi)+R(\nu)))^k}{k!}\leq C_0^k k!e^{M_1 (R(\xi)+R(\nu))}
\end{equation}
for all $\xi,\nu\in\mathbb{V}^*$ and $y\in\mathbb{V}$ where $C_0=C/M_1$. Finally, for each multi-index $\beta$, another application of Lemma \ref{lem:Scaling} gives $C'>0$ for which
\begin{equation*}
|(\xi-i\nu)^{\beta}|\leq|\xi^{\beta}|+|\nu^{\beta}|+\sum_{0<\gamma<\beta}c_{\gamma,\beta}|\xi^{\gamma}\nu^{\beta-\gamma}|\leq C'\left((R(\xi)+R(\nu))^n+1\right)
\end{equation*}
for all $\xi,\nu\in\mathbb{V}^*$ where $n\in\mathbb{N}$ has been chosen to satisfy $|\beta:2n\mathbf{m}|<1$. Consequently, there is a positive constant $C_{\beta}$ for which
\begin{equation}\label{eq:LegendreEstimate6}
|(\xi-i\nu)^{\beta}|\leq C_\beta e^{M_1(R(\xi)+R(\nu))}
\end{equation}
for all $\xi,\nu\in\mathbb{V}^*$. Upon combining \eqref{eq:LegendreEstimate4}, \eqref{eq:LegendreEstimate5} and \eqref{eq:LegendreEstimate6}, we obtain the inequality
\begin{equation*}
\left|P_p(y,\xi-i\nu)^k(\xi-i\nu)^{\beta}e^{-P_p(y,\xi-i\nu)}\right|\leq C_{\beta}C_0^k k! e^{-M_1 R(\xi)+(M+2M_1)R(\nu)}
\end{equation*}
which holds for all $\xi,\nu\in\mathbb{V}^*$ and $y\in\mathbb{V}$. Upon paying careful attention to the way in which our constants were chosen, we observe the claim is established by setting $M_2=M+2M_1$.

From the claim above, it follows that, for any $\nu\in\mathbb{V}^*$ and $y\in\mathbb{V}$, the following change of coordinates by means of a $\mathbb{C}^d$ contour integral is justified:
\begin{eqnarray*}
 \int_{\mathbb{V}^*}(P_p(y,\xi))^k\xi^{\beta}e^{-i\xi(x)}e^{-P_p(y,\xi)}\,d\xi &=&\int_{\xi\in\mathbb{V}^*}(P_p(y,\xi-i\nu)^k(\xi-i\nu)^{\beta}e^{-i(\xi-i\nu)(x)}e^{-P_p(y,\xi-i\nu)}\,d\xi\\
 &=&e^{-\nu(x)}\int_{\xi\in\mathbb{V}^*}(P_p(y,\xi-i\nu)^k(\xi-i\nu)^{\beta}e^{-i\xi(x)}e^{-P_p(y,\xi-i\nu)}\,d\xi.
\end{eqnarray*}
Thus, by virtue of the estimate \eqref{eq:LegendreEstimate3},
\begin{eqnarray*}
\left|\int_{\mathbb{V}^*}(P_p(y,\xi))^k\xi^{\beta}e^{-i\xi(x)}e^{-P_p(y,\xi)}\,d\xi\right |&\leq& C_\beta C_0^k k! e^{-\nu(x)}e^{M_2R(\nu)}\int_{\mathbb{V}^*}e^{-M_1 R(\xi)}\,d\xi\\
&\leq & C_{\beta}C_0^k k! e^{-(\nu(x)-M_2 R(\nu))}
\end{eqnarray*}
for all $x,y\in\mathbb{V}$ and $\nu\in\mathbb{V}^*$ where we have absorbed the integral of $\exp(-M_1R(\xi))$ into $C_{\beta}$. Upon minimizing with respect to $\nu\in\mathbb{V}^*$, we have
\begin{equation}\label{eq:LegendreEstimate7}
\left|\int_{\mathbb{V}^*}(P_p(y,\xi))^k\xi^{\beta}e^{-i\xi(x)}e^{-P_p(y,\xi)}d\xi\right |\leq C_{\beta}C_0^k k!e^{-(M_2R)^{\#}(x)}\leq C_{\beta}C_0^k k!e^{-MR^{\#}(x)}
\end{equation}
for all $x$ and $y\in\mathbb{V}$ because
\begin{equation*}
-(M_2R)^{\#}(x)=-\sup_{\nu}\{\nu(x)-M_2 R(\nu)\}=\inf_{\nu}\{-(\nu(x)-M_2R(\nu))\};
\end{equation*}
in this we see the natural appearance of the Legendre-Fenchel transform. The replacement of $(M_2 R)^{\#}(x)$ by $MR^{\#}(x)$ is done using Corollary \ref{cor:MovingConstants} and, as required, the constant $M$ is independent of $k$ and $\beta$. Upon combining \eqref{eq:LegendreEstimate2} and \eqref{eq:LegendreEstimate7}, we obtain the desired estimate \eqref{eq:LegendreEstimate1}.
\end{proof}

\noindent As a simple corollary to the lemma, we obtain Proposition \ref{prop:CCEstimates}.

\begin{proof}[Proof of Proposition \ref{prop:CCEstimates}]
Given a positive-homogeneous operator $\Lambda$, we invoke Proposition \ref{prop:OperatorRepresentation} to obtain $\mathbf{v}$ and $\mathbf{m}$ for which $\Lambda=\sum_{|\beta:\mathbf{m}|=2}a_{\beta}D_\mathbf{v}^{\beta}$. In other words, $\Lambda$ is an $(2\mathbf{m,v})$-positive-semi-elliptic operator which consists only of its principal part. Consequently, the heat kernel $K_{\Lambda}$ satisfies $K_{\Lambda}^t(x)=G_p(t,x;0)$ for all $x\in\mathbb{V}$ and $t>0$ and so we immediately obtain the estimate \eqref{eq:CCDerivativeEstimate} from the lemma. 
\end{proof}

\noindent Making use of Hypothesis \ref{hypoth:HolderCoefficients}, a similar argument to that given in the proof of Lemma \ref{lem:LegendreEstimate} yields the following lemma.

\begin{lemma}\label{lem:LegendreHolderEstimate}
There is a positive constant $M$ and, to each multi-index $\beta$, a positive constant $C_{\beta}$ such that
\begin{equation*}
|D^{\beta}_\mathbf{v}[G_p(t,x;y+h)-G_p(t,x;y)]|\leq C_{\beta} t^{-(\mu_H+|\beta:2\mathbf{m}|)}|h|_{\mathbf{v}}^{\alpha}\exp(-tMR^{\#}(x/t))
\end{equation*}
for all $t>0$, $x,y,h\in\mathbb{V}$. Here, in view of Hypothesis \ref{hypoth:HolderCoefficients}, $\alpha$ is the $\mathbf{v}$-H\"{o}lder continuity exponent for the coefficients of $H$.
\end{lemma}

\begin{lemma}\label{lem:ApproximateIdentity}
Suppose that $g\in C_b((t_0,T]\times\mathbb{V})$ where $0\leq t_0<T<\infty$. Then, on any compact set $Q\subseteq (t_0,T]\times\mathbb{V}$,
\begin{equation*}
\int_{\mathbb{V}}G_p(t,x-y;y)g(s-t,y)\,dy\rightarrow g(s,x)
\end{equation*}
uniformly on $Q$ as $t\rightarrow 0$. In particular, for any $f\in C_b(\mathbb{V})$,
\begin{equation*}
\int_{\mathbb{V}}G_p(t,x-y;y)f(y)\,dy\rightarrow f(x)
\end{equation*}
uniformly on all compact subsets of $\mathbb{V}$ as $t\rightarrow 0$. 
\end{lemma}
\begin{proof}
Let $Q$ be a compact subset of $(t_0,T]\times\mathbb{V}$ and write
\begin{eqnarray*}
\lefteqn{\int_{\mathbb{V}} G_p(t,x-y;y)g(s-t,y)\,dy}\\
&=&\int_{\mathbb{V}}G_p(t,x-y;x)g(s-t,y)\,dy
+\int_{\mathbb{V}}[G_p(t,x-y;y)-G_p(t,x-y;x)]g(s-t,y)\,dy\\
&:=&I_t^{(1)}(s,x)+I_t^{(2)}(s,x).
\end{eqnarray*}
Let $\epsilon>0$ and, in view of Corollary \ref{cor:ExponentialofLegendreTransformIntegrable}, let $K$ be a compact subset of $\mathbb{V}$ for which
\begin{equation*}
\int_{\mathbb{V}\setminus K}\exp(-MR^{\#}(z))\,dz<\epsilon
\end{equation*} 
where the constant $M$ is that given in \eqref{eq:LegendreEstimate1} of Lemma \ref{lem:LegendreEstimate}. Using the continuity of $g$, we have for sufficiently small $t>0$,
\begin{equation*}
\sup_{\substack{(s,x)\in Q\\z\in K}}|g(s-t,x-t^{E}z)-g(s,x)|<\epsilon.
\end{equation*}
We note that, for any $t>0$ and $x\in\mathbb{V}$,
\begin{equation*}
\int_{\mathbb{V}}G_p(t,x-y;x)\,dy=e^{-tP_p(x,\xi)}\Big\vert_{\xi=0}=1.
\end{equation*}
Appealing to Lemma \ref{lem:LegendreEstimate} we have, for any $(s,x)\in Q$,
\begin{eqnarray*}
|I^{(1)}_t(s,x)-g(s,x)|&\leq&\Big|\int_{\mathbb{V}}G_p(t,x-y;x)(g(s-t,y)-g(s,x))\,dy\Big|\\
&\leq&\int_{\mathbb{V}}|G_p(1,z;x)(g(s-t,x-t^{E}z)-g(s,x))|\,dz\\
&\leq& 2\|g\|_{\infty}C\int_{\mathbb{V}\setminus K}\exp(-MR^{\#}(z))\,dz\\
&&\hspace{3cm}+C\int_{K}\exp(-MR^{\#}(z))|(g(s-t,x-t^{E}z)-g(s,x))|\,dz\\
&\leq& \epsilon C\left(2\|g\|_{\infty}+\|e^{-MR^{\#}}\|_1\right) ;
\end{eqnarray*}
here we have made the change of variables: $y\mapsto t^{E}(x-y)$ and used the homogeneity of $P_p$ to see that $t^{\mu_H}G_p(t,t^{E}z;x)=G_p(1,z;x)$. Therefore $I^{(1)}_t(s,x)\rightarrow g(s,x)$ uniformly on $Q$ as $t\rightarrow 0$. 

Let us now consider $I^{(2)}$. With the help of Lemmas \ref{lem:LegendreSubscale} and \ref{lem:LegendreHolderEstimate} and by making similar arguments to those above we have
\begin{eqnarray*}
|I_t^{(2)}(s,x)|&\leq& C\|g\|_{\infty}\int_{\mathbb{V}}t^{-\mu_H}|x-y|_{\mathbf{v}}^{\alpha}\exp(-MR^{\#}(t^{-E}(x-y))\,dy\\
&\leq &\|g\|_{\infty} C t^{\sigma}\int_{\mathbb{V}} t^{-\tr E}(R^{\#}(t^{-E}(x-y)))^{\theta}\exp(-MR^{\#}(t^{-E}(x-y)))\,dy\\
&\leq&\|g\|_{\infty}Ct^{\sigma}\int_{\mathbb{V}}(R^{\#}(x))^{\theta}\exp(-MR^{\#}(z))\,dz\leq \|g\|_{\infty}C't^{\sigma}
\end{eqnarray*}
for all $s\in(t_0, T]$, $0<t<s-t_0$ and $x\in\mathbb{V}$; here $0<\sigma<1$. Consequently, $I_t^{(2)}(s,x)\rightarrow 0$ uniformly on $Q$ as $t\rightarrow 0$ and the lemma is proved.
\end{proof}

\noindent Combining the results of Lemmas \ref{lem:LegendreEstimate} and \ref{lem:ApproximateIdentity} yields at once:

\begin{corollary}\label{frozenfundsolcor}
For each $y\in\mathbb{V}$, $G_p(\cdot,\cdot-y;y)$ is a fundamental solution to \eqref{eq:FrozenCoeffHeat}.
\end{corollary}

\noindent\textbf{Step 2. Construction of $\phi$ and the integral equation}\\

\noindent For $t>0$ and $x,y\in\mathbb{V}$, put
\begin{eqnarray*}
K(t,x,y)&=&-(\partial_t+H)G_p(t,x-y;y)\\
&=&\big (H_p(y)-H\big)G_p(t,x-y;y)\\
&=&\int_{\mathbb{V}^*}e^{-i\xi(x-y)}\big(P_p(y,\xi)-P(x,\xi)\big)e^{-tP_p(y,\xi)}\,d\xi
\end{eqnarray*}
and iteratively define
\begin{equation*}
K_{n+1}(t,x,y)=\int_0^t\int_{\mathbb{V}}K_1(t-s,x,z)K_{n}(s,z,y)\,dzds
\end{equation*}
where $K_1=K$. In the sense of \eqref{eq:PhiHeuristic}, note that $K_{n+1}=L^n K$.\\

\noindent We claim that for some $0<\rho<1$ and positive constants $C$ and $M$,
\begin{equation}\label{eq:K1Bound}
|K(t,x,y)|\leq C t^{-(\mu_H+1-\rho)}\exp(-M R^{\#}(t^{-E}(x-y)))
\end{equation} for all $x,y\in\mathbb{V}$ and $0<t\leq T$. Indeed, observe that
\begin{equation*}
|K(t,x,y)|\leq \sum_{|\beta:\mathbf{m}|=2}|a_{\beta}(y)-a_{\beta}(x)||D^{\beta}_\mathbf{v} G_p(t,x-y;y)|\\
+C\sum_{|\beta:\mathbf{m}|<2}|D^{\beta}_\mathbf{v}G_p(t,x-y;y)|
\end{equation*}
for all $x,y\in\mathbb{V}$ and $t>0$ where we have used the fact that the coefficients of $H$ are bounded.  In view of Lemma \ref{lem:LegendreEstimate}, we have
\begin{eqnarray*}
|K(t,x,y)|\leq \sum_{|\beta:\mathbf{m}|=2}|a_{\beta}(y)-a_{\beta}(x)|Ct^{-(\mu_H+1)}\exp(-M R^{\#}(t^{-E}(x-y)))\\
+Ct^{-(\mu_H+\eta)}\exp(-M R^{\#}(t^{-E}(x-y)))
\end{eqnarray*}
for all $x,y\in\mathbb{V}$ and $0<t\leq T$ where 
\begin{equation*}
\eta=\max\{|\beta:2\mathbf{m}|:|\beta:\mathbf{m}|\neq 2\mbox{ and }a_{\beta}\neq 0\}<1.
\end{equation*}
Using Hypothesis \ref{hypoth:HolderCoefficients}, an appeal to Corollary \ref{cor:HolderLegendreEstimate} gives $0<\sigma<1$ and $\theta>0$ for which
\begin{eqnarray*}
|K(t,x,y)|&\leq& C t^{\sigma-(\mu_H+1)}(R^{\#}(t^{-E}(x-y)))^{\theta}\exp(-M R^{\#}(t^{-E}(x-y)))\\
&&\hspace{4cm}+Ct^{-(\mu_H+\eta)}\exp(-M R^{\#}(t^{-E}(x-y)))
\end{eqnarray*}
for all $x,y\in\mathbb{V}$ and $0<t\leq T$.  Our claim is then justified by setting 
\begin{equation}\label{eq:defofrho}
\rho=\max\{\sigma,1-\eta\}
\end{equation}
and adjusting the constants $C$ and $M$ appropriately to absorb the prefactor $(R^{\#}(t^{-E}(x-y)))^{\theta}$ into the exponent. It should be noted that the constant $\rho$ is inherently dependent on $H$. For it is clear that $\eta$ depends on $H$.  The constants $\sigma$ and $\theta$ are specified in Lemma \ref{lem:LegendreSubscale} and are defined in terms of the H\"{o}lder exponent of the coefficients of $H$ and the weight $\mathbf{m}$.\\

\noindent Taking cues from our heuristic discussion, we will soon form a series whose summands are the functions $K_n$ for $n\geq 1.$ In order to talk about the convergence of this series, our next task is to estimate these functions and in doing this we will observe two separate behaviors: a finite number of terms will exhibit singularities in $t$ at the origin; the remainder of the terms will be absent of such singularities and will be estimated with the help of the Gamma function. We first address the terms with the singularities.

\begin{lemma}\label{lem:KFiniteTerms}
Let $0<\rho<1$ be given by \eqref{eq:defofrho} and $M>0$ be any constant for which \eqref{eq:K1Bound} is satisfied. For any positive natural number $n$ such that $\rho(n-1)\leq \mu_H+1$ and $\epsilon>0$ for which $\epsilon n<1$, there is a constant $C_n(\epsilon)\geq 1$ such that
\begin{equation*}
|K_n(t,x,y)|\leq C_n(\epsilon)t^{-(\mu_H+1-n\rho)}\exp(-M(1-\epsilon n)R^{\#}(t^{-E}(x-y)))
\end{equation*}
for all $x,y\in\mathbb{V}$ and $0<t\leq T$. 
\end{lemma}
\begin{proof}
In view of \eqref{eq:K1Bound}, it is clear that the estimate holds when $n=1$. Let us assume the estimate holds for $n\geq 1$ such that $\rho n<1+\mu_H$ and $\epsilon>0$ for which $\epsilon n<\epsilon (n+1)<1$. Then
\begin{eqnarray}\label{eq:KFiniteTerms1}\nonumber
|K_{n+1}(t,x,y)|&\leq &\int_0^t\int_{\mathbb{V}}C_1(\epsilon)(t-s)^{-(\mu_H+1-\rho)}C_n(\epsilon)s^{-(\mu_H+1-n\rho)}\\
&&\hspace{1cm}\times\exp(-MR^{\#}((t-s)^{-E}(x-z)))\exp(-M_{\epsilon,n}R^{\#}(s^{-E}(z-y)))\,dz\,ds
\end{eqnarray}
for $x,y\in\mathbb{V}$ and $0<t\leq T$ where we have set $M_{\epsilon,n}=M(1-\epsilon n)$. Observe that
\begin{eqnarray}\label{eq:KFiniteTerms2}\nonumber
R^{\#}(t^{-E}(x-y))&=&\sup\{\xi(x-y)-tR(\xi)\}\\\nonumber
&=&\sup\{\xi(x-z)-(t-s)R(\xi)+\xi(z-y)-sR(\xi)\}\\
&\leq& R^{\#}((t-s)^{-E}(x-z))+R^{\#}(s^{-E}(z-y))
\end{eqnarray}
for all $x,y,z\in\mathbb{V}$ and $0< s\leq t$. Using the fact that $0<\epsilon n<\epsilon (n+1)<1$, \eqref{eq:KFiniteTerms2} guarantees that
\begin{eqnarray*}
\lefteqn{(1-\epsilon(n+1))R^{\#}(t^{-E}(x-y))+\epsilon\left(R^{\#}((t-s)^{-E}(x-z))+R^{\#}(s^{-E}(z-y))\right)}\\
&\leq&(1-\epsilon(n+1))\left(R^{\#}((t-s)^{-E}(x-z))+R^{\#}(s^{-E}(z-y))\right)\\
& &\hspace{4cm}+\epsilon\left(R^{\#}((t-s)^{-E}(x-z))+R^{\#}(s^{-E}(z-y))\right)\\
&\leq & (1-\epsilon n)R^{\#}((t-s)^{-E}(x-z))+(1-\epsilon n)R^{\#}(s^{-E}(z-y))\\
&\leq &R^{\#}((t-s)^{-E}(x-z))+(1-\epsilon n)R^{\#}(s^{-E}(z-y))
\end{eqnarray*}
or equivalently
\begin{eqnarray}\label{eq:KFiniteTerms3}\nonumber
\lefteqn{-MR^{\#}((t-s)^{-E}(x-z))-M_{\epsilon,n}R^{\#}(s^{-E}(z-y)}\\
&\leq& -M_{\epsilon,n+1}R^{\#}(t^{-E}(x-y))-\epsilon M \left(R^{\#}((t-s)^{-E}(x-z))+R^{\#}(s^{-E}(z-y))\right)
\end{eqnarray}
for all $x,y,z\in\mathbb{V}$ and $0<s\leq t$. Combining \eqref{eq:KFiniteTerms1} and \eqref{eq:KFiniteTerms3} yields
\begin{eqnarray}\label{eq:KFiniteTerms4}\nonumber
\lefteqn{|K_{n+1}(t,x,y)|}\\\nonumber
&\leq& C_1(\epsilon)C_n(\epsilon)\exp(-M_{\epsilon,n+1}R^{\#}(t^{-E}(x-y)))\int_0^t\int_{\mathbb{V}}(t-s)^{-(\mu_H+1-\rho)}s^{-(\mu_H+1-n\rho)}\\\nonumber
&&\hspace{4cm}\times\exp(-\epsilon  M(R^{\#}((t-s)^{-E}(x-z))+R^{\#}(s^{-E}(z-y)))\,dz\,ds\\\nonumber
&\leq&C_1(\epsilon)C_n(\epsilon)\exp(-M_{\epsilon,n+1}R^{\#}(t^{-E}(x-y)))\\\nonumber
&&\hspace{2cm}\times\Big[\int_0^{t/2}\int_{\mathbb{V}}(t-s)^{-(\mu_H+1-\rho)}s^{-(\mu_H+1-n\rho)}\times\exp(-\epsilon  MR^{\#}(s^{-E}(z-y)))\,dz\,ds\\
&&\hspace{2.1cm}+\int_{t/2}^{t}\int_{\mathbb{V}}(t-s)^{-(\mu_H+1-\rho)}s^{-(\mu_H+1-n\rho)}\exp(-\epsilon  MR^{\#}((t-s)^{-E}(x-z)))\,dz\,ds\Big]
\end{eqnarray}
for all $x,y\in\mathbb{V}$ and $0<t\leq T$, where we have used the fact that $R^{\#}$ is non-negative. Let us focus our attention on the first term above. For $0\leq s\leq t/2$, 
\begin{equation*}
(t-s)^{-(\mu_H+1-\rho)}s^{-(\mu_H+1-n\rho)}\leq (t/2)^{-(\mu_H+1-\rho)}s^{-(\mu_H+1-n\rho)}
\end{equation*}
because $\mu_H+1-\rho>0$. Consequently,
\begin{eqnarray}\label{eq:KFiniteTerms5}\nonumber
\lefteqn{\int_{0}^{t/2}\int_{\mathbb{V}}(t-s)^{-(\mu_H+1-\rho)}s^{-(\mu_H+1-n\rho)}\exp(-\epsilon MR^{\#}(s^{-E}(z-y)))\,dz\,ds}\\\nonumber
&\leq &(t/2)^{-(\mu_H+1-\rho)}\int_0^{t/2}s^{-(\mu_H+1-n\rho)}\int_{\mathbb{V}}\exp(-\epsilon  MR^{\#}(s^{-E}(z-y)))\,dz\,ds\\\nonumber
&\leq& (t/2)^{-(\mu_H+1-\rho)}\int_0^{t/2}s^{n\rho-1}\int_{\mathbb{V}}\exp(-\epsilon MR^{\#}(z))\,dz\,ds\\\nonumber
&\leq&t^{-(\mu_H+1-(n+1)\rho)}\frac{2^{(\mu_H+1-(n+1)\rho)}}{n\rho}\int_{\mathbb{V}}\exp(-\epsilon  MR^{\#}(z))\,dz\,ds
\end{eqnarray}
for all $y\in\mathbb{V}$ and $t>0$. We note that the second inequality is justified by making the change of variables $z\mapsto s^{-E}(z-y)$  (thus canceling the term $s^{-\tr E}=s^{-\mu_H}$ in the integral over $s$) and the final inequality is valid because $n\rho-1>\rho-1>-1$. By similar reasoning, we obtain 
\begin{eqnarray}\label{eq:KFiniteTerms6}\nonumber
\lefteqn{\int_{t/2}^t\int_{\mathbb{V}}(t-s)^{-(\mu_H+1-n\rho)}s^{-(\mu_H+1-\rho)}\exp(-\epsilon  MR^{\#}((t-s)^{-E}(x-z)))\,dz\,ds}\\
&&\leq t^{-(\mu_H+1-(n+1)\rho)}\frac{2^{(\mu_H+1-(n+1)\rho)}}{\rho}\int_{\mathbb{V}}\exp(-\epsilon  MR^{\#}(z))\,dz\,ds
\end{eqnarray}
for all $x\in\mathbb{V}$ and $t>0$. Upon combining the estimates \eqref{eq:KFiniteTerms4}, \eqref{eq:KFiniteTerms5} and \eqref{eq:KFiniteTerms6}, we have
\begin{equation*}
|K_{n+1}(t,x,y)|\leq C_{n+1}(\epsilon)t^{-(\mu_H+1-(n+1)\rho)}\exp(-M_{\epsilon,n+1}R^{\#}(t^{-E}(x-y))
\end{equation*}
for all $x,y\in\mathbb{V}$ and $0<t\leq T$ where we have put
\begin{equation*}
C_{n+1}(\epsilon)=C_1(\epsilon)C_{n}(\epsilon)\frac{n+1}{n\rho}2^{\mu_H+(1-(n+1)\rho)}\int_{\mathbb{V}}\exp(-\epsilon  M R^{\#}(z))\,dz
\end{equation*}
and made use of Corollary \ref{cor:ExponentialofLegendreTransformIntegrable}.
\end{proof}
\begin{remark}
The estimate \eqref{eq:KFiniteTerms2} is an important one and will be used again. In the context of elliptic operators, i.e., where $R^{\#}(x)=C_m|x|^{2m/(2m-1)}$, the analogous result is captured in Lemma 5.1 of \cite{Eidelman1969}. It is interesting to note that S. D. Eidelman worked somewhat harder to prove it. Perhaps this is because the appearance of the Legendre-Fenchel transform wasn't noticed.
\end{remark}

\noindent It is clear from the previous lemma that for sufficiently large $n$, $K_n$ is bounded by a positive power of $t$. The first such $n$ is $\bar n:=\lceil\rho^{-1}(\tr E+1)\rceil$. In view of the previous lemma,
\begin{equation*}
|K_{\bar n}(t,x,y)|\leq C_{\bar n}(\epsilon)\exp(-M(1-\epsilon\bar n)R^{\#}(t^{-E}(x-y)))
\end{equation*}
for all $x,y\in\mathbb{V}$ and $0<t\leq T$ where we have adjusted $C_{\bar n}(\epsilon)$ to account for this positive power of $t$. Let $\delta<1/2$ and set
\begin{equation*}
\epsilon=\frac{\delta}{\bar n},\hspace{.5cm}M_1=M(1-\delta)\hspace{.25cm}\mbox{ and }\hspace{.25cm}C_0=\max_{1\leq n\leq \bar n}C_n(\epsilon).
\end{equation*}
Upon combining preceding estimate with the estimates\eqref{eq:K1Bound} and \eqref{eq:KFiniteTerms2}, we have
\begin{eqnarray*}
\lefteqn{|K_{\bar n+1}(t,x,y)|}\\
&\leq& C_0^2\int_0^t\int_{\mathbb{V}}(t-s)^{-(\mu_H+(1-\rho))}\\
&&\hspace{2cm}\times\exp(-MR^{\#}((t-s)^{-E}(x-z))\exp(-M(1-\epsilon\bar n)R^{\#}(s^{-E}(z-y)))\,ds\,dz\\
&\leq& C_0^2\exp(-M_1R^{\#}(t^{-E}(x-y)))\int_0^t\int_{\mathbb{V}}(t-s)^{-(\mu_H+(1-\rho))}\exp(-C\delta R^{\#}((t-s)^{-E}(z)))\,dz\,ds\\
&\leq& C_0(C_0 F)\frac{t^{\rho}}{\rho}\exp(-M_1R^{\#}(t^{-E}(x-y)))
\end{eqnarray*}
for all $x,y\in\mathbb{V}$ and $0<t\leq T$ where
\begin{equation*}
F=\int_{\mathbb{V}}\exp(-M\delta R^{\#}(z))\,dz<\infty.
\end{equation*}
Let us take this a little further.
\begin{lemma}\label{lem:KSeries}
For every $k\in \mathbb{N}_+$, 
\begin{equation}\label{eq:KSeriesBound}
|K_{\bar n+k}(t,x,y)|\leq \frac{C_0}{\Gamma(\rho)} \frac{(C_0 F\Gamma(\rho))^k}{k!}t^{\rho k}\exp(-M_1 R^{\#}(t^{-E}(x-y)))
\end{equation}
for all $x,y\in\mathbb{V}$ and $0<t\leq T$. Here $\Gamma(\cdot)$ denotes the Gamma function.
\end{lemma}
\begin{proof} The Euler-Beta function $B(\cdot,\cdot)$ satisfies the well-known identity $B(a,b)=\Gamma(a)\Gamma(b)/\Gamma(a+b)$. Using this identity, one quickly obtains the estimate
\begin{equation*}
\prod_{j=1}^{k-1}B(\rho,1+j\rho)=\frac{\Gamma(\rho)^{k-1}}{\Gamma(1+k\rho)}\leq\frac{\Gamma(\rho)^{k-1}}{k!}.
\end{equation*}
It therefore suffices to prove that
\begin{equation}\label{eq:KSeriesBoundEulerBeta}
|K_{\bar n+k}(t,x,y)|\leq C_0 (C_0 F)^k\prod_{j=0}^{k-1} B(\rho,1+j\rho) t^{k\rho}\exp(-M_1 R^{\#}(t^{-E}(x-y)))
\end{equation}
for all $x,y\in\mathbb{V}$ and $0<t\leq T$.

We first note that $B(\rho,1)=\rho^{-1}$ and so, for $k=1$, \eqref{eq:KSeriesBoundEulerBeta} follows directly from the calculation proceeding the lemma. We shall induct on $k$. By another application of \eqref{eq:K1Bound} and \eqref{eq:KFiniteTerms2}, we have
\begin{eqnarray*}
J_{k+1}(t,x,y)&:=&\Big[C_0^2(C_0 F)^k\prod_{j=0}^{k-1}B(\rho,1+j\rho)\Big]^{-1}|K_{\bar n+k+1}(t,x,y)|\\
&\leq& \int_{0}^t\int_{\mathbb{V}}(t-s)^{-(\mu_H+(1-\rho))}s^{-k\rho}\exp(-MR^{\#}((t-s)^{-E}(x-z)))\\
&&\hspace{6cm}\times\exp(-M_1R^{\#}(s^{-E}(z-y)))\,dz\,ds\\
&\leq& \exp(-M_1R^{\#}(t^{-E}(x-y)))\\
&&\hspace{1cm}\times\int_{0}^t\int_{\mathbb{V}}(t-s)^{-(\mu_H+(1-\rho))}s^{-k\rho}\exp(-M\delta R^{\#}((t-s)^{-E}(x-z)))\,dz\,ds
\end{eqnarray*}
for all $x,y\in\mathbb{V}$ and $0<t\leq T$. Upon making the changes of variables $z\rightarrow (t-s)^{-E}(x-z)$ followed by $s\rightarrow s/t$, we have
\begin{eqnarray*}
J_{k+1}(t,x,y)&\leq &\exp(-M_1R^{\#}(t^{-E}(x-y)))F\int_0^1 (t-st)^{\rho-1}(st)^{k\rho}t\,ds \\
&\leq&\exp(-M_1R^{\#}(t^{-E}(x-y)))F t^{(k+1)\rho} B(\rho,1+k\rho)
\end{eqnarray*}
for all $x,y\in\mathbb{V}$ and $0<t\leq T$. Therefore \eqref{eq:KSeriesBoundEulerBeta} holds for $k+1$ as required.
\end{proof}

\begin{proposition}\label{prop:IntegralIdentity}
Let $\phi:(0,T]\times\mathbb{V}\times\mathbb{V}\rightarrow\mathbb{C}$ be defined by
\begin{equation*}
\phi=\sum_{n=1}^{\infty}K_n.
\end{equation*}
This series converges uniformly for $x,y\in\mathbb{V}$ and $t_0\leq t\leq T$ where $t_0$ is any positive constant. There exists $C\geq 1$ for which
\begin{equation}\label{eq:PhiEstimate}
|\phi(t,x,y)|\leq \frac{C}{t^{\mu_H+(1-\rho)}}\exp(-M_1R^{\#}(t^{-E}(x-y)))
\end{equation}
for all $x,y\in\mathbb{V}$ and $0<t\leq T$ where $M_1$ and $\rho$ are as in the previous lemmas. Moreover, the identity 
\begin{equation}\label{eq:IntegralIdentity}
\phi(t,x,y)=K(t,x,y)+\int_0^t\int_{\mathbb{V}}K(t-s,x,z)\phi(s,z,y)\,dz\,ds
\end{equation}
holds for all $x,y\in\mathbb{V}$ and $0<t\leq T$. 
\end{proposition}

\begin{proof}
Using Lemmas \ref{lem:KFiniteTerms} and \ref{lem:KSeries} we see that
\begin{equation*}
\sum_{k=1}^{\infty}|K_n(t,x,y)|\leq C_0\Big[\sum_{n=1}^{\bar n}t^{-(\mu_H+(1-n\rho))}
+\frac{1}{\Gamma(\rho)}\sum_{k=1}^{\infty} \frac{(C_0 F\Gamma(\rho))^{k}}{k!} t^{k\rho}\Big]\exp(-M_1R^{\#}(t^{-E}(x-y)))
\end{equation*}
for all $x,y\in\mathbb{V}$ and $0<t\leq T$ from which \eqref{eq:PhiEstimate} and our assertion concerning uniform convergence follow.  A similar calculation and an application of Tonelli's theorem justify the following use of Fubini's theorem: For $x,y\in\mathbb{V}$ and $0<t\leq T$,
\begin{eqnarray*}
\int_0^t\int_{\mathbb{V}}K(t-s,x,z)\phi(s,z,y)\,ds\,dz&=&\sum_{n=1}^{\infty}\int_0^t\int_{\mathbb{V}}K(t-s,x,z)K_n(s,z,y)\,dz\,ds\\
&=&\sum_{n=1}^{\infty}K_{n+1}(t,z,y)=\phi(t,x,y)-K(t,x,y)
\end{eqnarray*}
as desired. 
\end{proof}

\noindent The following H\"{o}lder continuity estimate for $\phi$ is obtained by first showing the analogous estimate for $K$ and then deducing the desired result from the integral formula \eqref{eq:IntegralIdentity}. As the proof is similar in character to those of the preceding two lemmas, we omit it. A full proof can be found in \cite[p.80]{Eidelman2004}. We also note here that the result is stronger than is required for our purposes (see its use in the proof of Lemma \ref{lem:WProperties}). All that is really required is that $\phi(\cdot,\cdot,y)$ satisfies the hypotheses (for $f$) in Lemma \ref{lem:IntegralDifferentiation} for each $y\in\mathbb{V}$.

\begin{lemma}\label{holderphilemma}
There exists $\alpha\in\mathbb{I}_+^d$ which is consistent with $\mathbf{m}$, $0<\eta<1$ and $C\geq 1$ such that
\begin{equation*}
|\phi(t,x+h,y)-\phi(t,x,y)|\leq \frac{C}{t^{\mu_H+(1-\eta)}}|h|_{\mathbf{v}}^\alpha\exp(-M_1R^{\#}(t^{-E}(x-y)))
\end{equation*}
for all $x,y,h\in\mathbb{V}$ and $0<t\leq T$.
\end{lemma}

\noindent\textbf{Step 3. Verifying that $Z$ is a fundamental solution to \eqref{eq:HeatEquation}}

\begin{lemma}\label{lem:IntegralDifferentiation}
Let $\alpha\in\mathbb{I}_+^d$ be consistent with $\mathbf{m}$ and, for $t_0> 0$, let $f:[t_0,T]\times\mathbb{V}\rightarrow\mathbb{C}$ be bounded and continuous. Moreover, suppose that $f$ is uniformly $\mathbf{v}$-H\"{o}lder continuous in $x$ on $[t_0,T]\times\mathbb{V}$ of order $\alpha$, by which we mean that there is a constant $C>0$ such that
\begin{equation*}
\sup_{t\in[t_0,T]}|f(t,x)-f(t,y)|\leq C|x-y|_{\mathbf{v}}^{\alpha}
\end{equation*}
for all $x,y\in\mathbb{V}$. Then $u:[t_0,T]\times\mathbb{V}\rightarrow \mathbb{C}$ defined by
\begin{equation*}
u(t,x)=\int_{t_0}^t\int_{\mathbb{V}}G_p(t-s,x-z;z)f(s,z)\,ds\,dz
\end{equation*}
is $(2\mathbf{m,v})$-regular on $(t_0,T)\times\mathbb{V}$. Moreover,
\begin{equation}\label{eq:IntegralDifferentiation1}
\partial_t u(t,x)=f(t,x)+\lim_{h\downarrow 0}\int_{t_0}^{t-h}\int_{\mathbb{V}}\partial_tG_p(t-s,x-z;z)f(s,z)\,dz\,ds
\end{equation}
and for any $\beta$ such that $|\beta:\mathbf{m}|\leq 2$, we have
\begin{equation}\label{eq:IntegralDifferentiation2}
D_{\mathbf{v}}^{\beta}u(t,x)=\lim_{h\downarrow 0}\int_{t_0}^{t-h}\int_{\mathbb{V}}D_{\mathbf{v}}^{\beta}G(t-s,x-z;z)f(s,z)\,dz\,ds
\end{equation}
for $x\in\mathbb{V}$ and $t_0< t< T$.
\end{lemma}

\noindent Before starting the proof, let us observe that, for each multi-index $\beta$, 
\begin{equation*}
|D_{\mathbf{v}}^{\beta}G_p(t-s,x-z;z)f(s,z)|\leq C(t-s)^{-(\mu_H+ |\beta:2\mathbf{m}|)}\exp(-MR^{\#}((t-s)^{-E}(x-z)))|f(s,z)|. 
\end{equation*}
Using the assumption that $f$ is bounded, we observe that 
\begin{eqnarray*}
\lefteqn{\hspace{-2cm}\int_{t_0}^t\int_{\mathbb{V}}|D_{\mathbf{v}}^{\beta}G_p(t-s,x-z;z)f(s,z)|\,dz\,ds}\\
&\leq& C\int_{t_0}^t\int_{\mathbb{V}}(t-s)^{-\mu_H+|\beta:2\mathbf{m}|}\exp(-MR^{\#}((t-s)^{-E}(x-z)))\,dz\,ds\\
&\leq& C\int_{t_0}^t\int_{\mathbb{V}}(t-s)^{-|\beta:2\mathbf{m}|}\exp(-MR^{\#}(z))\,dz\,ds\\
&\leq& C\int_{t_0}^t(t-s)^{-|\beta:2\mathbf{m}|}\,ds
\end{eqnarray*}
for all $t_0\leq t\leq T$ and $x\in\mathbb{V}$. When $|\beta:\mathbf{m}|<2$,
\begin{equation}\label{eq:Singularity}
\int_{t_0}^{t}(t-s)^{-|\beta:2\mathbf{m}|}\,ds
\end{equation}
converges and consequently
\begin{equation*}
D_{\mathbf{v}}^{\beta}u(t,x)=\int_{t_0}^t\int_{\mathbb{V}}D_{\mathbf{v}}^{\beta}G_p(t-s,z-x;z)f(s,z)\,dz\,ds
\end{equation*}
for all $t_0\leq t\leq T$ and $x\in\mathbb{V}$. From this it follows that $D_{\mathbf{v}}^{\beta}u(t,x)$ is continuous on $(t_0,T)\times\mathbb{V}$ and moreover \eqref{eq:IntegralDifferentiation2} holds for such an $\beta$ in view of Lebesgue's dominated convergence theorem. 
When $|\beta:\mathbf{m}|=2$, \eqref{eq:Singularity} does not converge and hence the above argument fails. The main issue in the proof below centers around using $\mathbf{v}$-H\"{o}lder continuity to remove this obstacle. 
\begin{proof}
Our argument proceeds in two steps. The fist step deals with the spatial derivatives of $u$. Therein, we prove the asserted $x$-regularity and show that the formula \eqref{eq:IntegralDifferentiation2} holds. In fact, we only need to consider the case where $|\beta:\mathbf{m}|=2$; the case where $|\beta:\mathbf{m}|<2$ was already treated in the paragraph proceeding the proof. In the second step, we address the time derivative of $u$. As we will see, \eqref{eq:IntegralDifferentiation1} and the asserted $t$-regularity are partial consequences of the results proved in Step 1; this is, in part, due to the fact that the time derivative of $G_p$ can be exchanged for spatial derivatives. The regularity shown in the two steps together will automatically ensure that $u$ is $(2\mathbf{m,v})$-regular on $(t_0,T)\times\mathbb{V}$.\\

\noindent\emph{Step 1. }Let $\beta$ be such that $|\beta:\mathbf{m}|=2$. For $h>0$ write 
\begin{equation*}
u_h(t,x)=\int_{t_0}^{t-h}\int_{\mathbb{V}}G_p(t-s,x-z;z)f(s,z)\,dz\,ds
\end{equation*}
and observe that
\begin{equation*}
D_{\mathbf{v}}^{\beta}u_h(t,x)=\int_{t_0}^{t-h}\int_{\mathbb{V}}D_{\mathbf{v}}^{\beta}G_p(t-s,x-z;z)f(s,z)\,dz\,ds
\end{equation*}
for all $t_0\leq t-h<t\leq T$ and $x\in\mathbb{V}$; it is clear that $D_{\mathbf{v}}^{\beta}u_h(t,x)$ is continuous in $t$ and $x$. The fact that we can differentiate under the integral sign is justified because $t$ has been replaced by $t-h$ and hence the singularity in \eqref{eq:Singularity} is avoided in the upper limit. We will show that $D_{\mathbf{v}}^{\beta}u_h(t,x)$ converges uniformly on all compact subsets of $(t_0,T)\times\mathbb{V}$ as $h\rightarrow 0$. This, of course, guarantees that $D_{\mathbf{v}}^{\beta}u(x,t)$ exists, satisfies \eqref{eq:IntegralDifferentiation2} and is continuous on $(t_0,T)\times\mathbb{V}$. To this end, let us write
\begin{eqnarray*}
D_{\mathbf{v}}^{\beta}u_h(t,x)&=&\int_{t_0}^{t-h}\int_{\mathbb{V}}D_{\mathbf{v}}^{\beta}G_p(t-s,x-z;z)(f(s,z)-f(s,x))\,dz\,ds\\
&&\hspace{3cm}+\int_{t_0}^{t-h}\int_{\mathbb{V}}D_{\mathbf{v}}^{\beta}G_p(t-s,x-z;z)f(s,x)\,dz\,ds\\
&=:&I^{(1)}_h(t,x)+I^{(2)}_h(t,x).
\end{eqnarray*}
Using our hypotheses, Corollary \ref{cor:HolderConsistent} and Lemma \ref{lem:LegendreSubscale}, for some $0<\sigma<1$ and $\theta>0$, there is $M>0$ such that 
\begin{equation*}
|f(s,z)-f(s,x)|\leq C(t-s)^{\sigma}(R^{\#}((t-s)^{-E}(x-z)))^{\theta}
\end{equation*}
for all $x,z\in\mathbb{V}$, $t\in[t_0,T]$ and $s\in[t_0,t]$. In view of the preceding estimate and Lemma \ref{lem:LegendreEstimate}, we have
\begin{eqnarray*}
\lefteqn{\hspace{-1cm}|D_{\mathbf{v}}^{\beta}G_p(t-s,x-z;z)(f(s,z)-f(s,x))|}\\
&\leq& C(t-s)^{-(\mu_H+1)}(t-s)^{\sigma}(R^{\#}((t-s)^{-E}(x-z)))^{\theta}\exp(-MR^{\#}((t-s)^{-E}(x-z)))\\
&\leq& C(t-s)^{-(\mu_H+(1-\sigma))}\exp(-MR^{\#}(t-s)^{-E}(x-z))
\end{eqnarray*}
for all $x,z\in\mathbb{V}$, $t\in[t_0,T]$ and $s\in[t_0,t]$, where $C$ and $M$ are positive constants. We then observe that
\begin{eqnarray*}
\lefteqn{\hspace{-2cm}\int_{t_0}^t\int_{\mathbb{V}}|D_{\mathbf{v}}^{\beta}G_p(t-s,x-z;z)(f(s,z)-f(s,x))|\,dz\,ds}\\
&\leq &C\int_{t_0}^t(t-s)^{-(\mu_H+(1-\sigma))}\int_{\mathbb{V}}\exp(-MR^{\#}((t-s)^{-E}(x-z)))\,dz\,ds\\
&\leq& C\int_{t_0}^t(t-s)^{\sigma-1}\int_{\mathbb{V}}\exp(-MR^{\#}(z))\,dz\,ds\\
&\leq &\frac{C(t-t_0)^{\sigma}}{\sigma}\int_{\mathbb{V}}\exp(-MR^{\#}(z))\,dz\\
&\leq & \frac{C(T-t_0)^{\sigma}}{\sigma}\int_{\mathbb{V}}\exp(-MR^{\#}(z))\,dz<\infty
\end{eqnarray*}
for all $t\in[t_0,T]$ and $x\in \mathbb{V}$, where the validity of the second inequality is seen by making the change of variables $z\mapsto (t-s)^{-E}(x-z)$ and canceling the term $(t-s)^{-\mu_H}=(t-s)^{-\tr E}$.  Consequently,
\begin{equation*}
I^{(1)}(t,x):=\int_{t_0}^t\int_{\mathbb{V}}D_{\mathbf{v}}^{\beta}G_p(t-s,x-z;z)(f(s,z)-f(s,x))\,dz\,ds
\end{equation*}
exists for each $t\in[t_0,T]$ and $x\in\mathbb{V}$. Moreover, for all $t_0\leq t-h<t\leq T$ and $x\in\mathbb{V}$,
\begin{eqnarray*}
|I^{(1)}(t,x)-I_h^{(1)}(t,x)|&\leq&\int_{t-h}^t\int_{\mathbb{V}}|D_{\mathbf{v}}^{\beta}G_p(t-s,x-z;z)(f(s,z)-f(s,x))|\,dz\,ds\\
&\leq& C\int_{t-h}^t\int_{\mathbb{V}}(t-s)^{\sigma-1}\exp(-MR^{\#}(z))\,dz\,ds\leq C h^{\sigma}.
\end{eqnarray*}
From this we see that $\lim_{h\downarrow 0}I_h^{(1)}(t,x)$ converges uniformly on all compact subsets of $(t_0,T)\times\mathbb{V}$.

We claim that for some $0<\rho<1$, there exists $C>0$ such that
\begin{equation}\label{eq:IntegralDifferentiation3}
\Big|\int_{\mathbb{V}}D_{\mathbf{v}}^{\beta}G_p(t-s,x-z;z)\,dz\Big |\leq C(t-s)^{-(1-\rho)}
\end{equation}
for all $x\in\mathbb{V}$ and $s\in[t_0,t]$. Indeed,
\begin{eqnarray*}
\lefteqn{\int_{\mathbb{V}}D_{\mathbf{v}}^{\beta}G_p(t-s,x-z;z)\,dz}\\
&=&
 \int_{\mathbb{V}}D_{\mathbf{v}}^{\beta}[G_p(t-s,x-z;z)-G_p(t-s,x-z;y)]\big\vert_{y=x}\,dz+\big [D_{\mathbf{v}}^{\beta}\int_{\mathbb{V}}G_p(t-s,x-z;y)\,dz\big]\big\vert_{y=x}.
\end{eqnarray*}
The first term above is estimated with the help of Lemma \ref{lem:LegendreHolderEstimate} and by making arguments analogous to those in the previous paragraph; the appearance of $\rho$ follows from an obvious application of Lemma \ref{lem:LegendreSubscale}. This term is bounded by $C(t-s)^{-(1-\rho)}$. The second term is clearly zero and so our claim is justified.

By essentially repeating the arguments made for $I_h^{(1)}$ and making use of \eqref{eq:IntegralDifferentiation3}, we see that
\begin{equation*}
\lim_{h\downarrow 0}I_h^{(2)}(t,x)=I^{(2)}(t,x)=:\int_{t_0}^{t}\int_{\mathbb{V}}D^{\beta}_{\mathbf{v}}G_p(t-s,x-z;z)f(s,x)\,dz\,ds
\end{equation*}
where this limit converges uniformly on all compact subsets of $(t_0,T)\times\mathbb{V}$.\\

\noindent{Step 2. } It follows from Leibnitz' rule that
\begin{eqnarray*}
\partial_tu_h(x,t)&=&\int_{\mathbb{V}}G_p(h,x-z;z)f(t-h,z)\,dz+\int_{t_0}^{t-h}\int_{\mathbb{V}}\partial_tG_p(t-s,x-z;z)f(s,z)\,dz\,ds\\
&=:&J_h^{(1)}(t,x)+J_h^{(2)}(t,x)
\end{eqnarray*}
for all $t_0<t-h<t< T$ and $x\in\mathbb{V}$. Now, in view of Lemma \ref{lem:ApproximateIdentity} and our hypotheses concerning $f$,
\begin{equation*}
\lim_{h\downarrow 0}J_h^{(1)}(t,x)=f(t,x)
\end{equation*}
where this limit converges uniformly on all compact subsets of $(t_0,T)\times\mathbb{V}$.

Using the fact that $\partial_tG_p(t-s,x-z;z)=-H_p(z)G_p(t-s,x-z;z)$, we see that
\begin{eqnarray*}
\lim_{h\downarrow 0}J_h^{(2)}(t,x)&=&\lim_{h\downarrow 0}\int_0^{t-h}\int_{\mathbb{V}}\Big(-\sum_{|\beta:\mathbf{m}|=2}a_{\beta}(z)D_{\mathbf{v}}^{\beta}\Big )G_p(t-s,x-z;z)f(s,z)\,dz\,ds\\
&=&-\sum_{|\beta:\mathbf{m}|=2}\lim_{h\downarrow 0}\int_0^{t-h}\int_{\mathbb{V}}D_{\mathbf{v}}^{\beta}G_p(t-s,x-z;z)(a_{\beta}(z)f(s,z))\,dz\,ds
\end{eqnarray*}
for all $t\in(t_0,T)$ and $x\in\mathbb{V}$. Because the coefficients of $H$ are $\mathbf{v}$-H\"{o}lder continuous and bounded, for each $\beta$, $a_{\beta}(z)f(s,z)$ satisfies the same condition we have required for $f$ and so, in view of Step 1, it follows that $J_h^{(2)}(t,x)$ converges uniformly on all compact subsets of $(t_0,T)\times\mathbb{V}$ as $h\rightarrow 0$. We thus conclude that $\partial_tu(t,x)$ exists, is continuous on $(t_0,T)\times\mathbb{V}$ and satisfies \eqref{eq:IntegralDifferentiation1}. 
\end{proof}

\begin{lemma}\label{lem:WProperties}
Let $W:(0,T]\times\mathbb{V}\times\mathbb{V}\rightarrow \mathbb{C}$ be defined by
\begin{equation*}
W(t,x,y)=\int_0^t\int_{\mathbb{V}}G_p(t-s,x-z;z)\phi(s,z,y)\,dz\,ds,
\end{equation*} 
for $x,y\in\mathbb{V}$ and $0<t\leq T$. Then, for each $y\in\mathbb{V}$,  $W(\cdot,\cdot,y)$ is $(2\mathbf{m,v})$-regular on $(0,T)\times\mathbb{V}$ and satisfies
\begin{equation}\label{eq:WProperties1}
(\partial_t+H)W(t,x,y)=K(t,x,y).
\end{equation}
for all $x,y\in\mathbb{V}$ and $t\in(0,T)$. Moreover, there are positive constants $C$ and $M$ for which
\begin{equation}\label{eq:WProperties2}
|W(t,x,y)|\leq Ct^{-\mu_H+\rho}\exp(-MR^{\#}(t^{-E}(x-y)))
\end{equation}
for all $x,y\in\mathbb{V}$ and $0<t\leq T$ where $\rho$ is that which appears in Lemma \ref{lem:KFiniteTerms}. 
\end{lemma}

\begin{proof}
The estimate \eqref{eq:WProperties2} follows from \eqref{eq:LegendreEstimate1} and \eqref{eq:PhiEstimate} by an analogous computation to that done in the proof of Lemma \ref{lem:KFiniteTerms}. It remains to show that, for each $y\in\mathbb{V}$, $W(\cdot,\cdot,y)$ is $(2\mathbf{m,v})$-regular and satisfies \eqref{eq:WProperties1} on $(0,T)\times\mathbb{V}$. These are both local properties and, as such, it suffices to examine them on $(t_0,T)\times\mathbb{V}$ for an arbitrary but fixed $t_0>0$. Let us write
\begin{eqnarray*}
W(t,x,y)&=&\int_{t_0}^t\int_{\mathbb{V}}G_p(t-s,x-z;z)\phi(s,z,y)\,dz\,ds+\int_0^{t_0}\int_{\mathbb{V}}G_p(t-s,x-z;z)\phi(s,z,y)\,dz\,ds\\
&=:&W_1(t,x,y)+W_2(t,x,y)
\end{eqnarray*}
for $x,y\in\mathbb{V}$ and $t_0<t<T$. In view of Lemmas \ref{holderphilemma} and \ref{lem:IntegralDifferentiation}, for each $y\in\mathbb{V}$, $W_1(\cdot,\cdot,y)$ is $(2\mathbf{m,v})$-regular on $(t_0,T)\times\mathbb{V}$ and
\begin{eqnarray}\label{eq:WProperties3}\nonumber
(\partial_t+H)W_1(t,x,y)&=&\partial_t W_1(t,x,y)+\sum_{|\beta:\mathbf{m}|\leq 2}a_{\beta}(x)D_{\mathbf{v}}^{\beta}W_1(t,x,y)\\\nonumber
&=&\phi(t,x,y)+\lim_{h\downarrow 0}\int_{t_0}^{t-h}\int_{\mathbb{V}}\partial_tG_p(t-s,x-z;z)\phi(s,z,y)\,dz\,dy\\\nonumber
&&\hspace{1cm}+\lim_{h\downarrow 0}\int_{t_0}^{t-h}\int_{\mathbb{V}}\sum_{|\beta:\mathbf{m}|\leq 2}a_{\beta}(x)D_{\mathbf{v}}^{\beta}G_p(t-s,x-z;z)\phi(s,z,y)\,dz\,ds\\\nonumber
&=&\phi(t,x,y)+\lim_{h\downarrow 0}\int_{t_0}^{t-h}\int_{\mathbb{V}}(\partial_t+H)G_p(t-s,x-z;z)\phi(s,z,y)\,dz\,ds\\
&=&\phi(t,x,y)-\lim_{h\downarrow 0}\int_{t_0}^{t-h}\int_{\mathbb{V}}K(t-s,x,z)\phi(s,z,y)\,dz\,ds
\end{eqnarray}
for all $x\in\mathbb{V}$ and $t_0<t<T$; here we have used the fact that 
\begin{equation*}
(\partial_t+H)G_p(t-s,x-z;z)=-K(t-s,x,z). 
\end{equation*}
Treating $W_2$ is easier because $\partial_t G_p(t-s,x-z;z)$ and, for each multi-index $\beta$, $D_{\mathbf{v}}^{\beta}G_p(t-s,x-z;z)$ are, as functions of $s$ and $z$, absolutely integrable on $(0,t_0]\times \mathbb{V}$ for every $t\in (t_0,T]$ and $x\in\mathbb{V}$ by virtue of Lemma \ref{lem:LegendreEstimate}. Consequently, derivatives may be taken under the integral sign and so it follows that, for each $y\in\mathbb{V}$, $W_2(\cdot,\cdot,y)$ is $(2\mathbf{m,v})$-regular on $(t_0,T)\times\mathbb{V}$ and
\begin{equation}\label{eq:WProperties4}
(\partial_t+H)W_2(t,x,y)=-\int_0^{t_0}\int_{\mathbb{V}}K(t-s,x,z)\phi(s,z,y)\,dz\,ds
\end{equation}
for $x\in\mathbb{V}$ and $t_0<t< T$. We can thus conclude that, for each $y\in\mathbb{V}$, $W(\cdot,\cdot,y)$ is $(2\mathbf{m,v})$-regular on $(t_0,T)\times\mathbb{V}$ and, by combining \eqref{eq:WProperties3} and \eqref{eq:WProperties4},
\begin{equation*}
(\partial_t+H)W(t,x,y)=\phi(t,x,y)-\lim_{h\downarrow 0}\int_0^{t-h}\int_{\mathbb{V}}K(t-s,x,z)\phi(s,z,y)\,dz\,ds
\end{equation*}
for $x\in\mathbb{V}$ and $t_0<t< T$. By \eqref{eq:K1Bound}, Proposition \ref{prop:IntegralIdentity} and the Dominated Convergence Theorem,
\begin{eqnarray*}
\lim_{h\downarrow 0}\int_0^{t-h}\int_{\mathbb{V}}K(t-s,x,z)\phi(s,z,y)\,dz\,ds&=&\int_0^{t}\int_{\mathbb{V}}K(t-s,x,z)\phi(s,z,y)\,dz\,ds\\
&=&\phi(t,x,y)-K(t,x,y)
\end{eqnarray*}
and therefore
\begin{equation*}
(\partial_t+H)W(t,x,y)=K(t,x,y)
\end{equation*}
for all $x,y\in\mathbb{V}$ and $t_0<t<T$.
\end{proof}

\noindent The theorem below is our main result. It is a more refined version of Theorem \ref{thm:FundamentalSolution} because it gives an explicit formula for the fundamental solution $Z$; in particular Theorem \ref{thm:FundamentalSolution} is an immediate consequence of the result below.
\begin{theorem}
Let $H$ be a uniformly $(2\mathbf{m,v})$-positive-semi-elliptic operator. If $H$ satisfies Hypothesis \ref{hypoth:HolderCoefficients} then $Z:(0,T]\times\mathbb{V}\times\mathbb{V}\rightarrow \mathbb{C}$, defined by
\begin{equation}\label{eq:Main1}
Z(t,x,y)=G_p(t,x-y;y)+W(t,x,y)
\end{equation}
for $x,y\in\mathbb{V}$ and $0<t\leq T$, is a fundamental solution to \eqref{eq:HeatEquation}. Moreover, there are positive constants $C$ and $M$ for which
\begin{equation}\label{eq:Main2}
|Z(t,x,y)|\leq \frac{C}{t^{\mu_H}}\exp\left(-tM R^{\#}\left(\frac{x-y}{t}\right)\right)
\end{equation}
for all $x,y\in\mathbb{V}$ and $0<t\leq T$. 
\end{theorem}

\begin{proof}
As $0<\rho<1$, \eqref{eq:WProperties2} and Lemma \ref{lem:LegendreEstimate} imply the estimate \eqref{eq:Main2}. In view of Lemma \ref{lem:WProperties} and Corollary \ref{frozenfundsolcor}, for each $y\in\mathbb{V}$, $Z(\cdot,\cdot,y)$ is $(2\mathbf{m,v})$-regular on $(0,T)\times\mathbb{V}$ and
\begin{eqnarray*}
(\partial_t+H)Z(t,x,y)&=&(\partial_t+H)G_p(t,x-y,y)+(\partial_t+H)W(t,x,y)\\
&=&-K(t,x,y)+K(t,x,y)=0
\end{eqnarray*}
for all $x\in\mathbb{V}$ and $0<t< T$. It remains to show that for any $f\in C_b(\mathbb{V})$,
\begin{equation*}
\lim_{t\rightarrow 0}\int_{\mathbb{V}}Z(t,x,y)f(y)\,dy=f(x)
\end{equation*}
for all $x\in\mathbb{V}$. Indeed, let $f\in C_b(\mathbb{V})$ and, in view of \eqref{eq:WProperties2}, observe that
\begin{eqnarray*}
\left|\int_{\mathbb{V}}W(t,x,y)f(y)\right|&\leq& Ct^{\rho}\|f\|_{\infty}\int_{\mathbb{V}}t^{-\mu_H}\exp(-MR^{\#}(t^{-E}(x-y)))dy\\
&\leq& Ct^{\rho}\|f\|_{\infty}\int_{\mathbb{V}}\exp(-MR^{\#}(y))\,dy\leq Ct^{\rho}\|f\|_{\infty}
\end{eqnarray*}
for all $x\in\mathbb{V}$ and $0<t\leq T$. An appeal to Lemma \ref{lem:ApproximateIdentity} gives, for each $x\in\mathbb{V}$,
\begin{eqnarray*}
\lim_{t\rightarrow 0}\int_{\mathbb{V}}Z(t,x,y)f(y)dy&=&\lim_{t\rightarrow 0}\int_{\mathbb{V}}G_p(t,x-y;y)f(y)\,dy+\lim_{t\rightarrow 0}\int_{\mathbb{V}}W(t,x,y)f(y)\,dy\\
&=&f(x)+0=f(x)
\end{eqnarray*}
as required. In fact, the above argument guarantees that this convergence happens uniformly on all compact subsets of $\mathbb{V}$.
\end{proof}

\noindent We remind the reader that implicit in the definition of fundamental solution to \eqref{eq:HeatEquation} is the condition that $Z$ is $(2\mathbf{m,v})$-regular. In fact, one can further deduce estimates for the spatial derivatives of $Z$, $D_{\mathbf{v}}^{\beta}Z$, of the form \eqref{eq:CCDerivativeEstimate} for all $\beta$ such that $|\beta:2\mathbf{m}|\leq 1$ (see \cite[p. 92]{Eidelman2004}). Using the fact that $Z$ satisfies \eqref{eq:HeatEquation} and $H$'s coefficients are bounded, an analogous estimate is obtained for a single $t$ derivative of $Z$.

\noindent Evan Randles\footnote{This material is based upon work supported by the National Science Foundation Graduate Research Fellowship under Grant No. DGE-1144153} : Department of Mathematics, University of California, Los Angeles, Los Angeles CA 90025.
\newline E-mail: randles@math.ucla.edu\\

\noindent Laurent Saloff-Coste\footnote{This material is based upon work supported by the National Science Foundation under Grant No. DMS-1404435}: Department of Mathematics, Cornell University, Ithaca NY 14853.
\newline E-mail: lsc@math.cornell.edu

\end{document}